\documentclass[reqno]{amsart}
\usepackage{graphicx, amsmath, amssymb, amsthm}
\usepackage{verbatim}
\usepackage{mathabx}
\usepackage{tikz-cd}
\usepackage{hyperref}
\usepackage{todonotes}
\usepackage{enumitem}
\usepackage[all]{xy}

\newcommand{\MAHAddress}{University of California Los Angeles, Los Angeles, CA 90095}
\newcommand{\MAHemail}{\tt{mikehill@math.ucla.edu}}

\newcommand{\Z}{{\mathbb  Z}}
\newcommand{\N}{{\mathbb N}}
\newcommand{\R}{{\mathbb R}}
\newcommand{\F}{{\mathbb F}}

\newcommand{\MU}{MU}
\newcommand{\MUR}{\MU_{\R}}
\newcommand{\MUG}{\MU^{(\!(G)\!)}}

\newcommand{\BP}{BP}
\newcommand{\BPR}{\BP_{\R}}
\newcommand{\BPG}{\BP^{(\!(G)\!)}}

\newcommand{\BPGn}[1][m]{\BPG\langle #1\rangle}

\newcommand{\smashover}[1]{\underset{#1}{\wedge}}

\DeclareMathOperator{\Tor}{Tor}

\DeclareMathOperator{\Ind}{Ind}

\newcommand{\tr}{tr}

\newcommand{\m}[1]{{\protect\underline{#1}}}
\newcommand{\mZ}{\m{\Z}}

\mathchardef\mhyphen=45

\newcommand{\EM}{Eilenberg--Mac~Lane}

\numberwithin{equation}{section}

\newtheorem{theorem}{Theorem}[section]
\newtheorem{lemma}[theorem]{Lemma}
\newtheorem{corollary}[theorem]{Corollary}

\newtheorem{proposition}[theorem]{Proposition}
\newtheorem{conjecture}[theorem]{Conjecture}
\newtheorem{quest}[theorem]{Question}
\newtheorem*{theorem*}{Theorem}
\newtheorem*{proposition*}{Proposition}

\newtheorem{thmx}{Theorem}

\theoremstyle{remark}
\newtheorem{remark}[theorem]{Remark}
\newtheorem{example}[theorem]{Example}
\newtheorem{notation}[theorem]{Notation}

\theoremstyle{definition}

\newtheorem{definition}[theorem]{Definition}

\newcommand{\defemph}[1]{\emph{#1}}

\newcommand{\BPCfour}{BP^{(\!(C_4)\!)}}
\newcommand{\MUCfour}{MU^{(\!(C_4)\!)}}
\newcommand{\BPCfourm}[1][m]{BP^{(\!(C_4)\!)}\langle #1 \rangle}
\newcommand{\BPCfourone}{\BPCfourm[1]}
\newcommand{\BPCfourmm}[1][m]{\BPCfourm[#1,#1]}
\newcommand{\tone}{\bar{t}_1}
\newcommand{\ttwo}{\bar{t}_2}
\newcommand{\BPGZnm}[2]{\BPG\langle #1,#2\rangle}
\newcommand{\kCnm}[1][n]{\BPGZnm{#1}{#1}}
\newcommand{\BPGmm}[1]{\BPG\langle #1,#1\rangle}

\DeclareMathOperator{\SliceSS}{SliceSS}
\DeclareMathOperator{\End}{End}

\newcommand{\Neg}{{\mhyphen}}
\newcommand{\bv}{\bar{v}}

\newcommand{\HA}{A}
\newcommand{\Gp}{G'}

\newcommand{\BPCfourtwotwo}{BP^{(\!(C_4)\!)}\langle 2,2 \rangle}

\newcommand{\done}{\bar{\mathfrak{d}}_{\bar{t}_1}}
\newcommand{\dtwo}{\bar{\mathfrak{d}}_{\bar{t}_2}}
\newcommand{\dthree}{\bar{\mathfrak{d}}_{\bar{t}_2}}

\newcommand{\HZ}{H\underline{\mathbb{Z}}}

\newtheorem*{corollary*}{Corollary}
\newtheorem*{prop*}{Proposition}

\theoremstyle{definition}
\newtheorem{summary}[theorem]{Summary}
\newtheorem*{example*}{Example}

\newcommand{\TLAddress}{University of Minnesota, Minneapolis, MN 55455}
\newcommand{\TLemail}{\tt{tlawson@umn.edu}}

\newcommand{\DSAddress}{University of Washington, Seattle, WA 98105}
\newcommand{\DSemail}{\tt{dannyshixl@gmail.com}}

\newcommand{\MZAddress}{Max Planck Institute for Mathematics, 53111 Bonn, Germany}
\newcommand{\MZemail}{\tt{mingcongzeng@gmail.com}}

\newcommand{\ABAddress}{University of Colorado Boulder, Boulder, CO 80309}
\newcommand{\ABemail}{\tt{agnes.beaudry@colorado.edu}}

\newcommand{\Aut}{\mathrm{Aut}}

\title[On the slice spectral sequence for quotients]{On the slice spectral sequence for quotients of norms of Real bordism
}

\author[AB]{Agn\`es Beaudry}
\address{\ABAddress}
\email{\ABemail}

\author[MAH]{Michael A.~Hill}
\address{\MAHAddress}
\email{\MAHemail}

\author[TL]{Tyler Lawson}
\address{\TLAddress}
\email{\TLemail}

\author[XDS]{XiaoLin Danny Shi}
\address{\DSAddress}
\email{\DSemail}

\author[MZ]{Mingcong Zeng}
\address{\MZAddress}
\email{\MZemail}

\begin{document}
\maketitle
\begin{abstract}
In this paper, we investigate equivariant quotients of the Real bordism spectrum's multiplicative norm $MU^{(\!(C_{2^n})\!)}$ by permutation summands. These quotients are of interest because of their close relationship with higher real $K$-theories. We introduce new techniques for computing the equivariant homotopy groups of such quotients.
   
As a new example, we examine the theories $BP^{(\!(C_{2^n})\!)}\langle m,m\rangle$.  These spectra serve as natural equivariant generalizations of connective integral Morava $K$-theories. We provide a complete computation of the $a_{\sigma}$-localized slice spectral sequence of $i^*_{C_{2^{n-1}}}BP^{(\!(C_{2^n})\!)}\langle m,m\rangle$, where $\sigma$ is the real sign representation of $C_{2^{n-1}}$. To achieve this computation, we establish a correspondence between this localized slice spectral sequence and the $H\mathbb{F}_2$-based Adams spectral sequence in the category of $H\mathbb{F}_2 \wedge H\mathbb{F}_2$-modules. Furthermore, we provide a full computation of the $a_{\lambda}$-localized slice spectral sequence of the height-4 theory $BP^{(\!(C_{4})\!)}\langle 2,2\rangle$. The $C_4$-slice spectral sequence can be entirely recovered from this computation.
\end{abstract}

\tableofcontents
\section{Introduction}

\subsection{Motivation} \label{subsec:Motivation}
Let \(E(k,\Gamma)\) be the Lubin--Tate spectrum associated to a formal group law \(\Gamma\) of height \(h\) over a finite field \(k\) of characteristic \(2\).  The Goerss--Hopkins--Miller theorem states that \(E(k,\Gamma)\) is a commutative ring spectrum and that there is an action of \(\Aut(\Gamma)\) on \(E(k,\Gamma)\) by commutative ring maps.  Given a finite subgroup \(G\) of \(\Aut(\Gamma)\), we can view \(E(k,\Gamma)\) as a \(G\)-equivariant commutative ring spectrum via the cofree functor, and define a theory \(EO_{(k,\Gamma)}(G) \) by taking the fixed points under the action of $G$: 
\[
    EO_{(k,\Gamma)}(G) \simeq E(k,\Gamma)^{hG}.
\]
These are the higher real \(K\)-theory spectra, so named because when the height is \(1\) and \(G\) is \(C_2\), these are a form of \(2\)-completed real $K$-theory.

Up to an \'etale extension, these spectra only depend on the height of $\Gamma$ and we will suppress $(k,\Gamma)$ from the notation by letting 
\[E_h=E(k,\Gamma) \quad  \text{and} \quad EO_h(G)=EO_{(k,\Gamma)}(G).\]
The spectra \(EO_h(G)\) play a central role in chromatic homotopy theory.  Reasons of their importance include: 
\begin{enumerate}
    \item 
They detect interesting elements in the homotopy groups of spheres. For example, Hill--Hopkins--Ravenel's work on manifolds of Kervaire invariant one \cite{HHR} and Ravenel's work \cite{Ravenel_arf} can be reinterpreted in terms of the Hurewicz images of $EO_4(C_8)$ and $EO_{p-1}(C_p)$. More recently, Li--Shi--Wang--Xu studied the Hurewicz image of $EO_h(C_2)$ \cite{LSWX}. 
\item They serve as fundamental building blocks for the $K(h)$-local sphere via the theory of finite resolutions. The theory of finite resolutions was developed by Goerss--Henn--Mahowald--Rezk \cite{GHMR} and expanded by Henn \cite{Henn_res}, followed by Bobkova--Goerss \cite{BobkovaGoerss}.
\item They are periodic theories that exhibit nice vanishing line properties \cite{DHS, HopkinsSmithNilpotence, DuanLiShiVanishing}.  This makes them more accessible to computations.  
\end{enumerate}

Historically, there have been few computations of homotopy groups of \(EO_h(G)\) for chromatic heights $h>2$ at $p=2$.
At height $h=1$, these computations are well understood via the relationship with complex and real $K$-theory. At $h=2$, computations are done using the close relationship of the higher real \(K\)-theories with the spectrum $tmf$ of topological modular forms and its analogues with level structures \cite{tmfbook, Bauertmf, BehrensOrmsby, MahowaldRezk, HillMeier, HHRkH}.  However, at chromatic heights $h>2$, such computations have been out of reach for a long time, in part due to the lack of nice geometric models for the higher real $K$-theories such as $ko$ and $tmf$.

More recently, the work of Hill--Hopkins--Ravenel \cite{HHR} has made such computations more achievable. This is the approach we take in this paper.  Specifically, we focus on the case of cyclic 2-groups, as a finite $2$-subgroup of $\Aut(\Gamma)$ at $p=2$ is either a cyclic \(2\)-group of order $2^n$ whenever $h=2^{n-1}m$ for $m \geq 1$, or the quaternions when $h=2m$ for $m$ odd \cite{Hewett}.  Our restriction to the case of cyclic 2-groups allows us to use the equivariant slice filtration and related machinery developed in \cite{HHR}.  

\subsubsection{\texorpdfstring{$G=C_2$}{C2}}
When $G = C_2$, there are two $C_2$-actions: one coming from the central subgroup \(C_2\) in \(\Aut(\Gamma)\) through ``formal inversion'' and the other coming from complex conjugation on the Real bordism spectrum $\MUR$ \cite{Fujii, Landweber}.  Hahn and Shi \cite{HahnShi} produced a Real orientation
\[\MUR\longrightarrow E_h\]
from $\MUR$ to any Lubin--Tate spectrum at the prime \(2\).  This Real orientation allows us to combine the two $C_2$-actions under one perspective, and to construct $E_h$ as a localization of a quotient of $\MUR$.  

After localization at $2$, $\MUR$ splits as a wedge of suspensions of the Real Brown--Peterson spectrum $\BPR$.  Let $\rho_2$ be the regular representation. By work of Araki \cite{Araki} and Landweber \cite{Landweber}, we have
\[\pi_{*\rho_2}^{C_2}\BPR \cong \Z_{(2)}[\bar{v}_1, \bar{v}_2, \bar{v}_3, \ldots]\]
for generators $\bar{v}_i \in \pi_{(2^i-1)\rho_2}^{C_2} \BPR$ whose underlying homotopy classes give generators $v_i \in \pi_{2(2^i-1)}BP$ for $\pi_*BP$. 

In this setup, we can refine two classical families of chromatic spectra to the $C_2$-equivariant world.  They are both constructed as quotients of $\BPR$:

\begin{enumerate}
\item The first family is the \textit{Real truncated Brown--Peterson spectrum}
\[\BPR \langle h \rangle = \BPR/(\bar{v}_j \mid j >h). \]
The underlying non-equivariant spectrum is the classical truncated Brown--Peterson spectrum $BP \langle h \rangle$.  The $K(h)$-localization of $\BPR \langle h \rangle$ gives, up to periodization, a model of Lubin--Tate theory $E_h$ with its canonical $C_2$-action obtained through Goerss--Hopkins--Miller theory.  These equivariant spectra and their $\bar{v}_h$-localizations were first studied by Hu--Kriz \cite{HuKriz} and Kitchloo--Wilson \cite{KitchlooWilson}.

\item The second family is the \textit{Real connective integral Morava \(K\)-theory}
\[\BPR\langle h,h\rangle := \BPR/(\bar{v}_i \mid i\neq 0, h),\] 
whose underlying spectra are the connective integral Morava $K$-theories.  After quotienting by 2 and periodization, we obtain the classical Morava $K$-theories $K(h)$. \end{enumerate}

\subsubsection{Larger cyclic \texorpdfstring{$2$}{2}-groups}
In this paper, we will study the $C_{2^n}$-generalizations of the integral Morava \(K\)-theories in great computational depth.

Let \(G\) be a finite subgroup of \(\Aut(\Gamma)\) containing \(C_2\). Since \(E_h\) is an equivariant commutative ring spectrum, the norm-forgetful adjunction produces a \(G\)-equivariant orientation map
\[
	\MUG:=N_{C_2}^{G}\MUR\longrightarrow E_h.
\]
Since we are working $2$-locally, we can substitute $\MUR$ with $\BPR$ using Quillen's idempotent, thereby obtaining a map
\[	\BPG:=N_{C_2}^{G}\BPR\longrightarrow E_h. \]
This map allows us to regard $\big(\BPG\big)^{G}$ as a global approximation for the theory \(EO_{h}(G)\).  

We will concentrate on the case $G=C_{2^n}$ and $h=2^{n-1}m$ in this paper. 
To compute the homotopy groups $\pi_\star^G \BPG$,
Hill, Hopkins, and Ravenel invented the equivariant slice spectral sequence \cite{HHR}.  The current approach of using the slice spectral sequence to understand $\pi_\star^G \BPG$ is to pass to quotients of $\BPG$ that generalize $\BPR \langle h\rangle$.  To define these quotients, note that as in \cite{HHR}, the $C_2$-equivariant homotopy groups of $\BPG$ in degrees an integer multiple of $\rho_2$ are
\[ \pi_{*\rho_2}^{C_2}\BPG \cong \Z_{(2)}[G\cdot \bar{v}_1^G, G \cdot \bar{v}_2^G,\ldots ], \]
where $\bar{v}_i^G \in \pi_{(2^i-1)\rho_2}^{C_2} \BPG$. Here, $G \cdot \bar{v}_i^G$ denotes a set of $2^{n-1}$ elements 
\[\left\{\bar{v}_i^G, \gamma  \bar{v}_i^G, \ldots, \gamma^{2^{n-1}-1} \bar{v}_i^G\right \},\] 
where $\gamma$ represents a generator of $G$ and the Weyl action of $G$ is made obvious by the notation except for $\gamma^{2^{n-1}} \bar{v}_i^G =-\bar{v}_i^G$.
The method of twisted monoid rings \cite[Section~2]{HHR} then allows one to form quotients of $\BPG$ by collections of \emph{permutation summands} of the form $G \cdot \bar{v}_i^G$.  These quotients are the main objects of study in this paper.

The generalizations of the Real truncated Brown--Peterson spectrum $\BPR\langle h \rangle$ and the Real connective integral Morava $K$-theory $\BPR \langle h, h \rangle$ are the following quotients by permutation summands:
\begin{enumerate}
\item The quotient 
\[\BPG\langle m \rangle := \BPG/(G\cdot \bar{v}_{m+1}^G, G \cdot \bar{v}_{m+2}^G,\ldots )\]
generalizes the spectrum $\BPR\langle h\rangle$.  These spectra were studied in \cite{BHSZOne}, where it was shown that $\BPG\langle m \rangle$ is of height ${\leq 2^{n-1}m}$ (\cite[Theorem~7.5]{BHSZOne}).  Up to periodization and $K(h)$-localization, the $G$-fixed points of $\BPG \langle m \rangle$ gives a model for $EO_h(G)$.  The theories $\BPG \langle m \rangle$ also give a chromatic filtration of $\BPG$ via the tower
\begin{equation}\label{tower:chromFiltrationBPG} 
\cdots \longrightarrow \BPG\langle m \rangle \longrightarrow \BPG\langle m-1 \rangle \longrightarrow \cdots \longrightarrow \BPG\langle 1 \rangle. 
\end{equation}
The slice spectral sequences for these quotients have been computed for $\BPR\langle m\rangle$ ($m \geq 1$), $BP^{(\!(C_4)\!)}\langle 1\rangle$, and $BP^{(\!(C_4)\!)}\langle 2\rangle$ \cite{HuKriz, HHRkH, HSWX}.  

\item The equivariant generalization of the connective integral Morava $K$-theories are the quotients
\[\BPG\langle m,m \rangle :=\BPG\langle m \rangle/(G\cdot \bar{v}_{1}^G, \ldots, G \cdot \bar{v}_{m-1}^G).\] 
\end{enumerate}

Given $\BPG \langle m, m \rangle$, we can apply further quotienting and localization to form the $G$-spectrum
\[K_G(h):= N_{C_2}^G(\bar{v}_m^G)^{-1} \BPG \langle m, m \rangle/G \cdot (\bar{v}_m^G - \gamma \bar{v}_m^G).\]
The underlying spectrum of $K_G(h)$ is the $2 \cdot (2^m-1)$-periodic Morava $K$-theory of height $h = 2^{n-1} \cdot m$, with coefficient ring 
\[\pi_*^e K_G(h) = \mathbb{F}_2[\bar{v}_m^\pm].\]  
The group $G$ acts trivially on $\pi_*^e K_G(h)$, but the action of $G$ on $K_G(h)$ is nontrivial and compatible with the stabilizer group $G$-action on the height-$h$ Lubin--Tate theory. 

An important feature of $\BPG\langle m, m \rangle$ is that its slice $E_2$-page contains significantly fewer classes compared to the slice $E_2$-page of $\BPG \langle m \rangle$.  The lengths of its differentials are also more concentrated in certain ranges.  These properties enhance the computational manageability of these theories, making them ideal as computable approximations for the equivariant truncated Brown--Peterson spectra $\BPG\langle m \rangle$ and $\BPG$.

\textit{To this end, the main objective of this paper is to investigate quotients of $\BPG$ by permutation summands, with a particular focus on the spectrum $\BPG\langle m, m \rangle$}.  By exploring the theoretical and computational properties of these equivariant integral Morava $K$-theories, we seek to gain a deeper understanding of the overall structure of $\BPG \langle m \rangle$ and the chromatic filtration tower (\ref{tower:chromFiltrationBPG}).  

\subsection{Main results}\label{sec:results}
We will now provide an outline of the paper and state our main results.  Throughout the paper, the group $G = C_{2^n}$.  

\subsubsection*{Section~\ref{sec:Section2}}
In the first section of this paper, we define various equivariant quotients by permutation summands and study their slice filtration.  We begin by defining permutation summands for \(\BPG\), which are collections of elements of the form \(G \cdot \bar{v}_j^G\) (see Definition~\ref{defn:PermBasis}).  Our main result in this section is the following, which provides a simple description of the slice associated graded for quotients by permutation summands: 
\begin{thmx}[Theorem~\ref{prop:SlicesOfRegularQuotients}]\label{thmA} The slice associated graded of the quotient 
\[
\BPG/(G \cdot \bar{v}_j^{G} \mid j\in J)
\]
where \(J\) is a subset of the natural numbers, is the generalized {\EM} spectrum
\[H\mZ[G\cdot\bv_{i}^G \mid i\not \in J].\]
\end{thmx}

Notably, Theorem~\ref{thmA} implies that the slice associated graded for \(\BPGZnm{m}{m}\)  is
\(H\mZ[G\cdot\bv_{m}^G]\).  We remark that our results do not depend on the specific choice of generators of the permutation summand: we can replace \(\bar{v}_m\) by any element \(\bar{s}_{2^m-1}\) in  \(\pi_{(2^m-1)\rho_2}^{C_2}\BPG\) that generates a permutation summand.  To streamline notation, we write 
\[G \cdot \bar{S} = \{G \cdot \bar{s}_{2^j-1} \mid j\in J\}\]
for \(\bar{S} = \{\bar{s}_{2^j-1} \mid j\in J\}\).

\subsubsection*{Section~\ref{sec:CHBPmm}}
In this section, we determine the chromatic heights of the generalized integral Morava $K$-theory spectra \(\BPG\langle m,m\rangle\).   
\begin{thmx}[Theorem~\ref{thm:BPGmmHeight}] \label{thmB}
The underlying spectrum of $\BPG \langle m, m\rangle$ has non-trivial chromatic localizations at heights equal to $rm$, where $0\leq r\leq 2^{n-1}$. That is,
\begin{enumerate}
\item for $r = km$ where $0 \leq k \leq 2^{n-1}$, $L_{K(r)} i_e^*\BPG \langle m, m \rangle \not\simeq * $, and
\item for all other $r \geq 0$, $L_{K(r)}i_e^*\BPG \langle m, m \rangle \simeq *$. 
\end{enumerate}
\end{thmx}
\noindent
In other words, the spectrum $\BPG \langle m, m \rangle$ captures chromatic information at heights $0$, $m$, $2m$, $3m$, $\ldots$, $2^{n-1}m$.  

We note that Theorem~\ref{thmB} is similar to \cite[Theorem~1.9]{BHSZOne}, where the underlying spectrum of $\BPG \langle m \rangle$ is shown to have nontrivial chromatic localizations at heights $\leq 2^{n-1}m$.  Furthermore, the $G$-actions on $\BPG \langle m, m \rangle$ and $\BPG \langle m \rangle$ are compatible with the $G$-action on $E_{2^{n-1}m}$ that is induced from the Morava stabilizer group.

\subsubsection*{Section~\ref{sec:locss}}
In this section, we introduce our main computational tool, the \textit{localized slice spectral sequence}.  This spectral sequence was developed in \cite{MeierShiZengSegal}, and we summarize some of its key features here.

Suppose $X$ is a $G$-spectrum and $H$ is a normal subgroup of $G$.  The localized slice spectral sequence is obtained by smashing the slice tower of $X$ with $\widetilde{E}\mathcal{F}[H]$, where $\mathcal{F}[H]$ is the family containing all proper subgroups of $H$.  By \cite{MeierShiZengSegal}, the localized slice spectral sequence converges strongly to the $G$-equivariant homotopy groups of $\widetilde{E}\mathcal{F}[H] \wedge X$, which is equal to the homotopy groups of $(\Phi^H(X))^{G/H}$ (see Theorem~\ref{MeierShiZeng1}).

When $G=H=C_{2^n}$, the localized slice spectral sequence is the $a_{\sigma}$-localization of the slice spectral sequence, where $\sigma$ the sign representation. If $G=C_{2^n}$ and $H=C_{2^{i}}$ for $i<n$, then the localized slice spectral sequence is the $a_{\lambda_{i-1}}$-localization of the slice spectral sequence, where $\lambda_{i-1}$ is the representation that rotates the plane by an angle of $\pi/2^{n-i}$.

For $G=C_{2^n}$, the slice spectral sequence of $X$ is divided into different regions, separated by the lines through the origin of slopes $(2^i-1)$, $0 \leq i \leq n$.  In \cite{MeierShiZengTranchromatic}, the authors proved the Slice Recovery Theorem (Theorem~\ref{thm:sliceRecovery}), which states that for $X$ a $(-1)$-connected $G$-spectrum, the map from the original slice spectral sequence of $X$ to the localized slice spectral sequence of $\widetilde{E}\mathcal{F}[C_{2^i}] \wedge X$ induces an isomorphism between all the differentials on or above the line of slope $(2^{i-1}-1)$ for all $1 \leq i \leq n$. In other words, even though the localized slice spectral sequence only computes the geometric fixed points, its $E_2$-page and differentials captures all the corresponding information in the original slice spectral sequence (which computes the fixed points) within a range.   


\subsubsection*{Section~\ref{sec:goefixG}}
In this section, we present the following computational result, using the tools discussed earlier. 

\begin{thmx}[Theorem~\ref{prop:geofixall} and Theorem~\ref{thm:unlocalizingasigma}]\label{thmC}
Let \(J \subseteq \mathbb{N}\), and let \(\bar{S} = \{\bar{s}_{2^j-1} \mid j\in J\}\) be a set of generators for permutation summands.  The following hold: 
\begin{enumerate}
\item The $a_\sigma$-localized slice spectral sequence of ${a_\sigma^{-1} \BPG/G \cdot \bar{S}}$ has only nontrivial differentials of lengths $\ell(k) = 2^n(2^{k+1}-1)+1$, where $k+1 \notin J$.  Moreover, for a fixed $k$, all the $\ell(k)$-differentials are multiples of a nontrivial $d_{\ell(k)}$-differential on the class $b^{2^k}$, where $b = u_{2\sigma}/a_\sigma^2$. 
\item The $a_\sigma$-localized slice spectral sequence of ${a_\sigma^{-1} \BPG/G \cdot \bar{S}}$, which converges to the homotopy groups of the $G$-geometric fixed points of ${\BPG/G \cdot \bar{S}}$, completely determines all the differentials on or above the line of slope ${(2^{n-1}-1)}$ in the slice spectral sequence of $\BPG/G \cdot \bar{S}$.  

\end{enumerate}
\end{thmx}

The computation in Theorem~\ref{thmC} is done using the Slice Differential Theorem of Hill--Hopkins--Ravenel \cite{HHR}.  The quotient map 
\[\BPG \longrightarrow \BPG/G \cdot \bar{S}\]
induces a map of the corresponding localized slice spectral sequences.  The Slice Differential Theorem produces all the differentials in the $a_\sigma$-localized slice spectral sequence of $a_\sigma^{-1}\BPG$.  Using the module structure and naturality, we deduce all the differentials in the $a_\sigma$-localized slice spectral sequence of $a_\sigma^{-1}\BPG/G \cdot \bar{S}$. 

When $G = C_2$,  the $a_\sigma$-localized slice spectral sequence of $a_\sigma^{-1} \BPR/\bar{S}$ produces all the differentials in the slice spectral sequence of $\BPR/\bar{S}$.   This is explained in Corollary~\ref{cor:BRPQuotient}.

Theorem~\ref{thmC} is used to show that, in stark contrast to the non-equivariant setting, most of the quotients of \(\BPG\) by permutation summands do not admit a ring structure even in the homotopy category.

\begin{thmx}[Theorem~\ref{prop:notaring}]\label{thmD}
Let \(J \subseteq \mathbb{N}\), and let \(\bar{S} = \{\bar{s}_{2^j-1} \mid j\in J\}\) be a set of generators for permutation summands. If there is a \(k\in J\) such that \((k+1)\notin J\), then \(\BPG/G\cdot\bar{S}\) does not have a ring structure in the homotopy category. 
\end{thmx}

\subsubsection*{Section~\ref{sec:relSS}}
In this section, we analyze the next region of the slice spectral sequence of $\BPG \langle m, m \rangle$, namely the region between the lines of slopes $(2^{n-2}-1)$ and $(2^{n-1}-1)$.  Let $\Gp = C_{2^{n-1}}$ be the index 2 subgroup of $G = C_{2^n}$.  To compute the differentials in this region, we examine the localized slice spectral sequence of 
\[\widetilde{E}\mathcal{F}[\Gp] \wedge \BPG \langle m, m \rangle \simeq a_{\lambda_{n-2}}^{-1}\BPG \langle m, m \rangle,\] 
which computes the homotopy groups of the $G/\Gp$-fixed points of the spectrum $\Phi^{\Gp}(\BPG \langle m, m \rangle)$.  

A valuable input to computing this spectral sequence is its restriction to the group $\Gp$, which computes the underlying homotopy groups of $\Phi^{\Gp}(\BPG \langle m, m \rangle)$. Using the Mackey functor structure, we can then deduce information about the $G$-equivariant spectral sequence from the simpler $G'$-equivariant spectral sequence.  

Another important spectral sequence that comes into play is the $H\F_2$-based Adams spectral sequence in the category of $A$-module spectra (as in Baker--Lazarev \cite{BakerLazarev}), where 
\[A:=H \mathbb{F}_2 \wedge H \mathbb{F}_2.\] Here, $H\F_2$ is given an $A$-module structure via the multiplication map $A \to H\F_2$.
 We call this spectral sequence the
\emph{relative Adams spectral sequence}. 

As non-equivariant spectra, there is an equivalence 
\[\Phi^{\Gp}(\BPG \langle m, m \rangle) \simeq  A/(\xi_i, \zeta_i: i \neq m)\]
for $\xi_i$ and $\zeta_i$ the usual Milnor generators and their conjugates.  In fact, there is an intimate connection between the more classical relative Adams spectral sequence and the $\Gp$-equivariant localized slice spectral sequence, which we establish in Section~\ref{sec:compare}.  
\begin{thmx}[Theorem~\ref{thm:adamsres}, Corollary~\ref{cor:speedupIsom}, and Corollary~\ref{cor:redindex}]\label{thmE}
After a reindexing of filtrations, the $\Gp$-equivariant localized slice spectral sequence of ${\BPG \langle m, m\rangle}$ is isomorphic to the relative Adams spectral sequence of \(A/(\xi_i, \zeta_i: i \neq m)\).
\end{thmx}

In Section~\ref{sec:compute}, we apply the techniques developed in \cite{BHLSZ} and use the correspondence established in Theorem~\ref{thmE} to obtain the following computational result:

\begin{thmx}[Theorem~\ref{thm:diffs} and Summary~\ref{sum:relvsslice}] \label{thmF}
We determine all the differentials in the following two spectral sequences: 
\begin{enumerate}
\item the relative Adams spectral sequence of $A/(\xi_i, \zeta_i: i \neq m)$;  
\item the $\Gp$-equivariant localized slice spectral sequence of $\BPG \langle m, m \rangle$.
\end{enumerate}
\end{thmx}

Theorem~\ref{thmF} demonstrates the effectiveness of our methods and allows us to proceed to the next stage of our analysis, which is to compute $\BPCfourtwotwo$.  As a concrete illustration of Theorem~\ref{thmF}, we present a detailed analysis of the $C_2$-equivariant localized slice spectral sequence of $\BPCfour \langle 2, 2 \rangle$ in Section~\ref{subsec:C2SliceSSBPCfourtwotwo}.

\subsubsection*{Section~\ref{sec:BPC422}}
In the final section of the paper, we provide a complete computation of the slice spectral sequence of $\BPCfourtwotwo$.  This serves as a showcase of the strength of our methods, and also offers insight into higher differentials phenomena when applied to higher heights and bigger groups. 

\begin{thmx} \label{thmG}
We determine all the differentials in the $a_\lambda$-localized slice spectral sequence of $\BPCfourtwotwo$.  The spectral sequence terminates after the $E_{61}$-page and has a vanishing line of slope $(-1)$ on the $E_\infty$-page.  
\end{thmx}

By the Slice Recovery Theorem (Theorem~\ref{thm:sliceRecovery}), this computation completely determines all the differentials in the slice spectral sequence of $\BPCfourtwotwo$ by truncating away the region below the line of filtration $s = 0$.  In particular, Theorem~\ref{thmG} shows that the slice spectral sequence of $\BPCfourtwotwo$ terminates after the $E_{61}$-page and has a horizontal vanishing line of filtration 61.  

According to Theorem~\ref{thmB}, the underlying spectrum of $\BPCfour \langle 2, 2 \rangle$ is of heights 0, 2 and 4, and is equipped with a $C_4$-action that is compatible with the stabilizer group action on a height-4 Lubin--Tate theory. 
The computation in Theorem~\ref{thmG} is a height-4 computation of a spectrum that is closely related to $\BPCfour \langle 2 \rangle$, studied in \cite{HSWX}.  There is a map 
\[\SliceSS(\BPCfour\langle 2 \rangle) \longrightarrow \SliceSS(\BPCfourtwotwo).\]
Compared to the slice spectral sequence of $\BPCfour\langle 2 \rangle$, the slice spectral sequence of $\BPCfourtwotwo$ has fewer classes on the $E_2$-page, and the lengths of differentials are concentrated in certain ranges. 

The following features of the slice spectral sequence of $\BPCfour \langle 2, 2 \rangle$ are essential to the computation in Theorem~\ref{thmG}:

\begin{enumerate}
\item Differentials on or above the line of slope 1 are determined by $a_\sigma$-localized slice spectral sequence, which is computed in Theorem~\ref{thmC}. 
\item The shorter differentials ($d_{\leq 31}$) are all determined from the $C_2$-slice differentials in Theorem~\ref{thmF} and Mackey functor structures. 
\item To determine the higher differentials, we identified two key classes, $\alpha = \dtwo^8 u_{24\sigma}a_{24\lambda}$ at $(48, 48)$ and $b^{32} = u_{32\lambda}/a_{32\lambda}$ at $(64, -64)$.  Multiplication with respect to these classes gives rise to periodicity of differentials and a vanishing line of slope $-1$ (Theorem~\ref{thm:VanishingThm}).  These phenomena determine all the higher differentials ($d_{>31}$).
\end{enumerate}

We believe these features can be extended to higher heights and larger groups, leading to a global description of all slice spectral sequence computations of quotients of $\BPG$ (see Section~\ref{subsec:openQuestions}).

\subsection{Open questions and future directions}\label{subsec:openQuestions}
The relationship between equivariant and chromatic homotopy theory is an exciting landscape whose exploration has only just begun. The results we present here reveal new aspects of the connection between $G$-spectra and $K(h)$-local phenomena. They also open questions and suggest conjectures. We end this introduction by highlighting a few.
\subsubsection*{$G$-equivariant periodic Morava-$K$ theory}

In this paper, our focus has mostly been on the theories $\BPG \langle m, m \rangle$.  Recall the spectrum 
\[K_G(h):= N_{C_2}^G(\bar{v}_m^G)^{-1} \BPG \langle m, m \rangle/G \cdot (\bar{v}_m^G - \gamma \bar{v}_m^G)\]
defined at the end of Section~\ref{subsec:Motivation}.  This spectrum is the $G$-equivariant generalization of the height-$h$ Morava $K$-theory.  The following questions about $K_G(h)$ will provide further insights into the structure of the Morava stabilizer group and the behavior of quotients of $\BPG$. 

\begin{quest}
What is the slice filtration of $K_G(h)$ and differentials in the slice spectral sequence of $K_G(h)$? 
\end{quest}
Note that since $K_G(h)$ is not a quotient by permutation summands, Theorem~\ref{thmA} does not directly apply.  Nonetheless, we have determined the slice filtration of $K_G(h)$ when $G = C_2$ and $C_4$, as well as the slice differentials for all $h$ when $G = C_2$, and for $h = 2$ when $G = C_4$.

\begin{quest}
Is it possible to build an equivariant chromatic fracture square with the theories $\BPG\langle m \rangle$, $\BPG \langle m-1 \rangle$, $\BPG \langle m, m \rangle$, and $K_G(h)$?  In particular, what is the relationship between $\BPG \langle m, m \rangle$, $K_G(h)$, and the fiber of the map $\BPG \langle m \rangle \to \BPG \langle m-1 \rangle$ in the chromatic filtration of $\BPG$? 
\end{quest}

\begin{quest}
What are the Hurewciz images of $\BPG \langle m, m \rangle$ and $K_G(h)$, compared to that of $\BPG \langle m \rangle$? 
\end{quest}

\subsubsection*{The $C_2$-equivariant relative Adams spectral sequence}

In Section~\ref{sec:relSS}, the correspondence established in Theorem~\ref{thmE} between the relative Adams spectral sequence and the $C_2$-localized slice spectral sequence played a crucial role in determining the $C_2$-slice differentials in the localized slice spectral sequence of $a_\lambda^{-1}\BPG\langle m , m \rangle$.  In \cite{HahnWilsonSegal}, Hahn and Wilson constructed a $C_2$-equivariant relative Adams spectral sequence, which can be utilized to compute the homotopy groups of $N_e^{C_2} H\mathbb{F}_2$-modules.  Notably, the $C_2$-geometric fixed points of quotients of $\BPCfour$ equipped with the residue $C_4/C_2$-action are $N_e^{C_2} H\mathbb{F}_2$-modules.  Both the $C_2$-equivariant relative Adams spectral sequence and the $C_4$-localized slice spectral sequence can be used to compute the homotopy groups of such quotients. 

\begin{quest}
Is there an equivariant analogue of the correspondence in Theorem~\ref{thmE} for $\BPCfour \langle m, m \rangle$ and general quotients of $\BPCfour$?  In particular, can we establish a correspondence between differentials in the $C_2$-equivariant Adams spectral sequence and the $C_4$-localized slice spectral sequence? 
\end{quest}

\subsubsection*{Global structures in the slice spectral sequence}

Understanding the equivariant homotopy groups of $\BPCfour \langle m, m \rangle$ for all $m \geq 1$ would significantly deepen our knowledge of higher chromatic heights and provide valuable insight into the $C_4$-fixed points of height-$(2m)$ Lubin--Tate theories. To facilitate these computations, it is important to establish certain general properties of the localized slice spectral sequences for these quotients.
 
Our computation of $\BPCfour \langle 2, 2 \rangle$ in this paper leads us to believe that certain structures, such as vanishing lines and periodicity of differentials, should be present in the localized slice spectral sequences for all $\BPCfour \langle m, m \rangle$.  Answering the following questions would significantly simplify the computation of localized slice spectral sequences for quotients of $\BPG$.   

\begin{quest}\label{conj:vanishingLine}
Do vanishing lines of slope $(-1)$ always exist on the $E_\infty$-pages of the localized slice spectral sequences for $\BPG \langle m, m \rangle$ and other quotients of $\BPG$?
\end{quest}

The presence of such vanishing lines is closely linked to the existence of horizontal vanishing lines in the slice spectral sequence for quotients of $\BPG$, a question raised in \cite{DuanLiShiVanishing}.  

\begin{quest}
Do there exist analogues of the classes $\alpha$ and $b^{32}$ that induce differential periodicity in the localized slice spectral sequence for $\BPG \langle m, m \rangle$ and other quotients of $\BPG$?
\end{quest}

For $\BPCfour \langle 1, 1 \rangle$, the analogous classes are $\alpha_1 = \done^4 u_{4\sigma}a_{4\lambda}$ in bidegree $(8, 8)$ and $b_1 = b^8 = u_{8\lambda}/a_{8\lambda}$ in bidegree $(16, -16)$.  In general, let
\[\alpha_m = \mathfrak{d}_{\bar{t}_m}^{2^{m+1}}u_{(2^m-1)2^{m+1}\sigma}a_{(2^m-1)2^{m+1}\lambda} \]
and 
\[b_m = b^{2^{2m+1}}= u_{2^{2m+1}\lambda}/a_{2^{2m+1}\lambda}.\]

\begin{conjecture}
In the $C_4$-localized slice spectral sequence of $\BPCfour \langle m, m \rangle$, multiplications by the classes $\alpha_m$ and $b_m$ induce differential periodicity and a vanishing line of slope $(-1)$. 
\end{conjecture}

\subsection{Acknowledgements}
The authors would like to thank Mark Behrens, Christian Carrick, Mike Hopkins, Hana Jia Kong, Guchuan Li, Yutao Liu, Lennart Meier, Juan Moreno, Doug Ravenel, Vesna Stojanoska, Guozhen Wang, Zhouli Xu, and Guoqi Yan for helpful conversations.  This material is based upon work supported by the National Science Foundation under Grant No. DMS-1906227 (first author), DMS-2105019 (second author) and DMS-2313842 (fourth author).

\section{Quotient modules of \texorpdfstring{\(\MUG\)}{MUG}}\label{sec:Section2}
\subsection{Slices for some \texorpdfstring{\(\MUG\)}{MUG} modules} 
For a graded ring \(R\) with augmentation ideal \(I\), let \(Q_nR\) denote the degree \(n\) elements in \(I/I^2\). One of the key computations in \cite{HHR} was a convenient choice of algebra generators for the \(*\rho_2\)-graded homotopy groups of \(\MUG\). In particular, we have an isomorphism of graded \(\Z[G]\)-modules
\[
Q_{\ast\rho_2}\big(\pi_{\ast\rho_2}^{C_2}\MUG\big)\cong 
\bigoplus_{k\geq 1} \Sigma^{k\rho_2} \Ind_{C_2}^{G} (\mathbb Z_{-}^{\otimes k}),
\]
where \(\mathbb Z_{-}\) is the integral sign representation.

\begin{definition}\label{defn:PermBasis}
Let \(J\subseteq \mathbb{N}\) and
\[
\bar{S}=\{
\bar{s}_j\in\pi_{j\rho_2}^{C_2}\MUG \mid j\in J
\}
\]
be a collection of elements. Associated to \(\bar{S}\), we have a \(C_2\)-equivariant map
\[
f_{\bar{S}}\colon \bigoplus_{j\in J} \Sigma^{j\rho_2}\mathbb Z_{-}^{\otimes j}\to
\bigoplus_{j\in J} Q_{j\rho_2}\big(\MUG\big).
\]
We say that \(\bar{S}\) \emph{generates a permutation summand} if the adjoint \(G\)-equivariant map 
\[
\tilde{f}_{\bar{S}}\colon \bigoplus_{j\in J} \Sigma^{j\rho_2}\Ind_{C_2}^{G}(\mathbb Z_{-}^{\otimes j})\to
\bigoplus_{j\in J} Q_{j\rho_2}\big(\MUG\big)
\]
is an isomorphism.
\end{definition}

Given any element in the \(RO(C_2)\)-graded homotopy of \(\MUG\), we can use the method of twisted monoid algebras from \cite{HHR}. For each \(j\in J\), we have a free associative algebra
\[
\mathbb S^0[\bar{s}_j]:=\bigvee_{i=0}^{\infty} S^{ij\rho_2},
\]
and we have a canonical associative algebra map
\[
\mathbb S^0[\bar{s}_j]\to i_{C_2}^\ast\MUG
\]
adjoint to the map defining \(\bar{s}_j\).
Using the norm maps on \(\MUG\) and the multiplication, we get an associative algebra map
\[
\mathbb S^0[G\cdot\bar{S}]=\bigwedge_{j\in J} N_{C_2}^{G}\mathbb S^0[\bar{s}_j]\to \MUG.
\]

\begin{definition}
For \(J\subseteq \mathbb{N}\), let 
\[
\MUG/G\cdot\bar{S}:=\MUG/\big(G\cdot\bar{s}_{j}\mid j\in J\big)=\MUG\smashover{\mathbb S^0[G\cdot \bar{S}]} \mathbb S^0.
\]
We call \(\MUG/G\cdot\bar{S}\) a \emph{quotient by permutation summands}.
\end{definition}

The following is a slight generalization of the Slice Theorem of Hill--Hopkins--Ravenel \cite[Theorem~6.1]{HHR}. Recall from \cite{HHR} that
\[
\pi_{*\rho_2}^{C_2}\MUG \cong \Z_{(2)}[G \cdot \bar{r}_1, G \cdot \bar{r}_2, \ldots ]
\]
for generators \(\bar{r}_i\) in \(\pi_{i\rho_2}^{C_2}\MUG\) introduced in (5.39) of \cite{HHR}. In fact, the conditions on the classes \(\bar{s}_j\) guarantee that we can use them instead of the \(\bar{r}_j\) for \(j\in J\). We then extend the set \(\bar{S}\) to form a set of equivariant algebra generators, as in \cite[Section 5]{HHR}. Any set 
\[
\bar{S}':=\{\bar{s}_j\mid j\notin J\}
\]
of elements \(\bar{s}_j \in \pi_{j\rho_2}^{C_2}\MUG\) which generate a permutation summand will do to extend   \(\bar{S}\) to a set  \(\bar{S}\cup \bar{S}'\) of equivariant generators for \(\pi_{*\rho_2}^{C_2}\MUG\).

\begin{definition}Let \(J \subseteq \mathbb{N}\) and \(\bar{S}\) and \(\bar{S}'\) be as above. 
Define 
\[H\mZ[G\cdot\bar{s}_1, G\cdot \bar{s}_2, \dots]:=H\mZ\smashover{\mathbb{S}^0} \mathbb{S}^0[G \cdot \bar{s}_1, \ldots]\] 
and
\[H\mZ[G\cdot\bar{s}_1, G\cdot \bar{s}_2,\dots]/G \cdot \bar{S} := H\mZ[G\cdot\bar{s}_1,\dots]\smashover{ \mathbb{S}^0[G\cdot \bar{S}]} \mathbb{S}^0 \ . \]
\end{definition}

\begin{remark}
The spectrum \(H\mZ[G\cdot\bar{s}_1, G\cdot \bar{s}_2,\dots]/G \cdot \bar{S}\) is very simple. In fact, it is equivalent to \(H\mZ[G \cdot \bar{S}']\), which itself is the smash product over \(j\not\in J\) of the norms, in the category of \(H\mZ\)-modules, of \(H\mZ[\bar{s}_j] \simeq \bigvee_{i=0}^{\infty} H\mZ\wedge S^{ij\rho_2}\) .
\end{remark}

\begin{theorem}\label{prop:SlicesOfRegularQuotients}
The slice associated graded of \(\MUG/G\cdot\bar{S}\) is the generalized {\EM} spectrum
\[
H\mZ[G\cdot\bar{s}_1,G\cdot \bar{s}_2,\dots]/(G\cdot\bar{S} ).
\]
\end{theorem}
\begin{proof}
We have a natural equivalence
\[
\MUG/G\cdot\bar{S}\simeq \MUG\smashover{\mathbb S^0[G\cdot\bar{s}_1,\dots]} \mathbb S^0[G\cdot\bar{S}'].
\]
The result now follows exactly as \cite[Slice Theorem 6.1]{HHR}, using the natural degree filtration on \(\mathbb S^0[G\cdot\bar{S}']\).
\end{proof}

Letting \(\bar{S}\) be the generators killed by the Quillen idempotent, this recovers the usual form of the slice associated graded for \(\BPG\). We could moreover always append this to any collection \(\bar{S}\) we consider, which allows us to deduce all of the analogous results for \(\BPG\). We will do so without comment moving forward.

\begin{remark}
 The left action of \(\MUG\) on itself always endows \(\MUG/G\cdot\bar{S}\) with a canonical \(\MUG\)-module structure, and the same is true with \(\BPG\) instead.
\end{remark}

\begin{notation}
In the homotopy of the spectrum \(\BPG\), let
\[
{\bv}_k^G:=\bar{t}_{k}^G\in\pi_{(2^k-1)\rho_2}^{C_2}\BPG,
\]
as defined and considered in \cite{BHSZOne}. 
\end{notation}

\begin{definition}\label{defn:formBPm}
For each \(m\geq 0\), let \(J_m=\{k\mid k>m\}\). Let 
\[
\bar{S}_m=\big\{\bar{s}_{2^j-1}\in \pi_{(2^j-1)\rho_2}^{C_2}\BPG \mid j\in J_m\big\}
\]
generate a permutation summand.  When for each \(j\in J_m\), \(\bar{s}_{2^j-1}={\bv}_j^G\), we name the quotient 
\[
\BPG/G\cdot\bar{S}_m=\BPGn.
\]
More generally, we say that the \(\BPG\)-module
\[
\BPG/G\cdot\bar{S}_m
\]
is a {\defemph{form of \(\BPGn\)}}.
\end{definition}

\begin{notation}
Let \(v_k^G\) be the restriction to the trivial group of \(\bar{v}_k^G\).
\end{notation}

\begin{remark}
Just as in \cite{HillLawson}, we note that since the underlying rings are all polynomial rings, the map
\[
\pi_{\ast}^{\{e\}}\BPG\to\pi_{\ast}^{\{e\}}\BPGn=\Z_{(2)}[G\cdot{v}_1^G,\dots,G\cdot{v}_{m}^G]
\]
has a section.

A form of \(\BPGn\) is a quotient module \(M\) with the property that for any section, the composite
\[
\Z_{(2)}[G\cdot{v}_1^G,\dots,G\cdot{v}_m^G]\to\pi_\ast^{\{e\}}\BPG\to \pi_\ast^{\{e\}}M
\]
is an isomorphism. The difference between the forms lies in the \(\BPG\)-module structure, not in the underlying homotopy groups.
\end{remark}

\begin{corollary}
The slice associated graded for any form of \(\BPGn\) is 
\[
H\mZ[G\cdot{\bv}_1^G,\dots,G\cdot{\bv}_m^G].
\]
\end{corollary}

\begin{definition}
Let \(k\) and \(m\) be natural numbers with \(1\leq k\leq m\). Let
\[
\bar{S}_{k,m}=\big\{\bv_j^G\mid 0<j<k\text{ or }j>m\big\},
\]
and let 
\[
\BPGZnm{k}{m}=\BPG/G\cdot \bar{S}_{k,m}.
\]
\end{definition}

\begin{remark}
As in Definition~\ref{defn:formBPm}, we also define forms of \(\BPGZnm{k}{m}\) as quotients by elements \(\bar{s}_{2^j-1}\), for \(0<j<k\) or \(j>m\) that generate permutation summands.
\end{remark}

\begin{corollary}\label{cor:sliceBPmm}
The slice associated graded for \(\BPGZnm{k}{m}\) (or for any form) is
\[
H\mZ[G\cdot\bv_{k}^G,\dots,G\cdot\bv_{m}^G].
\]
\end{corollary}

One of the main examples we will analyze is
\[
\kCnm[m]=\BPG/\big(G\cdot{\bv}_1^G,\dots,G\cdot{\bv}_{m-1}^G,G\cdot{\bv}_{m+1}^G,\dots\big)
\]
where \(m\geq 1\).

The slice associated graded for \(\kCnm[m]\) is very simple, given by
\[
H\mZ[G\cdot{\bv}_m^G].
\]

\section{Chromatic Height of \texorpdfstring{$\BPG\langle m,m\rangle$}{BP<m,m>}}\label{sec:CHBPmm}
In this section, we study the underlying chromatic height of the spectra $\BPG\langle m,m\rangle$.

\begin{theorem} \label{thm:BPGmmHeight}\hfill
\begin{enumerate}
\item For $r = km$ where $0 \leq k \leq 2^{n-1}$, $L_{K(r)} i_e^*\BPG \langle m, m \rangle \not\simeq *$;
\item For all other $r \geq 0$, $L_{K(r)}i_e^*\BPG \langle m, m \rangle \simeq *$. 
\end{enumerate}
\end{theorem}

\begin{proof}
Our proof will be similar to that of Proposition~7.4 and Theorem~7.5 in \cite{BHSZOne}.  In this proof, let $X = i_e^*\BPG \langle m, m \rangle$.  
For any $r$, there is a cofinal sequence $J(i) = (j_0, j_1, \ldots, j_{r-1})$ of positive integers and generalized Moore spectra 
\[M_{J(i)} = S^0/(v_0^{j_0}, \ldots, v_{r-1}^{j_{r-1}})\]
with maps $M_{J(i+1)} \to M_{J(i)}$ such that 
\[
L_{K(r)} X \simeq  \text{holim}_i \left(L_rX \wedge M_{J(i)} \right).
\]
See \cite[Prop. 7.10]{HoveyStrickland666}. 

Since \(X\) is a \(BP\)-module, it follows from \cite[Cor. 1.10]{Hovey_Bousfield} that the natural map \(L_r^fX \to L_rX\) is an equivalence (since it is a \(BP\)-equivalence between \(BP\)-local spectra). Therefore, 
\[L_rX \wedge M_{J(i)}\simeq L_r^fX \wedge M_{J(i)} \simeq X \wedge L_r^fM_{J(i)} \simeq X\wedge v_r^{-1} M_{J(i)} \simeq v_r^{-1}X\wedge M_{J(i)}.  \]
 Here, we have used the fact that since \(M_{J(i)}\) is a type \(r\) spectrum, its finite localization is the telescope \cite[Prop. 3.2]{MahowaldSadofsky}.  We have also used the fact that \(L_r^f\) is smashing.

To prove (1), we assume \(r\) is of the form \(km\) with \(0\leq k\leq 2^{n-1}\). We will first show that under the map 
\[
BP \longrightarrow v_r^{-1}X \wedge M_{J(i)},
\]
the image of $v_r \in \pi_*BP$ is nonzero. Note that 
\[
v_r^{-1}X \wedge M_{J(i)} = v_r^{-1}X \wedge_{MU} MU/(v_0^{j_0}, \ldots, v_{r-1}^{j_{r-1}}) = v_r^{-1}X/(v_0^{j_0}, \ldots, v_{r-1}^{j_{r-1}}).
\]
By an iterative application of the formula 
\[v_r^{C_{2^{n-1}}} \equiv v_r^{C_{2^n}} + \gamma_n v_r^{C_{2^n}} + \sum_{j = 1}^{r-1}\gamma_n v_j^{C_{2^n}} (v_{r-j}^{C_{2^n}})^{2^j} \pmod{I_r}\]
(where $I_r = (2, v_1, \ldots, v_{r-1})$) in \cite[Theorem 1.1]{BHSZOne}, the images of $v_r^j = (v_r^{C_2})^j$ in $(\pi_*X)/ (v_0, \ldots, v_{r-1})$ are all nonzero for $j \geq 1$.  This implies that their images are also nonzero in $\pi_* \left(X/(v_0, \ldots, v_{r-1})\right)$.  Therefore, the image of $v_r$ in $\pi_*(v_r^{-1}X/(v_0^{j_0}, \ldots, v_{r-1}^{j_{r-1}}))$ is nonzero.  After taking the homotopy limit, the image of $v_r$ under the map $\pi_*BP \to \pi_* L_{K(r)}X$ will also be nonzero.  It follows that $\pi_* L_{K(r)}X \not \simeq *$. 

To prove (2), we will consider two cases, based on the divisibility of $r$ by $m$.  If $r$ is not divisible by $m$, then the degree of $v_r$, $2(2^r -1)$, is not divisible by $2(2^m-1)$.  However, the homotopy groups of $X$ are concentrated in degrees that are divisible by $2(2^m-1)$.  This implies that the multiplication by $v_r$ map 
\[\Sigma^{|v_r|}X \longrightarrow X\]
induces the zero map on homotopy, and \[\pi_*v_r^{-1} X\cong  v_r^{-1}\pi_*X= 0.\]
It follows that $ v_r^{-1}X \simeq *$ and therefore $L_{K(r)}X \simeq *$. 

Now, suppose $m$ divides $r$.  Let $r = km$ for some $k > 2^{n-1}$. 
The result of \cite[Proposition 7.3]{BHSZOne} implies that $v_r \in (2,v_1, \ldots, v_{r-1})$ so that $v_r^q \in (v_0^{j_0}, \ldots, v_{r-1}^{j_{r-1}})$ for some $q>0$. 
Now,
\[ X \wedge M_{J(i)} = X \wedge_{MU} MU/(v_0^{j_0}, \ldots, v_{r-1}^{j_{r-1}})\]
and there is a K\"unneth spectral sequence \cite[Theorem IV.4.1]{EKMM}
\[E_2^{s,t} = \Tor_{-s, t}^{MU_*}(\pi_*X, \pi_* MU/(v_0^{j_0}, \ldots, v_{r-1}^{j_{r-1}})) \Longrightarrow \pi_{t-s} \left(X \wedge M_{J(i)}\right).\]
This is a cohomologically graded lower half-plane spectral sequence.
As in the proof of Theorem~7.5(2) in \cite{BHSZOne}, the fact that for some $q$ multiplication by $v_r^q$ raises filtration implies that every element in the homotopy groups of $X \wedge M_{J(i)}$ is killed by some finite power of $v_r$. It follows that $L_{K(r)} X = \text{holim}_i \left(v_r^{-1}X \wedge M_{J(i)}\right) \simeq *$. 
\end{proof}

\section{Localized spectral sequences}\label{sec:locss}
In our computations below, we will make use of various localizations of the slice spectral sequence of quotients of \(\MUG\). In this section, we recall results from \cite{SlicePrimer} and \cite{MeierShiZengSegal} that we will use here. As a reminder,  we continue to let \(G=C_{2^n}\).

\subsection{Some notation} Here, we introduce some notation. We refer the reader to \cite{HHR:HZ} for more details.

Consider \(2\)-local homotopy equivalence classes of representation spheres \(S^V\) where \(V\) is a finite dimensional orthogonal representation. This is a semi-group with respect to the smash product. Let \(JO(G)\) be the group completion.

\begin{definition}\label{defn:lambdas}
Define \(\lambda_{j} = \lambda_j(G)\) to be the \(2\)-dimensional irreducible real representation of \(G\) for which the generator \(\gamma \in G\) acts on \(\R^2\) by a rotation by \(2\pi/2^{n-j}\). 
We also have the one-dimensional sign representation \(\sigma_n =\sigma(G)\), for which the generator acts by multiplication by \(-1\). 
\end{definition}
Note that \(\lambda_{n-1}=2\sigma_n\) .
 There is an isomorphism of underlying abelian groups
\[JO(G) \cong \Z\{1, \sigma_n, \lambda_{0}, \ldots, \lambda_{n-2}\}  \]
where the equivalence sends \(S^V\) to \(V\).

\begin{definition}
For each representation \(V\), there is a homotopy class 
\[a_V \colon S^0 \to S^V\]
which corresponds to the inclusion of \(S^0 =\{0,\infty\}\). We call this the \emph{Euler class}.

If \(V\) is an orientable representation of dimension \(d\), we also get classes \[u_{V} \in \pi_{d-V}H\underline{\Z}.\] 
We call these \emph{orientation classes}.
\end{definition}

We have commutative diagrams
\[\xymatrix{S^0 \ar[r]^{a_{\lambda_i}}  \ar[dr]_-{a_{\lambda_{i+1}}} & S^{\lambda_i} \ar[d]^{(-)^2} \\
& S^{\lambda_{i+1}} }\]
where the vertical arrow is a double cover. Therefore, \(a_{\lambda_i}\) divides \(a_{\lambda_{i+1}}\) for each \(0\leq i \leq n-2\).

\subsection{Localizations and isotropy separation}\label{sec:locssdesc}

\begin{definition}\label{defn:families}
    For each \(0\leq i\leq n-1\), we have families
    \[
        \mathcal{F}_{i}=\mathcal{F}[C_{2^{i+1}}]=
        \{
         H\mid H\subsetneq C_{2^{i+1}}
        \}
    \]
    and \(\mathcal{F}_n=\mathcal All\) be the family of all subgroups of \(G\).
\end{definition}

These families interpolate between
\(\mathcal F_{0}=\{e\}\) and \(\mathcal F_{n}=\mathcal All\). 

The universal and couniversal spaces for the family \(\mathcal F_i\) can be written in very algebraic terms.

\begin{proposition}
    For \(0\leq i\leq n-1\), we have
    \[
        E\mathcal F_{i}\simeq \lim_{\to} S(k\lambda_i)=S(\infty \lambda_i)
    \]
    and
    \[
        \widetilde{E}\mathcal F_{i}\simeq S^{\infty\lambda_i}=S^0[a_{\lambda_i}^{-1}].
    \]
\end{proposition}
\begin{proof}
    The representation \(\lambda_i\) has kernel exactly \(C_{2^i}\subset G\), and the residual action of \(G/C_{2^i}\) is faithful. The result follows.
\end{proof}

We now state two results of Meier--Shi--Zeng that we will use later.
\begin{theorem}[Meier--Shi--Zeng \cite{MeierShiZengSegal}]\label{MeierShiZeng1}
For \(X\) a \(G\)-spectrum with regular slice tower \(P^{\bullet}X\), the spectral sequence associated to the tower  
\(\widetilde{E}\mathcal{F}_i \wedge P^{\bullet}X\), which corresponds to the \(a_{\lambda_{i}}\)-localized spectral sequence
\[ a_{\lambda_{i}}^{-1}E_2^{s,t+\alpha}= a_{\lambda_{i}}^{-1}\pi_{t-s+\alpha}^GP_t^tX \Longrightarrow \pi_{t-s+\alpha}^G(a_{\lambda_{i}}^{-1} X) , \quad \quad \alpha \in JO(G) \]
converges strongly.
\end{theorem}

\begin{theorem}[Meier--Shi--Zeng \cite{MeierShiZengTranchromatic}] \label{thm:sliceRecovery}
Let \(X\) be a $(-1)$-connected $G$-spectrum.  Let $L_i$ be the line of slope $(2^{i}-1)$ through the origin.  The following statements hold: 
\begin{enumerate}
\item On the integer graded page, the map from the slice spectral sequence of $X$ to the $a_{\lambda_{i}}$-localized slice spectral sequence of $X$ induces an isomorphism on the $E_2$-page for the  classes above $L_i$, and a surjection for the classes that are on $L_i$.  
\item The map of spectral sequences above induces an isomorphism between differentials that originate from classes that are on or above $L_i$. 
\end{enumerate}
\end{theorem}

In what follows, we will compute heavily with localized slice spectral sequences. The following remark explains the advantages of this approach.
\begin{remark}\label{rem:loccool}
It follows from Theorem~\ref{thm:sliceRecovery} that all the differentials in the slice spectral sequence of $X$ that are on or above $L_i$ can be immediately recovered from the $a_{\lambda_{i}}$-localized slice spectral sequence of $X$ by truncating off the latter spectral sequence below $L_i$.  In particular, all the differentials in the slice spectral sequence of $X$ can be recovered by truncating off the $a_{\lambda_0}$-inverted slice spectral sequence below the horizontal line $s=0$.  

In addition, the \(a_{\lambda_i}\)-localized slice spectral sequence are individually easier to compute that the non-localized spectral sequences. Of course, by localizing, we loose the information below the line  \(L_i\), but the approach is to work inductively, starting with the \(a_{\lambda_{n-1}}\)-localization (which is the same as the \(a_{\sigma_n}\)-localization) and ending with the \(a_{\lambda_0}\)-localization. As we explained above, all differentials can be recovered from the  \(a_{\lambda_0}\) localization so that at that stage, we have not actually lost any information at all.

One final remark on the advantage of computing with the  \(a_{\lambda_0}\)-localized spectral sequence is that it actually records information about the slice spectral sequences
\[E_2^{s,t+\star} = \pi_{t-s+\star}^GP_t^tX \Rightarrow \pi_{t-s+\star}^GX\]
for any \(\star=k\lambda_0\) where \(k\in \Z\). Indeed, we can recover all the differentials in the \(*+k\lambda_0\)-graded spectral sequence by truncating the \(a_{\lambda_0}\)-localized spectral sequence below the horizontal line \(s = -2k\). So, the localized spectral sequence contains much more information than simply the integer graded spectral sequence. 
\end{remark}

\section{ \texorpdfstring{\(C_{2^n}\)}{G}-geometric fixed points and quotients of \texorpdfstring{\(\BPG\)}{BPG}}\label{sec:goefixG}

As a proof-of-concept and for later computations, in this section, we will compute the homotopy of the geometric fixed points 
\[\Phi^G(\BPG/G \cdot \bar{S} )\simeq (\widetilde{E}\mathcal{P} \wedge \BPG/G \cdot \bar{S})^G  \simeq (a_{\sigma}^{-1}\BPG/G\cdot \bar{S})^{G}\] 
of quotients by permutation summands via the \(a_{\sigma}\)-localized slice spectral sequence.

On the one hand, we know the answer, since we know the homotopy type of the geometric fixed points.

\begin{proposition}
    We have a weak-equivalence of \(H\F_2\)-modules
    \[
        \Phi^G\big(\BPG/G\cdot\bar{S}\big)\simeq
        H\F_2\wedge\Sigma^\infty_+ \left(\prod_{j\in J} S^{2^j}\right).
    \]
\end{proposition}
\begin{proof}
    The geometric fixed points functor is strong symmetric monoidal, and we have
    \[
        \Phi^G(\BPG)=H\F_2.
\qedhere    \]
\end{proof}

On the other hand, the \(a_{\sigma}\)-localization map of the slice spectral has a particular simple target, and this will tell us a great deal about any of the slice spectral sequences for these quotients.

\subsection{General quotients of \texorpdfstring{\(\BPG\)}{BPG}}\label{sec:gen}

We now consider the \(a_\sigma\)-localized slice spectral sequence 
\begin{equation*}\label{eq:ssgeoG}
    a_{\sigma}^{-1}E_2^{s,t} :=  a_{\sigma}^{-1}\pi_{t-s}^GP_t^t(\BPG/G \cdot \bar{S})  \Longrightarrow \pi_{t-s}^G a_{\sigma}^{-1}(\BPG/G \cdot \bar{S} ).
    \end{equation*}

Inverting \(a_\sigma\) has the effect of killing the transfer from any proper subgroups. This means that the \(E_2\)-page of the \(a_\sigma\)-localizaed slice spectral sequence has a particular simple form:
\[
a_{\sigma}^{-1}E_2^{*,\star}=\mathbb F_2[N_{C_2}^{G}{\bar{s}}_i \mid i\not\in J][u_{2\sigma}, a_{\sigma}^{\pm 1}, a_{\lambda_0}^{\pm 1}, \ldots, a_{\lambda_{n-2}}^{\pm 1}] \ .
\]

\begin{definition}
    For each \(j\in\mathbb N\), let
    \[
        \bar{f}_j=a_{\bar{\rho}}^{2^j-1}N_{C_2}^{G}\bar{s}_j.
    \]
\end{definition}

This definition {\emph{a priori}} depends heavily on the choices of the \(\bar{s}_i\). However, from the point of view of differentials, these choices will not matter, due to a small lemma.

\begin{lemma}\label{lem:NormsForsj}
Let \(\bar{s}_{2^m-1}\) be any element in degree \((2^m-1)\rho_2\) that generates a permutation summand. We have
\[
N_{C_2}^G\bar{s}_{2^m-1}\equiv N_{C_2}^G{\bv}_m^G\mod (N_{C_2}^G{\bv}_1^G,\dots,N_{C_2}^G{\bv}_{m-1}^G)+Im(tr),
\]
where \(Im(tr)\) denotes the image of the transfer.

In particular, \(\bar{f}_i\) is independent of the choice of \(\bar{s}_i\), modulo the lower \(\bar{v}_j\).
\end{lemma}
\begin{proof}
We have
\[
\pi_{\ast\rho_G}^G\big(\BPG\big)/Im(tr)\cong\F_2\big[N_{C_2}^G\bv_1^G, N_{C_2}^G\bv_2^G, \dots\big],
\]
and the map
\[
x\mapsto N_{C_2}^G x
\]
gives a ring homomorphism
\[
\pi_{\ast\rho_2}^{C_2}\BPG\to \pi_{\ast\rho_G}^G\big(\BPG\big)/Im(tr).
\]
Additionally, Weyl equivariance of the norm shows that for any \(\gamma\in G\), 
\[
N_{C_2}^G(\gamma x)=\gamma N_{C_2}^G(x)\equiv N_{C_2}^{G}(x)\mod Im(tr).
\]

The lemma can be restated as saying that for any generator \[\bar{s}_{2^{m}-1} \in Q_{(2^m-1)\rho_2}\big(\pi_{\ast\rho_2}^{C_2}\BPG\big),\] 
we have
\[
N_{C_2}^G\bar{s}_{2^m-1}=N_{C_2}^G\bv_{m}^G\in Q_{(2^m-1)\rho_G}\big(\F_2[N_{C_2}^G\bv_1^G,\dots]\big).
\]
The above argument shows that the norm is a Weyl-equivariant ring homomorphism, and hence it induces a linear map
\[
\Big(Q_{(2^m-1)\rho_2}\big(\pi_{\ast\rho_2}^{C_2}\BPG\big)\Big)_G\to 
Q_{(2^m-1)\rho_G}\big(\F_2\big[N_{C_2}^G\bv_1^G,\dots\big]\big).
\]
Both the source and target are isomorphic to \(\F_2\), and choosing \(\bv_m^G\) as the generator of the source shows the map to be non-zero. It is therefore non-zero on any generator for the source.
\end{proof}

\begin{corollary}
    The \(E_2\)-term for the \(a_\sigma\)-localized slice spectral sequence for \(\BPG\) is given by
    \[
        \F_2[a_{\sigma}^{\pm 1}, a_{\lambda_{n-2}}^{\pm 1}, \dots,a_{\lambda_0}^{\pm 1}][b, \bar{f}_1,\dots],
    \]
    where the bidegree of \(\bar{f}_i\) is \(\big(2^i-1,(2^n-1)(2^i-1)\big)\) and where the bidegree of 
    \[
        b=u_{2\sigma}/a_{\sigma}^2
    \]
    is \((2,-2)\).
\end{corollary}

\begin{corollary}\label{cor:E2asigmaquotients}
    For any \(\bar{S}\), the  \(a_\sigma\)-localized slice spectral sequence for \(\BPG/G\cdot\bar{S}\) is a module over that for \(\BPG\), and the \(E_2\)-term is the quotient
    \[
        \F_2[a_{\sigma}^{\pm 1}, a_{\lambda_{n-2}}^{\pm 1}, \dots,a_{\lambda_0}^{\pm 1}][b, \bar{f}_1,\dots]/\big(\bar{f}_j\mid j\in J\big)
    \]
\end{corollary}

We start with examining the \(a_\sigma\)-localized slice spectral sequence for \(\BPG\), since all other cases are modules over this. 

\begin{proposition}
Let \(\bar{s}_i\) for \(i\in \mathbb{N}\) be any choice of permutation summand generators for \(\pi_{*\rho_2}^{C_2}\BPG\). %

Then in the \(a_\sigma\)-localized slice spectral sequence for \(\BPG\), the  differentials are determined by
\begin{equation}\label{eq:slicethmdiff}
d_{\ell(k)}(b^{2^k})=\bar{f}_{k+1}, \quad \ell(k):=2^n(2^{k+1}-1)+1, \quad k\geq 0 .
\end{equation}
\end{proposition}
\begin{proof}
Lemma~\ref{lem:NormsForsj} shows that 
\[\bar{f}_j\equiv a_{\bar{\rho}}^{(2^j-1)}N_{C_2}^{G}{\bar{v}}_j^G \] 
modulo the earlier generators \(\bar{f}_i\) with \(i<j\), so the Slice Differentials Theorem of \cite{HHR} implies that we have the differentials \eqref{eq:slicethmdiff}.
\end{proof}

\begin{remark}
The \(\bar{f}_i\) all lie on the line of slope \(2^n-1\) through the origin in the \((t-s,s)\)-plane. This is a vanishing line for both the spectral sequence of  \(a_\sigma^{-1}\BPG\) and that of quotient by permutation summands, so the differential in (\ref{eq:slicethmdiff}) is the last possible on \(b^{2^n}\). 
\end{remark}

We next use this result to study the \(a_\sigma\)-localized slice spectral sequence of other quotients.
\begin{definition}
 Let \(A(J)\) be the set of non-negative integers \(r\) such that the dyadic expansion \(r=\varepsilon_0+\varepsilon_1 \cdot 2+\varepsilon_2 \cdot 4+\ldots\) satisfies  \(\varepsilon_i=0\) if \(i+1\not\in J\).
\end{definition}
  
\begin{theorem}\label{prop:geofixall}
The $a_\sigma$-localized slice spectral sequence of $a_\sigma^{-1}(\BPG/G \cdot \bar{S})$ can be completely described as follows: 
\begin{enumerate}
\item the only non-trivial differentials are of lengths \(\ell(k) =2^n(2^{k+1}-1)+1\) for some \(k\geq 0\).
\item The \(E_{\ell(k)}\)-page  is the module over 
\[\F_2[\bar{f}_i \mid i\geq k+1 \ \text{and} \ i\not\in J][b^{2^k}]\] 
generated by the set of permanent cycles \(b^{r}\) where \(0\leq r<2^k\) and \(r\in A(J)\).
\item If \(k+1 \not\in J\), then there are non-trivial differentials are multiples of
\[d_{\ell(k)}(b^{2^k}) =\bar{f}_{k+1}\]
by the \(d_{\ell(k)}\)-cycles 
\[\F_2[\bar{f}_i \mid i\geq k+1 \ \text{and} \ i\not\in J][b^{2^{k+1}}]\{ b^r \mid r<2^k \ \text{and} \  r\in A(J)\}.\]
There are no other differentials of that length.
\item If \(k+1 \in J\), then \(E_{\ell(k)} = E_{\ell(k+1)}\). 
\item Consequently, \(E_{\infty} \cong \F_2\{b^r \mid r\in A(J) \}\) . 
\end{enumerate}
\end{theorem}

\begin{proof}
We will prove the statements by  induction on \(k\).  For the base case when \(k=0\),  we have \(\ell(0)=2^n+1\). The first possible non-trivial differential by sparseness is \(d_{\ell(0)}(b) = \bar{f}_1\). Therefore, \(E_{\ell(0)}=E_2\) and the class \(b^0\) is a permanent cycle.  The claims hold. 

Now, suppose that the \(E_{\ell(k)}\)-page is as claimed.  If \(k+1 \in J\), then \(b^{2^k}\) is a \(d_{\ell(k)}\)-cycle, and hence a permanent cycle. Any element on the \(E_{\ell(k)}\)-page of the form \(b^r\) with \(r<2^{k+1}\) is of the form \(r'+\varepsilon_k \cdot 2^k\) for \(r'<2^k\) with \(r'\in A(J)\) and \(\varepsilon_k\in \{0,1\}\).  Since \(k+1 \in J\),  \(r \in A(J)\).  Using the module structure over the $a_\sigma$-localized spectral sequence of \(a_\sigma^{-1}\BPG\), the elements \(b^r\) are also  \(d_{\ell(k)}\)-cycles. By sparseness of the  \(E_{\ell(k)}\)-page, it is impossible for $b^r$ to support a differential of length longer than $\ell(k)$ because there are no possible targets.  Therefore, the elements $b^r$ are permanent cycles. This implies that \(E_{\ell(k)}=E_{\ell(k+1)}\) which proves our claims when \(k+1\in J\).

On the other hand, if \(k+1 \not\in J\), we have a non-trivial differential \(d_{\ell(k)}(b^{2^k}) = \bar{f}_{k+1}\). For \(\alpha\geq 1\), consider the element \(\bar{f} \bar{f}_{k+1}^\alpha b^{r+2^kt}\) with \(0\leq r<2^k\), \(t\geq 0\) even, and \(\bar{f}\) a monomial in the \(\bar{f}_i\)'s for \(i>k+1\) and \(i\not\in J\). Such an element is the target of the \(d_{\ell(k)}\)-differential on \(\bar{f} \bar{f}_{k+1}^{\alpha-1}b^{r+2^k(t+1)}\).  It follows that the \(E_{\ell(k)+1}\)-page is a polynomial algebra over 
\[\F_2[\bar{f}_i \mid i\geq k+2 \ \text{and} \ i\not\in J][b^{2^{k+1}}]\] 
generated by the already established permanent cycles \(b^r\), \(r<2^{k}\). Since \(k+1\not\in J\), the set \(\{r \in A(J) \mid r<2^{k+1}\}\) is equal to \(\{r\in A(J) \mid r<2^k\}\).  This completes the induction step. \end{proof}

As an immediate consequence, using Theorem~\ref{thm:sliceRecovery}, we have: 
\begin{theorem}\label{thm:unlocalizingasigma}
In the integer graded slice spectral sequence
\[E_2^{s,t} = \pi_{t-s}^GP^{t}_t\BPG/G\cdot \bar{S} \Rightarrow \pi_{t-s}\BPG/G\cdot \bar{S},\]
we have:
\begin{enumerate}
\item Above the line of slope \(2^{n-1}-1\), the \(E_2\)-page is isomorphic to the integer graded part of the \(E_2\)-page of the \(a_{\sigma}\)-localization of the spectral sequence, as described in Corollary~\ref{cor:E2asigmaquotients}.
    \item The only non-trivial differentials whose sources lie on or above the line of slope \(2^{n-1}-1\) are in one-to-one correspondence with the non-trivial differentials of the integer graded \(a_{\sigma}\)-localized slice spectral sequence above this line.
    They are of lengths \(\ell(k)\) for \(k+1 \not\in J\), and are generated under the module structure of the spectral sequence for \(\BPG\) by the differentials 
\[d_{\ell(k)}(u_{2\sigma}^{2^k}) =a_{\sigma}^{2^{k+1}}a_{\bar{\rho}}^{2^{k+1}-1}N_{C_2}^G\bar{s}_{2^{k+1}-1}\]
for \(k+1 \not\in J\). 
\end{enumerate}
\end{theorem}

We will now give some examples to illustrate the results above. 
We start with example for the group \(G=C_{2}\). Since \(\bar{v}_j \equiv \bar{v}_{k+1}\) modulo the previous \(\bar{v}_j\)'s, we write \(\bar{v}_{k+1} = \bar{s}_{2^{k+1}-1}\), but any choice of permutation summand generators gives the same results.

\begin{example}
Consider \(\BPR/\bar{S}\) for \(\bar{S}=\{\bar{s}_{2^j-1} \mid j\in J\}\). The \(E_2\)-page are the \(a_{\sigma}\)-localized slice spectral sequence is
\[
a_{\sigma}^{-1}E_2^{*,\star}=\mathbb F_2[\bar{v}_i \mid i\not\in J][u_{2\sigma}, a_{\sigma}^{\pm 1}],
\]
the 
non-trivial differentials are generated by
\( d_{2^{k+2}-1}(b^{2^k}) = \bar{f}_{k+1}  \)
and 
\[E_{\infty}^{*,\star} \cong \F_2[a_{\sigma}^{\pm 1}]\{b^r \mid r\in A(J)\}.\]
Figure~\ref{fig:sliceSSBPCtwomm2} shows the example for \(\BP_{\R}\langle 2\rangle\).
\end{example}

\begin{figure}[ht]
\includegraphics[width=0.6\textwidth, page =1]{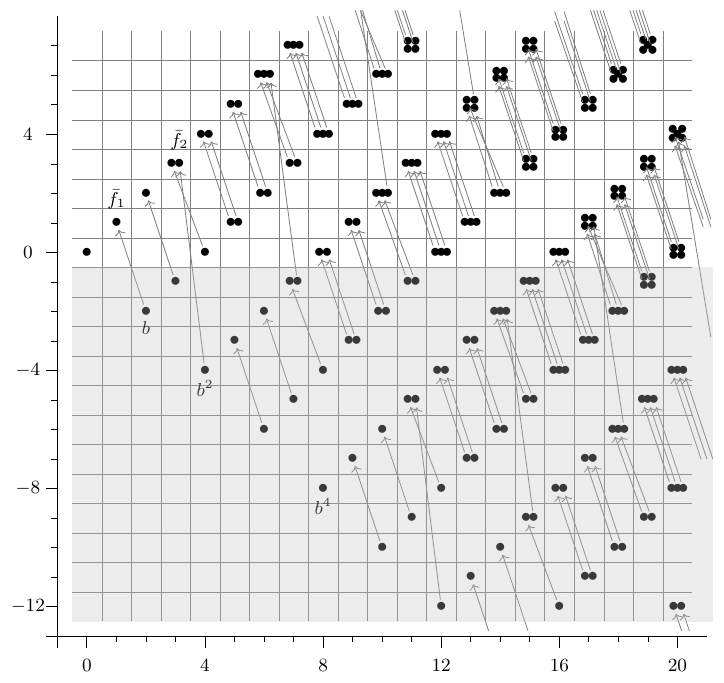}
\caption{The $a_\sigma$-localized slice spectral sequence of \(a_{\sigma}^{-1}\BP_{\R}\langle 2\rangle\) in integer degrees. The slice spectral sequence of \(\BP_{\R}\langle 2\rangle\) is obtained by removing the region below the horizontal line \(s=0\) and replacing \(\bullet=\Z/2\) by copies of \(\mathbb{Z}\), which reintroduces the transfers.}
\label{fig:sliceSSBPCtwomm2}
\end{figure}

In this case, using  Theorem~\ref{thm:sliceRecovery} as in Theorem~\ref{thm:unlocalizingasigma} together with the fact that the \(a_{\sigma}\)-localized spectral sequence records information about many   \(JO(C_2)\) degrees of the slices spectral sequence (as noted in Remark~\ref{rem:loccool}), we can easily  describe a large part of the (unlocalized) \(JO(C_2)\)-graded slice spectral sequence of \(\BPR/\bar{S}\).

\begin{corollary}\label{cor:BRPQuotient}
Let \(JO(C_2)^+ \subseteq JO(C_2)\) be the elements of the form \(a+b\sigma\) with \(a-b\geq 0\).
 The \(JO(C_2)^+\) graded slice \(E_2\)-page of \(BP_\R/\bar{S}\) is
\begin{align}
\label{eq:slicessquotientsbpr}
E_2^{*,\star} \cong \Z_{(2)}[\bar v_i \mid i\not\in J][u_{2\sigma}, a_{\sigma}]/(2a_\sigma), \quad \star \in JO(C_2)^+  \ .
\end{align}

\begin{enumerate}
\item The only non-trivial differentials are of lengths $(2^{k+2}-1)$ for some $k \geq 0$. 
\item If $k+1 \notin J$, then the nontrivial-differentials are multiples of 
\[d_{2^{k+2}-1}(u_{2\sigma}^{2^k+r} ) = \bar{v}_{k+1}u_{2\sigma}^ra_{\sigma}^{2^{k+2}-1}, \,\,\, r \in A(J)\]
by the $d_{2^{k+2}-1}$-cycles $\mathbb{F}_2[\bar{v}_i, a_\sigma \,|\, i \geq k+1 \text{ and } i \notin J]$.  
\item If $k+1 \in J$, then $E_{2^{k+2}-1} = E_{2^{k+3}-1}.$
\item Consequently, the $E_\infty$-page is 
\[E_{\infty}^{*,\star} \cong \mathbb{Z}_{(2)}[\bar{v}_i \mid i\not\in J][a_{\sigma}]/(2a_{\sigma}, \bar{v}_ia_{\sigma}^{2^{i+1}-1})\{u_{2\sigma}^{r}, 2u_{2\sigma}^s \mid r\in A(J), s \notin A(J) \}. \]
\end{enumerate}
\end{corollary}

As an explicit example, we show the computation of the slice spectral sequences of \(BP_{\R}\langle 2,2\rangle\), deduced from that of \(a_{\sigma}^{-1}BP_{\R}\langle 2,2\rangle\). The computation is illustrated in Figure~\ref{fig:sliceSSBPRtwomm2}.

\begin{remark}
On the $E_\infty$-page of $\BPR\langle 2, 2 \rangle$ (the third picture of Figure~\ref{fig:sliceSSBPRtwomm2}), there is an exotic extension $\eta = \bar{v}_1a_\sigma$, as shown by the dashed line.  This extension follows from the work in \cite{BHSZExtension}.  More precisely, letting $(n,k,b)=(2,1,0)$ in Corollary~3.11 of \cite{BHSZExtension} gives the exotic $\pi_\star \MUR$-multiplication 
$\bar{v}_1u_{2\sigma} = \bar{v}_2 a_{\sigma}^4$.  It follows that 
\[(\bar{v}_2 a_\sigma u_{2\sigma}) \cdot (\bar{v}_1a_\sigma) = (\bar{v}_2a_{2\sigma})\cdot(\bar{v}_1u_{2\sigma}) = (\bar{v}_2a_{2\sigma})\cdot(\bar{v}_2a_{4\sigma}) = \bar{v}_2^2 a_\sigma^6.\]
\end{remark}

\begin{figure}[ht]
\includegraphics[width=\textwidth, page =1]{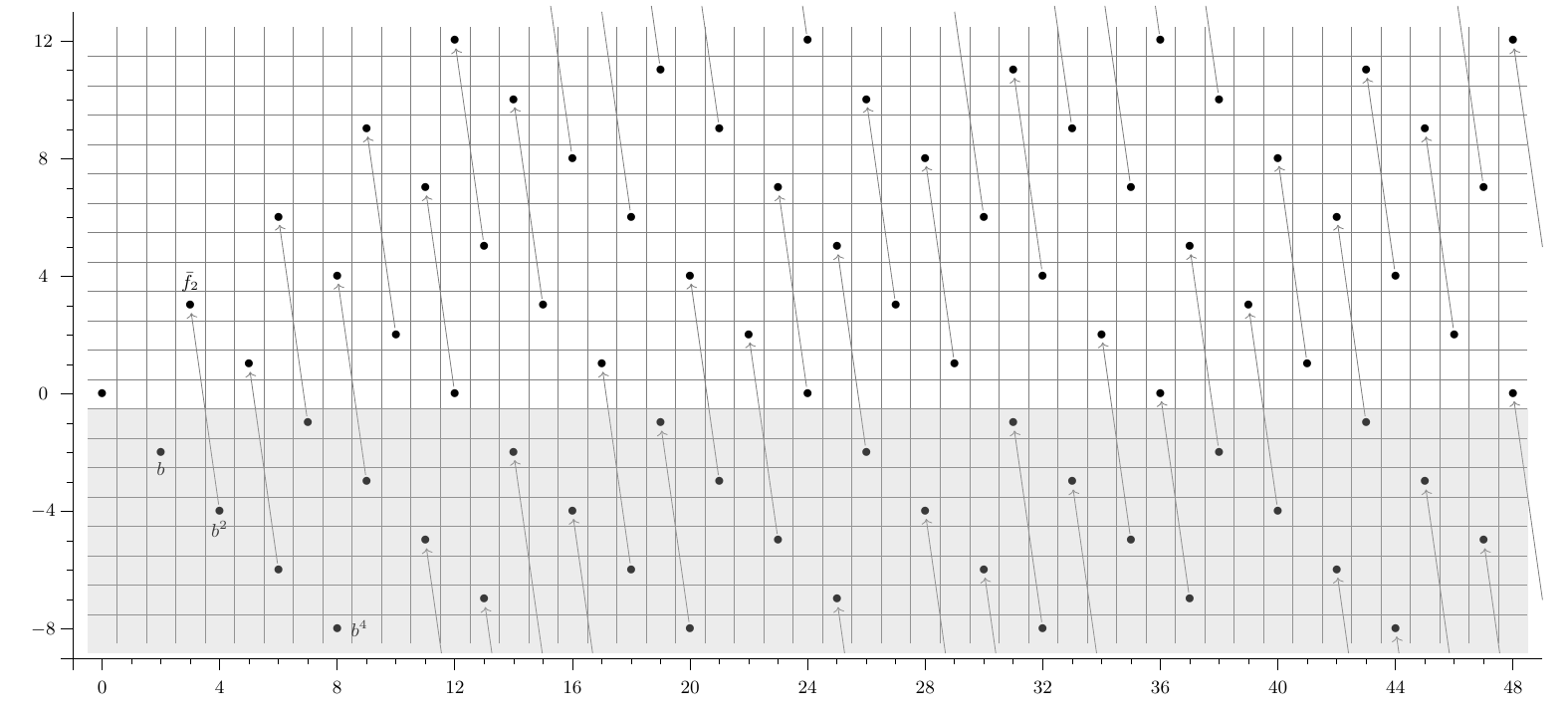}
\includegraphics[width=\textwidth, page =1]{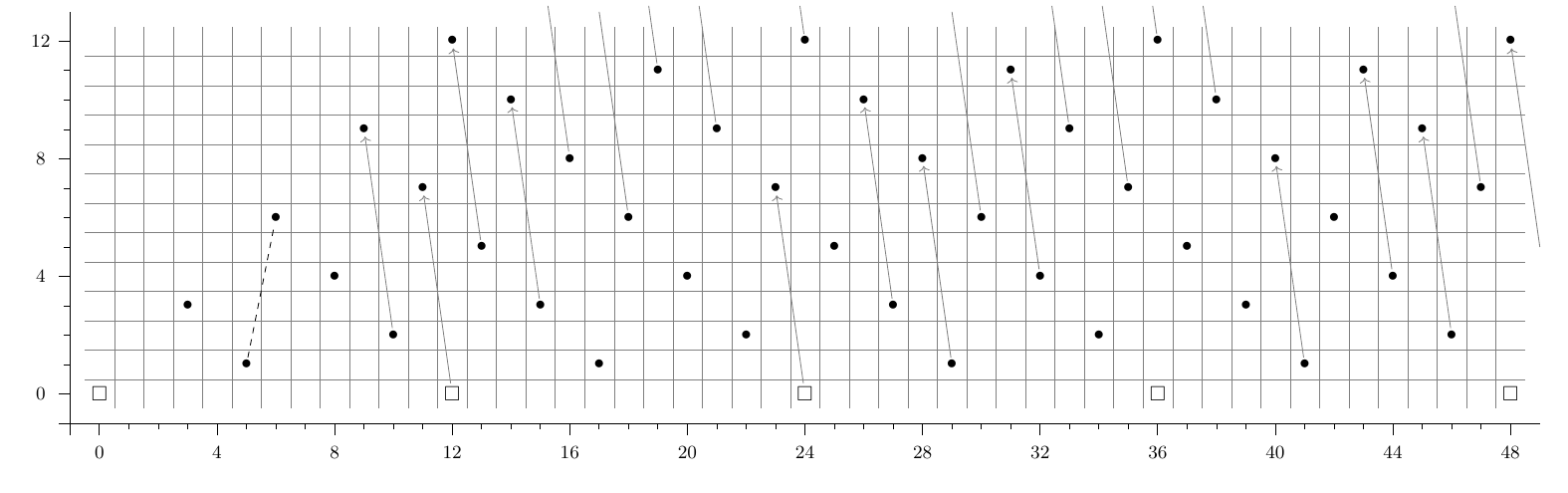}
\includegraphics[width=\textwidth, page =2]{BP22C2GeometricFixedPointsSSCUT.pdf}
\caption{The \(a_{\sigma}\)-localized slice spectral sequences of \(a_{\sigma}^{-1}\BP_{\R}\langle 2,2\rangle\) (top). The middle figure is the slice spectral sequence of \(\BP_{\R}\langle 2,2\rangle\) and the bottom is its \(E_{\infty}\)-page.  A \(\square\) denotes \(\Z_{(2)}\), a \(\bullet\) denotes \(\Z/2\).}
\label{fig:sliceSSBPRtwomm2}
\end{figure}

For the next two examples, the group \(G=C_4\) with generators \(\bar{s}_{2^i-1}  = \bar{v}_i^G\), but any choice of permutation summand generators gives the same results.

\begin{example}
Consider \(\BPCfourm[2]\), so that \(J = \{j \in \N \mid j>2\}\).  The \(E_2\)-page are the \(a_{\sigma}\)-localized slice spectral sequence is
\[a_{\sigma}^{-1}E_2^{*,*} \cong \F_2[\bar{f}_1, \bar{f}_2 ,b] \ .\]
The spectral sequence has two types of differentials, namely
\[d_5(b) = \bar{f}_1 , \quad \text{and} \quad d_{13}(b^2)=\bar{f}_2 \ . \]
The class \(b^4\) is a permanent cycle, and we have
\[
\pi_*^{C_4}a_\sigma^{-1}\BPCfourm[2] \cong \F_2[b^{4}]. \]
The computation is illustrated in Figure~\ref{fig:sliceSSBPCfourm2}.

\begin{figure}[ht]
\includegraphics[width=0.4\textwidth, page =1]{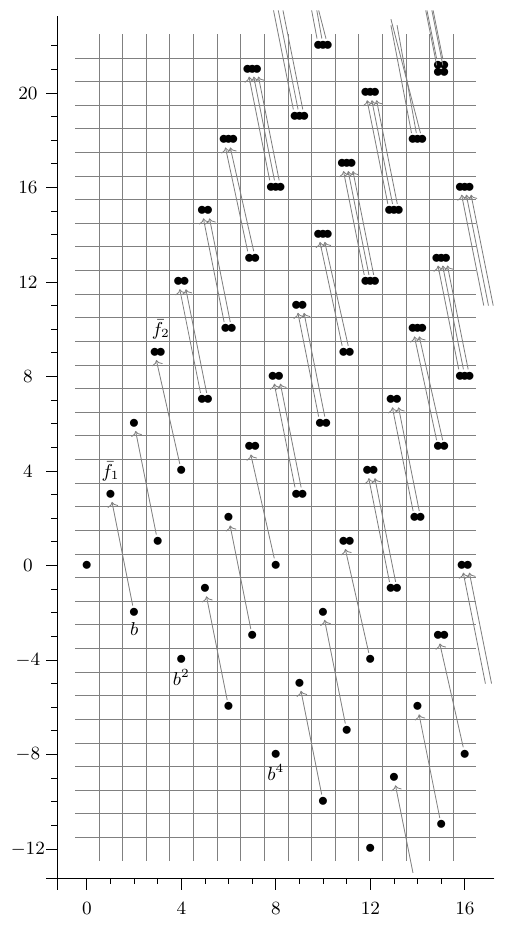}
\includegraphics[width=0.4\textwidth, page =2]{BP2C4GeometricFixedPointsSS.pdf}
\caption{The $a_\sigma$-localized slice spectral sequence of \(a_{\sigma}^{-1}\BPCfourm[2]\).}
\label{fig:sliceSSBPCfourm2}
\end{figure}
\end{example}

\begin{example}\label{ex:asigmaBPC422}
Consider \(\BPCfourmm[2]\). In this case, the \(E_2\)-page are the \(a_{\sigma}\)-localized slice spectral sequence is
\[a_{\sigma}^{-1}E_2^{*,*} = \F_2[\bar{f}_2,b] \ .\]
There is only one family of differentials, generated by 
\[d_{13}(b^2) = \bar{f}_2\]
and the answer is
\[\pi_*^{C_4}a_{\sigma}^{-1}\BPCfourmm[2] \cong \F_2[b^4]\{1,b\} \ . \]
The computation is illustrated in Figure~\ref{fig:sliceSSBPCfourmm2}.
\begin{figure}[ht]
\includegraphics[width=0.4\textwidth, page =1]{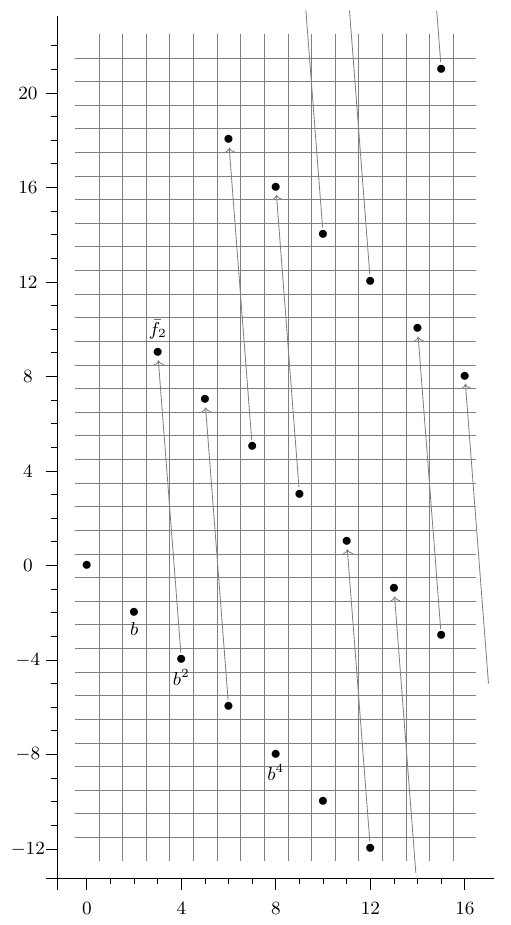}
\caption{The $a_\sigma$-localized slice spectral sequence of \(a_{\sigma}^{-1}\BPCfourmm[2]\).}
\label{fig:sliceSSBPCfourmm2}
\end{figure}
\end{example}

\subsection{Application: Multiplicative structure}
While the left action of \(\MUG\) on itself always endows \(\MUG/G\cdot\bar{S}\) with a canonical \(\MUG\)-module structure, and the same is true with \(\BPG\) instead,  much less is known for ring structures. We do have the following restrictive condition on quotients as a straightforward consequence of the techniques introduced above.

\begin{theorem}\label{prop:notaring}
Let \(J \subseteq \mathbb{N}\) and \(\bar{S} = \{\bar{s}_{2^j-1} \mid j\in J\}\) be a set of generators for permutation summands. If there is a \(k\in J\) such that \((k+1)\notin J\), then 
\(\BPG/G\cdot\bar{S}\) does not have a ring structure in the homotopy category. 
\end{theorem}
\begin{proof}
If there is a ring structure on \(\BPG/G\cdot \bar{S}\), then the map  \(\BPG/G\cdot \bar{S} \to H\mZ\) to the zero slice induces a map of multiplicative spectral sequences. This remains true after inverting \(a_{\sigma}\). Since \(\pi_*^G(a_{\sigma}^{-1}H\mZ) \cong \F_2[b]\) and the map from \(\pi_*^G(a_{\sigma}^{-1}\BPG/G\cdot \bar{S})\) to \(\pi_*^G(a_{\sigma}^{-1}H\mZ)\) is the natural inclusion, the former is a subring of the latter. However, if \(k\in J\) and \(k+1 \not\in J\), then \(b^{2^{k-1}}\) is nonzero in \(\pi_*^G(a_{\sigma}^{-1}\BPG/G\cdot \bar{S})\), but its square \(b^{2^k}\) is zero. This is a contradiction.
\end{proof}

Put another way, Theorem~\ref{prop:notaring} says that the only possible \(\BPG\)-module quotients  \(\BPG/G\cdot\bar{S}\) by permutation summands which could be rings are the forms of \(\BPGn\). Even here, we know very little.

\begin{example}
For \(G=C_2\), \(\BPR\langle 1 \rangle=k_{\mathbb R}\) and \(tmf_1(3)\) is a form of \(\BPR\langle 2 \rangle\). Both admit commutative ring structures. For \(m>2\) we do not know if \(\BPR \langle m \rangle\) admits an associative ring structure.

For \(G=C_4\), \(tmf_1(5)\) is a form of \(\BPCfourone\). For \(m>1\), we do not know if \(\BPCfour\langle m \rangle\) admits even an associative ring structure.
\end{example}

If we instead look only at the underlying spectrum, then work of Angeltveit and of Robinson shows that we have nice ring structures \cite{AngeltveitRegular,Robinson}. This has been refined by Hahn--Wilson to show that this is still true in the category of \(\MUG\) or \(\BPG\)-modules \cite{HahnWilsonQuotient}.

\begin{proposition}\label{prop:UnderlyingRing}
For any \(J\subseteq \mathbb{N}\) and for any \(\bar{S}\), the spectrum
\[
i_e^\ast\BPG/G\cdot\bar{S}
\]
is an associative \(i_e^\ast\BPG\)-algebra spectrum.
\end{proposition}
\begin{proof}
The assumptions on \(\bar{S}\) ensure that the sequence
\[
\big(\gamma^i\bar{s}_j\mid 0\leq i\leq 2^{n-1}-1, j\in J\big)
\]
forms a regular sequence in the homotopy groups of the even spectrum \(i_e^\ast\BPG\). The result follows from \cite[Theorem A]{HahnWilsonQuotient}.
\end{proof}

\begin{remark}
The Hahn--Shi Real orientation shows the restriction to \(C_2\) of the spectrum \(\BPG/G\cdot\bar{S}\) always admits an \(E_\sigma\)-algebra structure \cite{HahnShi}.
\end{remark}

\section{The \texorpdfstring{\(C_{2^{n-1}}\)}{index 2}-geometric fixed points} \label{sec:relSS}

Let \(\Gp = C_{2^{n-1}}\), the subgroup of index two in \(G=C_{2^n}\). We extend the results of the previous section, considering the \(a_{\lambda_{n-2}}\)-localized slice spectral sequence for permutation quotients. This is again a spectral sequence of Mackey functors, now essentially for \(C_{4}\cong G/C_{2^{n-2}}\). In this section, we study the \(C_2\cong G'/C_{2^{n-2}}\)-equivariant level, since we can tell an increasingly complete story here. The \(C_4\)-fixed points are more subtle, as we will see in Section~\ref{sec:BPC422}.

Note that since 
\[
    i_{G'}^\ast \lambda_{n-2}=\lambda_{(n-1)-1}=2\sigma,
\]
the restriction to \(G'\) of the \(a_{\lambda_{n-2}}\)-localized slice spectral sequence for \(\BPG/G\cdot \bar{S}\) is the \(a_\sigma\)-localized slice spectral sequence for 
\[
    i_{G'}^\ast \BPG/G\cdot\bar{S}.
\]
Just as for the \(G\)-geometric fixed points, we start by identifying the homotopy type. In this case, since 
\[
    i_{G'}^\ast\BPG\simeq \BP^{((G'))}\wedge\BP^{((G'))},
\]
we have
\[
    \Phi^{G'}\BPG\simeq H\F_2\wedge H\F_2,
\]
and all of the geometric fixed points we consider will take place in the category of modules over
\[
    A=H\F_2\wedge H\F_2.
\]
Composing with the localization map
\[
    i_{G'}^\ast\BPG\to \widetilde{E}\mathcal P\wedge i_{G'}^{\ast}\BPG,
\] 
the element \(N_{C_2}^{G'}\bar{s}_i\) gives us a polynomial
in the dual Steenrod algebra. 
\begin{definition}
    Let \(g_i\in\pi_i A\) be the image of \(N_{C_2}^{G'}\bar{s}_i\).
\end{definition}

Note that the residual \(C_2\cong G'/C_{2^{n-2}}\)-action interchanges
\[
    N_{C_2}^{G'}\bar{s}_i\text{ and }\gamma N_{C_2}^{G'}\bar{s}_i,
\]
while acting as the conjugation in the dual Steenrod algebra.

\begin{lemma}
    The \(G'\)-geometric fixed points of \(\BPG/G\cdot\bar{S}\) are the \(A\)-module
    \[
        A/(g_i,\bar{g}_i\mid i\in J).
    \]
\end{lemma}

In general, the homotopy type of this module very heavily depends on the choices of generators. We have several cases where we can explicitly identify the images, however. Using \cite[Proposition~2.57]{HHR} and \cite[Proposition~6.2]{MeierShiZengSegal}, we see that for Hill--Hopkins--Ravenel generators \(\overline{v}_i^G\) of \( \BPG \), the \(G'\)-geometric fixed points of their norms satisfy
\begin{align*}
\xi_i & = \Phi^{G'} N_{C_2}^G \overline{v}_i^G \\
\zeta_i & = \Phi^{G'}N_{C_2}^G \gamma \overline{v}_i^G,
\end{align*}
where $\xi_i$ are the Milnor generators of the mod \(2\) dual Steenrod algebra, and $\zeta_i$ are their dual. 

\subsection{Forms of \texorpdfstring{\(\BPGZnm{k}{m}\)}{k[m,m]}}
We can get much more explicit answers for the geometric fixed points with certain forms of \(\BPGZnm{k}{m}\), since here we can identify the geometric fixed points of the norms exactly.

\begin{corollary}
    The \(G'\)-geometric fixed points of \(\BPGZnm{k}{m}\) are given by the \(A\)-module
    \[
        A/\big(\xi_i,\zeta_i\mid i<k\text{ or }i>m\big)\simeq A/\big(\xi_i,\zeta_i\mid i<k\big) \smashover{A} A/\big(\xi_j,\zeta_j\mid j>m\big).
    \]
\end{corollary}

Writing this module in several ways makes working with this easier, as we can connect this with a series of modules and computations studied in \cite{BHLSZ}.

\begin{definition}
    For any subset \(I\) of the natural numbers, let \[
        M_I=\bigwedge_{i\in I} A/(\xi_i)\text{ and } \overline{M}_I=\bigwedge_{i\in I} A/(\zeta_i).
    \]
    Let 
    \[
        R_k=\End_A\big(M_{\{1,\dots,k-1\}}\big)
    \]
    and let
    \[
        A\langle m\rangle = M_{\{m+1,m+2,\dots\}}.
    \]
\end{definition}

As the endomorphisms of a module, \(R_k\) is always an associative algebra. By \cite{BHLSZ}, for any \(m\), \(A\langle m\rangle\) and \(\overline{A}\langle m\rangle\) are associative algebras as well. More surprisingly, by \cite{BHLSZ}, we have
\[
    R_k\simeq \Sigma^{} M_{\{1,\dots,k-1\}}\smashover{A}\overline{M}_{\{1,\dots,k-1\}},
\]
which allows us to rewrite \(\Phi^{G'}\BPGZnm{k}{m}\).

\begin{corollary}
    For any \(k\leq m\), we have
    \[
        \Phi^{G'}\BPGZnm{k}{m}\simeq\Sigma^{}R_k\smashover{A} A\langle m\rangle \smashover{A}\overline{A}\langle m\rangle,
    \]
    the suspension of an associative \(A\)-algebra.
\end{corollary}

The extreme case of this is \(\BPGmm{m}\).
\begin{corollary}\label{cor:GPrimeBPGmm}
    The \(G'\)-geometric fixed points of \(\BPGmm{m}\) are given by the \(A\)-module
    \[
        \Sigma^{} R_m\smashover{A}A\langle m\rangle \smashover{A}\overline{A}\langle m\rangle.
    \]
\end{corollary}
The homotopy of this \(A\)-module is more complicated than one might have initially expected. These kinds of modules were studied by the authors \cite{BHLSZ}, where we used a Baker--Lazarev style Adams spectral sequence based on \(H\F_2\)-homology, but in the category of \(A\)-modules \cite{BakerLazarev}. A remarkable feature of the case of \(\BPGmm{m}\) is that this relative Adams spectral sequence completely determines the \(a_\sigma\)-localized slice spectral sequence.

\subsection{A comparison of spectral sequences}\label{sec:compare}

Let \(P_{\bullet}=P_{\bullet}\BPGmm{m}\) be the slice covering tower of \(\BPGmm{m}\). That is, \(P_t\) is the homotopy fibre of the cannonical map 
\[\BPGmm{m} \rightarrow P^{t-1} \BPGmm{m} \] where \(P^{\bullet} = P^{\bullet}\BPGmm{m}\) is the regular slice tower. 

The slices \(P_t^t\BPGmm{m}\) are non-trivial only in dimensions of the form \(t=2i(2^m - 1)\). Therefore we can ``speed-up'' the slice tower without losing any information. Define
\[\tilde{P}_t = P_{2t(2^m - 1)} \ . \]
This re-indexes the slice tower, so that 
\[\tilde{P}^t_t = P^{2t(2^m - 1)}_{2t(2^m - 1)}\BPGmm{m}.\]

Since there is an equivalence 
\[\Phi^{\Gp} \BPG \simeq H\F_2 \wedge H\F_2 = \HA,\] 
\(\Phi^{\Gp} \tilde{P}_{\bullet}\)
is a covering tower converging to 
\( \Phi^{\Gp}\BPGmm{m}\)
in the category of \(\HA\)-modules. 

\begin{theorem}\label{thm:adamsres}
   The tower \(\Phi^{\Gp} \tilde{P}_{\bullet}\BPGmm{m}\)  is an \(H\F_2\)-Adams resolution of \(\Phi^{\Gp}\BPGmm{m}\) in the category of \(\HA\)-modules. 
\end{theorem}

\begin{proof}
Let \(Q_{\bullet} = \Phi^{\Gp}\tilde{P}_{\bullet}\BPGmm{m}\) for convenience. Then \(Q_{\bullet}\) is an \(H\F_2\)-Adams resolution of \(Q_{0} = \Phi^{\Gp}\BPGmm{m} \)  in \(\HA\)-modules if the following conditions are met for each \(i \geq 0\) \cite[Def.~2.2.1.3]{GreenBook}:
    \begin{enumerate}
        \item \(Q_i^i\) is a wedge of suspensions of \(H\F_2\)'s, and
        \item the map \(Q_i \rightarrow Q_i^i\) is monomorphic in \(H\F_2\)-homology.
    \end{enumerate}
We now verify the first condition. By definition, 
\[
Q_0^0 = \Phi^{\Gp} P^0_0 \BPGmm{m} = \Phi^{\Gp} H\underline{\mathbb{Z}} = H\F_2 [b],
\]
with \(\HA\)-module structure defined by the geometric fixed points of the reduction map \(\BPG \rightarrow H\underline{\mathbb{Z}}\). By \cite[Prop.~7.6]{HHR}, for each \(i\), \(\overline{v}_i^G\) and its conjugate \(\gamma \overline{v}_i^G\) act trivially on \(H\underline{\mathbb{Z}}\), thus the geometric fixed points of  \(N_{C_2}^{\Gp} \overline{v}_i^G \) and \( N_{C_2}^{\Gp}\gamma \overline{v}_i^G \), which are \(\xi_i\) and \(\zeta_i\), act trivially on \(H\F_2[b]\). Therefore, as an \(\HA\)-module, \(Q_0^0 \simeq \bigvee_{j = 0}^{\infty} \Sigma^{2j} H\F_2\). The Slice Theorem \cite[Thm.~6.1]{HHR} implies that for \(i > 0\), \(Q_i^i\) is a wedge of suspensions of \(Q_0^0\), thus the first condition is met.
    
We verify the second condition by an alternative construction of the slice covering tower of \(\BPGmm{m}\). As in \cite[\S6]{HHR}, let \(R = \mathbb{S}^0[G \cdot \overline{v}_m^G]\) be the homotopy refinement of \(\BPGmm{m}\), and \(M_{i}\) be the subcomplex of \(R\) consisting of spheres of dimension \(\geq 2i(2^m - 1)\). The arguments in \cite[\S6.1]{HHR} tell us that 
\[
    \tilde{P}_{i} \simeq \BPGmm{m} \wedge_R M_{i}.
\]
Notice that \(\Gp\)-equivariantly, \(M_{i+1} \subset M_{i}\) is the sub \(R\)-module \((\overline{v}_m^G,\gamma \overline{v}_m^G)M_{i}\), thus the quotient \(M_{i}/M_{i+1} \) is equivalent to \(M_{i}/(\overline{v}^G_m,\gamma \overline{v}_m^G)M_i\). Taking the \(\Gp\)-geometric fixed points on the cofibration
\[
     \BPGmm{m} \wedge_R M_{i+1} \rightarrow  \BPGmm{m} \wedge_R M_{i} \rightarrow \BPGmm{m} \wedge_R M_{i}/M_{i+1},
\]
we obtain the cofibration
\[
    Q_{i+1} \rightarrow Q_i \rightarrow Q_i^i \simeq Q_i/(\xi_m,\zeta_m)Q_i
\]
because \(\Phi^{\Gp}N_{C_2}^{\Gp} \overline{v}_m^G = \xi_m\) and \(\Phi^{\Gp} N_{C_2}^{\Gp}\gamma \overline{v}_m^G = \zeta_m\). Since \(\xi_m\) and \(\zeta_m\) have trivial image under \(A \rightarrow H\F_2\), the map \(Q_i \rightarrow Q_i^i\) induces a monomorphism in \(H\F_2\)-homology.
\end{proof}

\begin{corollary}\label{cor:speedupIsom}
    We have an isomorphism of spectral sequences between the relative Adams spectral sequence for \(A/(\xi_i,\zeta_i\mid i\neq m)\) and the speeded-up \(a_\sigma\)-localized slice spectral sequence for \(i_{G'}^{\ast}\BPGmm{m}\).
\end{corollary}

The dictionary here can be a little confusing, due to the scaling in the slice filtration. We record the un-scaled version here:
\begin{remark}
    A relative Adams \(d_r\) corresponds to an ordinary \(a_\sigma\)-localized slice differential \(d_{2(2^m-1)r+1}\).
\end{remark}

\begin{corollary}\label{cor:redindex}
    The integer graded \(E_{2^{m+1}}\)-page of the \(\Gp\)-equivariant \(a_{\sigma}\)-localized slice spectral sequence of \(i_{\Gp}^*\BPGmm{m}\) computing \(\pi_*^{\Gp}a_{\sigma}^{-1}i_{\Gp}^*\BPGmm{m}\) is isomorphic to the \(E_2\)-page of the relative Adams spectral sequence of the spectrum \(\Phi^{\Gp}i^*_{\Gp}\BPGmm{m}\).
\end{corollary}

\subsection{The Relative Adams spectral sequence for \texorpdfstring{\(\Phi^{\Gp}i^*_{\Gp}\BPGmm{m}\)}{G'-geometric fixed points}}\label{sec:compute}
By Corollary~\ref{cor:GPrimeBPGmm}, the \(G'\)-geometric fixed points of \(\BPGmm{m}\) are a suspension of an associative ring spectrum. Because of this, we instead work with the associative algebra, since then the spectral sequence will be one of associative algebras.

One of the most surprising results from \cite{BHLSZ} was a decomposition of \({A}\langle m\rangle\smashover{A}\overline{A}\langle m\rangle\), and hence of further quotients, into simpler, finite pieces. This makes our computations even easier.

\begin{definition}
    For each \(m\geq 1\), let
    \[
        \tilde{A}\langle m\rangle=A\langle m\rangle\smashover{A} \overline{M}_{\{m+1,\dots,2m\}}.
    \]
\end{definition}

\begin{proposition}[{\cite[Corollary 5.6]{BHLSZ}}]
    For each \(m\), we have a decomposition of \(A\)-modules
    \[
        A\langle m\rangle\smashover{A} \overline{A}\langle m\rangle\simeq \bigvee_{k\geq 0}\Sigma^{2^{2m+1}k}\tilde{A}\langle m\rangle.
    \]
\end{proposition}

\begin{proposition}[{\cite[Theorem 5.9]{BHLSZ}}]
    There is an associative algebra structure on \(\tilde{A}\langle m\rangle\) such that the projection map
    \[
        A\langle m\rangle\smashover{A} \overline{A}\langle m\rangle\to \tilde{A}\langle m\rangle
    \]
    is a map of associative algebras.
\end{proposition}

This reduces the study of modules of the form
\[
    M\smashover{A} A\langle m\rangle\smashover{A} \overline{A}\langle m\rangle
\]
to the study of \(\tilde{A}\langle m\rangle\)-modules of the form
\[
    M\smashover{A} \tilde{A}\langle m\rangle.
\]

We apply this in the case of \(M=R_m\), where we again have an associative algebra structure.
\begin{definition}
    Let 
    \[
        R\langle m\rangle=R_m\smashover{A} \tilde{A}\langle m\rangle.
    \]
\end{definition}

By \cite{BHLSZ}, the \(E_2\)-page of the relative Adams spectral sequence of  \(R\langle m \rangle\) is given by
\[
\F_2[\xi_1, \ldots, \xi_m,e]/e^{2^m} \otimes E(\beta_{\Neg2}, \beta_{\Neg4}, \ldots, \beta_{\Neg 2^{m-1}})/\big(\xi_i + \xi_{i+1}\beta_{\Neg 2^i} \mid 1 \leq i \leq m-1\big),
\]
where the bidegrees are given by: 
\begin{align*}
    |e|&=(2^{m+1},0)\\
    |\xi_m|&=(2^m-1,1)\\
    |\beta_{\Neg k}|&=(-k,0).
\end{align*}
Since for \(1\leq i<m\), we have the relation
\[
    \xi_i=\xi_{i+1}\beta_{\Neg2^{i}},
\]
this simplifies as an algebra to 
\[
\F_2[\xi_m] \otimes E(\beta_{\Neg2}, \beta_{\Neg4}, \ldots, \beta_{\Neg 2^{m-1}}) \otimes \F_2[e_{2^{m+1}}]/(e_{2^{m+1}}^{2^m}).
\]

\begin{notation}
We will use the following convenient notation:
\[
\beta(2\epsilon_{1} + 4 \epsilon_2+\ldots + 2^{m-1}\epsilon_{m-1}):=
\beta_{\Neg 2}^{\epsilon_1}\beta_{\Neg 4}^{\epsilon_{2}} \cdots \beta_{\Neg 2^{m-1}}^{\epsilon_{m-1}} \ .
\]
We get elements,
\[
\beta(0), \beta(2), \ldots, \beta(2^{m}-2)
\]
where \(\beta(\ell)\) has degree \(-\ell\). Note that, as element on the \(E_2\)-page for \(R\langle m\rangle\), for \(0\leq \ell\leq \ell'\)
\begin{align}\label{eq:multells}
\beta(\ell) \beta(\ell') =\begin{cases}
\beta(\ell') & \ell=0 \\
\left(\binom{\ell'}{\ell}-1\right)\beta(\ell+\ell')  & \ell>0 \ .
\end{cases}  \end{align}
\end{notation}

In \cite{BHLSZ}, we determined a number of differentials in these kinds of relative Adams spectral sequences.
\begin{proposition}[{\cite[Corollary~7.5]{BHLSZ}}]
    In the relative Adams spectral sequence for \(\widetilde{\HA}\langle m \rangle\), for each \(1 \leq i \leq m\) and \(n \geq 0\), we have differentials
    \[
        d_{1+2^{i+1}}(e^{2^{i} + 2^{i+1} n})= \xi_m^{2^{i+1}}\xi_{i+1} \cdot e^{2^{i+1}n}.
    \]
\end{proposition}
The spectrum \(R\langle m \rangle\) is an \(\widetilde{A}\langle m \rangle\)-algebra.  Therefore, by naturality and the relations on the \(E_2\)-page, in the relative Adams spectral sequence of \(R\langle m \rangle\) there are differentials 
\begin{align}\label{eq:keydiff}
d_{1+2^{i+1}}(e^{2^{i} + 2^{i+1} n})&= \xi_m^{1+2^{i+1}} e^{2^{i+1} n}
\beta(2^m-2^{i+1}) 
\end{align}
for \( 0 \leq i \leq m-1\) and \(0 \leq n < 2^{m-(i+1)}\), provided that the target survives to the \(E_{1+2^{i+1}}\)-page.  We will see that all other differentials will be determined by these and the multiplicative structure of the spectral sequence. 

We start with two useful lemmas.
\begin{lemma}\label{lem:diffkind}
If \(d_{1+2r}(\beta(\ell)e^k)\) is non-zero, then 
\[d_{1+2r}(\beta(\ell)e^k)= \xi_m^{1+2r} \beta(\ell+2^m-2r)e^{k-r}\]
for some \(1\leq r \leq k\).
\end{lemma}
\begin{proof}
Let
\[d_{1+2r}(\beta(\ell)e^k) =\xi_m^{1+2r} \beta(\ell')e^{k'} .\]
Then \(k'\leq k\) so we let \(s\) be a number such that \(k'=k-s\). Note that \(0\leq s<2^m\) and \( 0< r<2^m\).
Computing degrees, we obtain the equation
\begin{align}\label{eq:sort}
 2^{m+1}k - \ell -1 = (2^m-1)(1+2r)+2^{m+1}(k-s) -\ell'
\end{align}
This simplifies to
\[-2^m+(\ell'-\ell) + 2r=2^{m+1}(r-s).\]
Since \(0\leq \ell, \ell' \leq 2^m-2\), we have
\[ 2-2^m\leq \ell'-\ell \leq 2^m-2.\]
Furthermore, since \(0 < r < 2^m\), \(0 < 2r < 2^{m+1}\).  Therefore the absolute value of \(-2^m + (\ell' - \ell) + 2r\) is less than \(2^{m+1}\).  The equation above implies that this quantity is divisible by \(2^{m+1}\).  This implies that both sides of the equation must be zero.  It follows that \(r = s\) and \(\ell' = \ell + 2^m - 2r\). 
\end{proof}

The next lemma is a straightforward but annoying exercise analyzing  inequalities and we do not include the proof here.
\begin{lemma}\label{lem:pairs}
Consider pairs \((\ell, k)\), where \(\ell\) is even and \(0\leq \ell, k \leq 2^m-1\). Define 
subsets of such pairs by
\[
S =\{(\ell, k): k\leq \ell\}, \quad \text{and} \quad S' =\{(\ell, k): \ell< k\} 
\]
as follows. Let \(k\leq \ell\). Set
\begin{align*} 
  j &= \mathrm{max}\left\{0\leq r\leq m-1 :  \binom{\ell}{2^r} \equiv 0 \mod 2\right\}, \\
 i &= \mathrm{min}\left\{j\leq r\leq m-1 :  \binom{k}{2^r} \equiv 0 \mod 2\right\} .
\end{align*}
Then letting 
\[
\phi(\ell, k)=(\ell-(2^m-2^{i+1}), k +2^i)
\]
gives a bijection \(\phi\colon S\to S'\).
\end{lemma}
\begin{remark}
If \(\ell <k\) and we fix \(i\geq 0\), then for \(2^i\leq \kappa <2^{i+1}\) and  \(\kappa-2^i\leq \ell < \kappa\), if \(k\) can be written in the form \(k=\kappa + 2^{i+1}n\), then \(\phi^{-1}(\ell, k)=(\ell + 2^m-2^{i+1},k-2^{i})\). This formulation of the above bijection will be useful for proving the next result.
 \end{remark}
 
\begin{theorem}\label{thm:diffs}
In the relative Adams spectral sequence of \(R\langle m\rangle\),  for 
\[
0\leq k, \ell\leq 2^{m}-1
\] 
with \(\ell\) even, we have the following:
\begin{enumerate}
 \item the class \(\beta( \ell)e^k\)  is a permanent cycle if and only if 
 \(  k\leq \ell\);
 \item if \(\ell<k\), the class \(\beta(\ell)e^k\) supports a non-trivial differential, determined as follows. Fix \(i\geq 0\). For \(2^i\leq \kappa <2^{i+1}\) and  \(\kappa-2^i\leq \ell < \kappa\), if \(k=\kappa + 2^{i+1}n\), then there is a differential
\[
d_{1+2^{ i+1}}(\beta(\ell) e^{k})=\xi_m^{1+2^{ i+1}}\beta(\ell + 2^m-2^{i+1})e^{k-2^{i}} \ .
\]
These are the only non-trivial differentials.
\end{enumerate}
\end{theorem}

\begin{proof}
This is a multiplicative spectral sequence. At \(E_2\)-page, there is a vanishing line of slope \(1/(2^m-1)\) with intercept on the \((t-s)\)-axis at \(2^m-2\)
(the vanishing line is formed by the \(\xi_m\)-multiples on \(\beta(2^m-2)\)).
Furthermore, looking at the map of spectral sequences from \(\widetilde{A}\langle n \rangle\), we see that the class
\(e^{2^i}\) survives to the \(E_{1+2^{i+1}}\)-page for \(1 \leq i \leq m-1\). Therefore, the differentials \(d_r\) are \(e^{2^i}\)-linear for \(r<2^{i+1}+1\). The first non-zero class in positive filtration is \(\xi_m\beta(2^m-2)\) which has topological degree \((2^m-1)-(2^m-2)=1\). Therefore, every element of \(E(\beta)\)
is a permanent cycle and the spectral sequence is one of modules over this exterior algebra. 
 
We will prove the following statements inductively on \(0\leq j \leq m-1\): 
\begin{enumerate}
\item For \(2^j\leq k <2^{j+1}\), if \( k-2^j\leq \ell < k\),
then
\[d_{1+2^{ j+1}}(\beta(\ell) e^{k+2^{j+1}n})=\xi_m^{1+2^{ j+1}} \beta(\ell + 2^m-2^{j+1}) e^{k-2^{j} +2^{j+1}n} \ .\]
\item For \(2^j\leq k <2^{j+1}\) and \(k\leq \ell\), the class \(\beta(\ell)e^k\) is a permanent cycle. 
\item There are no other non-trivial differentials until the \(E_{1+2^{j+2}}\)-page.
\end{enumerate}

We note that (1) and (2) inductively imply that any class \(\beta(\ell)e^k\) with \(k<2^{j+1}\) either supports a differential \(d_r\) for \(r\leq 1+2^{j+1}\), or is a permanent cycle. 




%

To prove the inductive claim, we start with \(j=0\), so that \(k=1\) in (1). Using that \(\ell\) is even, in (1), the range forces \(\ell=0\). 
The first possible non-trivial differential for degree reasons is on \(e\), and this differential is forced by the \(d_3\)-differential
\[
d_3(e) = \xi_m^3\beta(2^m-2)
\]
in \(\widetilde{A}\langle n \rangle\).  All \(d_3\)-s are then determined by \(e^2\)-linearity and given by
\[
d_3(e^{1+2n})=\xi_m^3\beta(2^m-2)e^{2n}.
\]
Here, we have used the fact that the differentials are linear over \(\F_2[\xi_m]\otimes E(\beta)\).
For degree reasons, the classes \(\beta(\ell)e\) for \(\ell\geq 2\) are permanent cycles, proving (2). The differentials are \(e^2\)-linear and all other classes that could support a \(d_3\) are the product of \(e^2\) with permanent cycles. So they survive to the \(E_5\)-page.
 
Let \(i>0\) and assume that (1), (2), (3) hold for smaller values of \({0\leq j<i}\).
As noted above, the differentials in the spectral sequence of \(\widetilde{A}\langle m \rangle\) implies the differentials
\[
d_{1+2^{ i+1}}(e^{2^{i}+2^{i+1}n} ) 
=\xi_m^{1+2^{ i+1}}\beta(2^m-2^{i+1})e^{2^{i+1}n} ,
\]
provided that the targets survive to the \(E_{1+2^{i+1}}\)-page. By the induction hypothesis and Lemma~\ref{lem:pairs}, this is the case. This proves the differential of (1) for \(k=2^i\) and \(\ell=0\).


Now, choose \(k\) and \(\ell\) so that \(2^i \leq k <2^{i+1}\) and \( k-2^i \leq \ell < k\)
as in (1). In particular, \(\ell >0\).  The class \([\beta(\ell)e^{k-2^i}]\) is a permanent cycle by the induction hypothesis.
Therefore,
\begin{align*}
d_{1+2^{ i+1}}(\beta(\ell) e^{k+2^{i+1}n}) &=d_{1+2^{ i+1}}([\beta(\ell)e^{k-2^i} ]e^{2^{i} + 2^{i+1}n }) \\
&=[\beta(\ell)e^{k-2^i} ]d_{1+2^{ i+1}}(e^{2^i}) e^{ 2^{i+1}n} \\
&=[\beta(\ell)e^{k-2^i} ]\xi_m^{1+2^{ i+1}}\beta(2^m-2^{i+1}) e^{2^{i+1}n }   \\
&=\xi_m^{1+2^{ i+1}} \beta(\ell) \beta(2^m-2^{i+1}) e^{k-2^{i}+2^{i+1}n}  \ .
\end{align*}
The binomial expansion of \(2^m-2^{i+1}\) is
\[2^{m-1}+\ldots + 2^{i+2}+2^{i+1} . \]
The bounds on \(\ell\) give \(0<\ell <2^{i+1}\), which guarantees that \(\binom{2^m-2^{i+1}}{\ell}=0\) since \(\ell> 0\). So, by \eqref{eq:multells} 
\[\beta(\ell) \beta(2^m-2^{i+1}) =\beta(\ell + 2^m-2^{i+1}) .\]
We get a non-trivial differential as long as the target is non-zero, which is the case by the induction hypothesis and Lemma~\ref{lem:pairs}. This proves (1).

We next show that the classes \(\beta(\ell)e^{k}\) for \(k\leq \ell\) and  \(2^i\leq k<2^{i+1}\) are permanent cycles. Suppose that for \(2^i\leq r\leq k\),
\[d_{1+2r}(\beta(\ell)e^k) = \xi_m^{1+2r} \beta(\ell+2^m-2r)e^{k-r}.\]
The form of the differential comes from Lemma~\ref{lem:diffkind}. Note that
\[ k-r \leq \ell-r \leq \ell - r +2^m-r = \ell +2^m-2r.\]
This shows that the target is a permanent cycle by the induction hypothesis.  We will now show that this target is actually killed by a shorter differential.

Since
\[\ell+2^m-2r - (2^m-2^{i+1}) = 2^{i+1}+\ell-2r \geq 2^{i+1} +k -2k = 2^{i+1}-k>0. \]
Therefore, \(\ell+2^m-2r > 2^m-2^{i+1}\) and so we can write
\begin{align*}
 \ell+2^m-2r &= 2^m-2^{j+1} + \bar \ell ,  &  \bar\ell<2^j, \ 0\leq  j\leq i  \\
k-r &= \bar{k} + 2^{h}-2^{j} + 2^{h+1}n & j\leq h,\ \bar{k}<2^j,  \ 0\leq n 
\end{align*}

Let 
\[\ell' =\ell+2^m-2r-(2^{n}-2^{h+1})= \bar \ell + 2^{h+1}-2^{j+1} \] 
and 
\[\kappa' =k-r-2^{h+1}n+2^h = \bar{k} +2^{h+1}-2^j.\] 
Then,  \(\kappa'-2^h \leq \ell' <\kappa'\) and \(2^h \leq \kappa' <2^{h+1}\). So by the induction hypothesis,
\begin{align*}
d_{1+2^{h+1}}(\xi_m^{2r-2^{h+1}}\beta(\ell')e^{\kappa'+2^{h+1}n}) &= \xi_m^{2r-2^{h+1}} \left(\xi_m^{1+2^{h+1}} \beta(\ell' +2^m-2^{h+1})e^{\kappa'-2^h+2^{h+1}n}\right) \\
&=\xi_m^{1+2r} \beta(\ell+2^m-2r)e^{k-r} \ .
\end{align*}
Therefore, the target was killed by a shorter differential. This proves (2). 

Finally, (3) holds by the linearity of the differentials with respect to the \(d_{2^{i+2}}\)-cycle \(e^{2^{i+1}}\).
\end{proof}

Finally, the correspondence between the \(a_{\sigma}\)-localized slice spectral sequence of \(i_*^{\Gp} \BPG\langle m,m\rangle\) and the relative Adams spectral sequence is as follows:
\begin{summary}\label{sum:relvsslice}
The \(E_2\)-page of the relative Adams spectral sequence of the geometric fixed points \(\Phi^{\Gp}\BPGmm{m}\) is, additively,
\[ E^{*,*}_2 \cong  \Sigma^{2^m-2} \F_2[\xi_m] \otimes E(\beta) \otimes \F_2[e_{2^{m+1}}] \]
where the shift preserves the filtration and adds \(2^m-2\) to the topological degree.

The correspondence between the slice spectral sequence  and the relative Adams spectral sequence is as follows:
    \begin{enumerate}
        \item The elements \(b^{k}\) for \(0\leq k \leq 2^{m-1} - 1\) correspond to \(\beta(2^{m}-2(k+1))\). Note that \(b^{2^{m-1}-1}\) corresponds to \(\beta(0)=1\), the multiplicative unit in the relative Adams spectral sequence of \(R\langle m\rangle\). 
             \item For \(\ell = 2^{n-1}(2^{m}-1)+1\),  \(0\leq k <2^{m-1}\) and \(r \geq 0\), the element \(b^{k+2^{m-1}}\) in the localized sliced spectral sequence supports the differential \(d_{\ell}(b^{k+2^{m-1}}) = b^k(\bar{f}_m+\gamma\bar{f}_m)\),
forced by the
slice differential 
\[
d_{\ell}(u_{2\sigma}^{2^{m-1}}) = a_{\sigma}^{2^{m+1}-1}(\overline{v}_m^G + \gamma \overline{v}_m^G).
\]
These differentials are \(b^{2^m}\)-linear, leaving behind \(b^k\) where the dyadic expansion of \(k= \varepsilon_0 +\varepsilon_1\cdot 2+\ldots\) satisfies \(\varepsilon_{m-1}=0\).
This family of differentials identifies \(\overline{f}_m\) with \(\gamma \overline{f}_m\) in the whole slice spectral sequence.
        \item After the slice differentials above, multiplication by either \( \bar{f}_m\) or \( \gamma\overline{f}_m\) corresponds to multiplication by \(\xi_m\) in the relative Adams spectral sequence.
        \item The remaining elements \(b^k\) are in one-to-one correspondence with the elements in filtration \(s=0\) in the relative Adams spectral sequence, shifted by a degree \(2^{m} - 2\). In particular, the element \(b^{2^{m}+2^{m-1}-1}\) corresponds to \(e_{2^{m+1}}\).
    \end{enumerate}
\end{summary}

\subsection{The \texorpdfstring{$C_2$}{C2}-slice spectral sequence of \texorpdfstring{$\BPCfourtwotwo$}{BP C4 22}} \label{subsec:C2SliceSSBPCfourtwotwo}
As an application, we illustrate the above correspondence of spectral sequences by computing the  $C_2$-slice spectral sequence of $\BPCfourtwotwo$. 

The \(C_2\)-slice spectral sequence of \(\BPCfourtwotwo\) is determined by the relative Adams spectral sequence for
\[ R\langle 2\rangle = \widetilde{A}\langle 2 \rangle/(\zeta_5, \ldots) \wedge_A \End_{\HA}(M_1),\]
whose computation follows from Section~\ref{sec:compute} above. The essential features were also completely computed in Section~7.3 of \cite{BHLSZ}.
The \(E_2\)-page is 
\[
\mathbb{F}_2[\xi_2] \otimes E(\beta_{\Neg 2}) \otimes \mathbb{F}_2[e_8].
\]
There are only \(d_3\)- and \(d_5\)-differentials.  The \(d_3\)-differentials are generated by
\begin{align*} 
d_3(e_8^{1 + 2 *}) &= e_8^{2*}\xi_2^3 \beta_{-2},
\end{align*}
and the $d_5$-differentials are generated by
\begin{align*}
d_5(e_8^{2+4*}) &= e_8^{4*}\xi_2^5,  \\
d_5(e_8^{3+4*} \beta_{\Neg 2}) &= e_8^{1 + 4*} \xi_2^5 \beta_{\Neg 2}.
\end{align*}

The \(E_2\)-term of the \(a_{\sigma}\)-localized slice spectral sequence for \(a_\sigma^{-1}i_{C_2}^*\BPCfourtwotwo\) is 
\[\mathbb{F}_2[u_{2\sigma}, a_\sigma^{\pm 1}, \ttwo, \gamma \ttwo].\]
As before, let \(b =u_{2\sigma}/a_{2\sigma}\).
The shortest differentials in this spectral sequence are the \(d_7\)-differentials, whose effects are to identify \(\bar{t}_2\) with \(\gamma\bar{t}_2\). 
The \(d_7\)-differentials are generated by
\begin{align*}
d_7(b^{2+4*}) &= b^{4*}(\ttwo + \gamma \ttwo)a_\sigma^3, \\
d_7(b^{3+4*}) &= b^{1+4*}(\ttwo + \gamma \ttwo)a_\sigma^3.
\end{align*}

After the \(d_7\)-differentials, we can then import the differentials from the relative Adams spectral sequence as explained in Summary~\ref{sum:relvsslice}. As a result, the Adams \(d_3\)-differentials become the slice \(d_{19}\)-differentials, generated by
\[d_{19}(b^{5 + 8*}) = b^{1+8*} \ttwo^3 a_{\sigma}^{9}.\]
The Adams \(d_5\)-differentials become the slice \(d_{31}\)-differentials, generated by 
\begin{align*}
d_{31}(b^{9+16*}) &= b^{1+16*} \ttwo^5 a_\sigma^{15}, \\
d_{31}(b^{12+16*}) &=b^{4+16*} \ttwo^5 a_\sigma^{15}.
\end{align*}

Figure~\ref{fig:C2SliceSS} shows the $a_\sigma$-localized slice spectral sequence for $a_\sigma^{-1}i_{C_2}^*\BPCfourtwotwo$. At \(E_2\), the classes \(\bullet\) denote families of monomials formed by the classes $a_{\sigma}^3\bar{t}_2$ and $a_{\sigma}^3\gamma\bar{t}_2\}$ on the various powers of \(b\).
At \(E_8\), a \(\bullet\) depicted as the target of a \(d_7\)-differential becomes a copy of \(\Z/2\), represented by a class of the form \(a_{\sigma}^{3\ast}\bar{t}_2^{\ast}b^k\) for \(k\neq 2,3 \mod 4\). Each \(\bullet\) depicted as the source of a \(d_7\)-differential is completely gone as the \(d_7\)-differentials  are injective.

As in Theorem~\ref{thm:sliceRecovery}, to obtain the differentials in the \(C_2\)-slice spectral sequence of \(\BPCfourtwotwo\), we truncate at the horizontal line of filtration \(s=0\) and remove the region below this line.   

\begin{center}
\begin{figure}
\includegraphics[width=\textwidth]{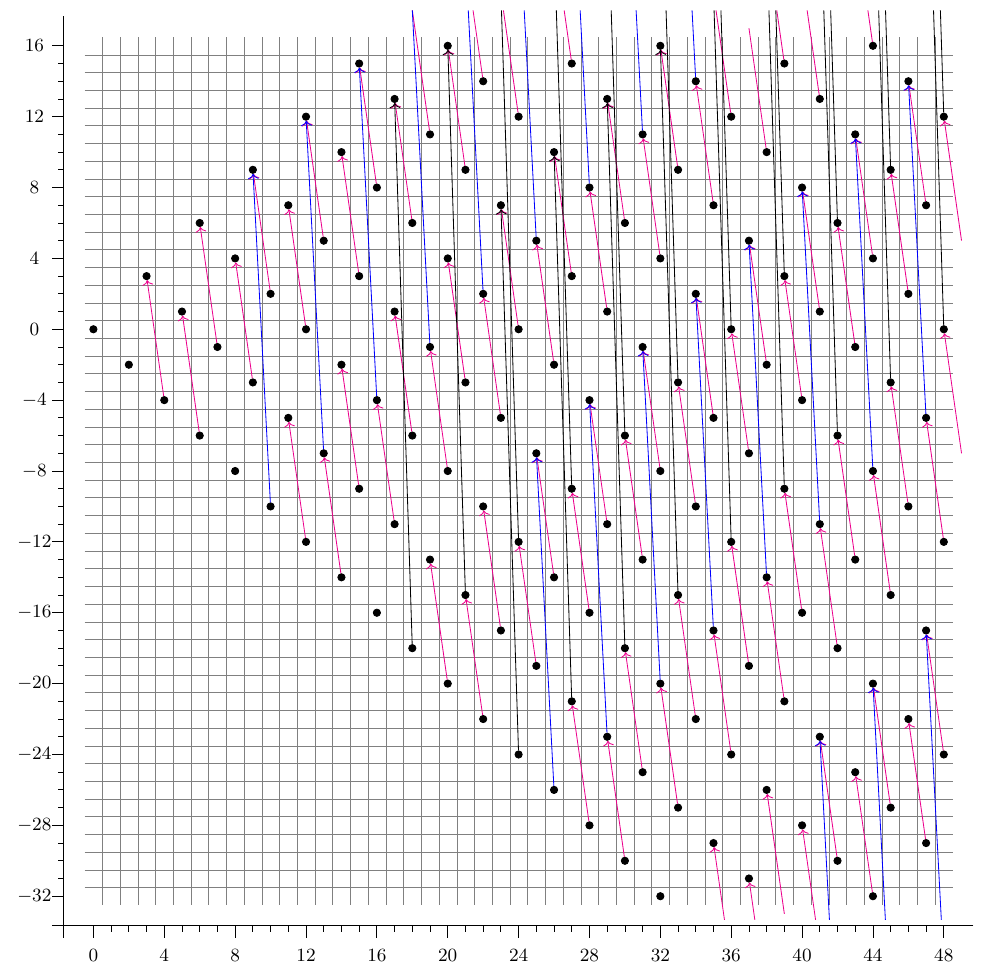}
\caption{The \(a_\sigma\)-localized slice spectral sequence of \(a_\sigma^{-1}i_{C_2}^*\BPCfourtwotwo\).}
\label{fig:C2SliceSS}
\end{figure}
\end{center}

\section{The \texorpdfstring{$C_4$}{C4}-localized slice spectral sequence of \texorpdfstring{\(a_\lambda^{-1}\BPCfourtwotwo\)}{the a-lambda localized BP C4 22}}\label{sec:BPC422}

In this section, we compute the integer-graded $C_4$-localized slice spectral sequence of $a_{\lambda}^{-1}\BPCfourtwotwo$, using all the tools that we have developed in the previous sections.  This computation serves to demonstrate the robustness of our techniques as well as providing insights to higher differentials phenomena when generalized to higher heights.

\begin{remark}Just like the integer-graded spectral sequence, the full $RO(C_4)$-graded spectral sequence can be computed by the exact same method.  We have opted to only compute the integer-graded slice spectral sequence because it is more convenient to present its diagrams. 
\end{remark}

\begin{remark}\label{rem:finalrecovery}
By the discussion in Section~\ref{sec:locss}, the slice spectral sequence of the unlocalized spectrum $\BPCfourtwotwo$ is completely determined by the $a_\lambda$-localized slice spectral sequence by truncating away the region below the line of filtration \(s = 0\).
\end{remark}

\begin{remark}
The following facts are good to keep in mind while doing the computation:
\begin{enumerate}
    \item The differentials with source on or above the line of slope 1 in the $C_4$-localized slice spectral sequence of $a_\lambda^{-1}\BPCfourtwotwo$ are  determined by the $C_4$-localized slice spectral sequence of $a_\sigma^{-1}\BPCfourtwotwo$ computed in Example~\ref{ex:asigmaBPC422}. These are all $d_{13}$-differentials.

\item Many of the differentials $d_{\leq 31}$ are determined  using the $C_2$-slice differentials computed in Section~\ref{subsec:C2SliceSSBPCfourtwotwo} and the Mackey functor structure (i.e. restriction and transfer).

\item The $C_4$-localized slice spectral sequence of $a_\lambda^{-1}\BPCfourtwotwo$ is a module over the spectral sequence of \(a_{\lambda}^{-1}\MUCfour\), but very little of that structure is needed for the computation (see Section~\ref{sec:strategy}). Multiplication with respect to two key classes gives rise to periodicity of differentials and a vanishing line (Theorem~\ref{thm:VanishingThm}). These phenomena determine all higher differentials ($d_{>31}$).
\end{enumerate}
\end{remark}

\subsection{Organization of the slice associated graded}
For the rest of this section, we let 
\[
\tone = \bar{v}_1^{C_4} \quad \text{and} \quad \ttwo = \bar{v}_2^{C_4} . 
\] 
From Corollary~\ref{cor:sliceBPmm}, the slice associated graded for \(\BPCfourtwotwo\) is \(\HZ[\ttwo, \gamma \ttwo]\), so the \(E_2\)-page of our spectral sequence is obtained by \(a_{\lambda}\)-localization of the homotopy of this slice associated graded.

We organize the slice cells by powers of 
\[\dtwo := N_{C_2}^{C_4}(\ttwo).\]  
\begin{remark}This mirrors the approach taken by Hill--Hopkins--Ravenel in \cite{HHRkH} to compute the slice spectral sequence of \(\BPCfourone = \BPCfour\langle 1, 1 \rangle\), where they organized the slice cells by powers of \(\done := N_{C_2}^{C_4}(\tone)\).  
\end{remark}

The slice cells are organized according to the following matrix:
\begin{align}\label{eq:matrix}\begin{pmatrix}
\dtwo^0 & \dtwo^1 & \dtwo^2 & \cdots \\ 
\dtwo^0(\ttwo, \gamma \ttwo) & \dtwo^1(\ttwo, \gamma \ttwo) & \dtwo^2 (\ttwo, \gamma \ttwo) & \cdots \\ 
\dtwo^0(\ttwo^2, \gamma \ttwo^2) & \dtwo^1 (\ttwo^2, \gamma \ttwo^2) & \dtwo^2 (\ttwo^2, \gamma \ttwo^2) & \cdots \\ 
\vdots & \vdots & \vdots & \ddots
\end{pmatrix}
\end{align}
To read this, note that \(i^*_{C_2}\dtwo = 
\ttwo \gamma \ttwo\) so that every monomial in \(\ttwo, \gamma \ttwo\) appears in some entry of the matrix. 

Entries \(\dtwo^j\) in the matrix represent slice cells of the form
\[ \HZ\wedge S^{3j\rho_4},\]
where \(\rho_4\) is the regular representation of \(C_4\).  These are the regular (or non-induced) cells.  Entries of the form \(\dtwo^j(\ttwo^i, \gamma \ttwo^i)\) represent slice cells of the form
\[\HZ\wedge ({C_4}_+ \wedge_{C_2} S^{i^*_{C_2}(3j\rho_4+3i\rho_2)}) \simeq \HZ\wedge ({C_4}_+ \wedge_{C_2} S^{(6j+3i)\rho_2}).\]
These are the induced cells.

For the non-induced cells, the homotopy groups of \(\HZ\wedge S^{*\rho_4}\) are computed in \cite{HHRkH} and  depicted in Figure~3 of that reference. For the induced cells, we get the induced Mackey functors
\[ \HZ_*({C_4}_+ \wedge S^{*\rho_2}) \cong \uparrow_{C_2}^{C_4}\HZ_*^{C_2}(  S^{*\rho_2}), \]
(see Definition~2.6 of \cite{HHRkH})
whose values are also known.
 After inverting \(a_{\lambda}\), we obtain the \(E_2\)-page depicted in Figure~\ref{fig:BP22C4E2}.

We use the same notation as in Table~1 of \cite{HHRkH} for the Mackey functors. Blue Mackey functors are supported on induced cells and represent multiple copies of the Mackey functor \(\hat{\bullet}\) of \cite{HHRkH}, supported on the various monomials in the matrix \eqref{eq:matrix} that are not in the first row.

We name the classes that do not come from induced cells. First, there are classes of order four (which have non-trivial restrictions):
\begin{itemize}
  \item   \(\dtwo^{2i} u_{(6i-j)\lambda}u_{6i\sigma}a_{j\lambda}\) in degree \((24i-2j ,2j)\) for \(0\leq i\) and \(j\leq 6i\).
   %
   
   %
   \end{itemize}
   Next, the classes of order two which do not come from induced cells are:
   \begin{itemize}
   \item \(\dtwo^{i}u_{2k\sigma}a_{3i\lambda}a_{(3i-2k)\sigma} \) in degree \((3i+2k,9i-2k)\) for \(0\leq i\) and \(0\leq 2k < 3i\). These are above the line of slope one.
   \item \(\dtwo^{i}u_{j\lambda} u_{(3i-1)\sigma}a_{(3i-j)\lambda}a_{\sigma} \) in degree \((6i+2j-1,6i-2j+1)\) for \(i\) odd,  \(0\leq i\) and \(0\leq j\). 
\end{itemize}
The induced cells are named by treating them as images of the transfer map from the corresponding classes in the $C_2$-slice spectral sequence. 

\begin{center}
\begin{figure}
\includegraphics[width=\textwidth]{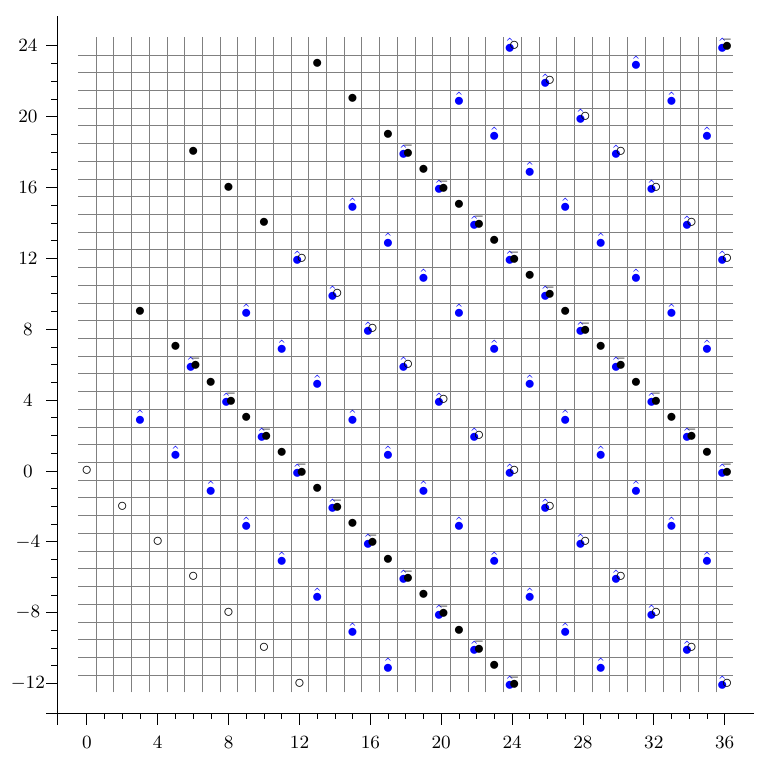}
\caption{The \(E_2\)-page of the Mackey functor valued \(a_\lambda\)-localized slice spectral sequence of \(\BPCfourtwotwo\).}
\label{fig:BP22C4E2}
\end{figure}
\end{center}

\subsection{The \texorpdfstring{\(d_7\)}{d-seven}-differentials}
The first differentials are the $d_7$-differentials.  They occur between classes supported on slice cells that are in the same column of the matrix \eqref{eq:matrix}.  The $d_7$-differentials are all proven by restricting to the $a_\sigma$-localized slice spectral sequence of $a_\sigma^{-1}i_{C_2}^*\BPCfourtwotwo$.  More specifically, the restriction of certain classes to the $C_2$-spectral sequence support the $d_7$-differentials discussed in Section~\ref{subsec:C2SliceSSBPCfourtwotwo}, and therefore by naturality and degree reasons their preimages must also support $d_7$-differentials in the $C_4$-spectral sequence.  By naturality and degree reasons, these are all the \(d_7\)-differentials that can occur, after which we obtain the \(E_8\)-page.

Let $b = \frac{u_\lambda}{a_\lambda}$.  The restriction of this class is the class $b = \frac{u_{2\sigma}}{a_{2\sigma}}$ in the $C_2$-spectral sequence.  We have \(d_7\)-differentials 
\begin{align*}
    d_7\left(b^{2+4*}\right)&=\tr_{C_2}^{C_4}(\ttwo a_{3\sigma_2})b^{4*}\\
           d_7\left(b^{3+4*}\right)
     &= \tr_{C_2}^{C_4}(\ttwo a_{3\sigma_2})b^{1+4*}.
\end{align*}

Figure~\ref{fig:BP22C4E7} depicts the \(E_7\)-page.
 The \(\circ\) classes are \(\Z/4\)'s coming from non-induced cells.
 The black \(\bullet\) classes are \(\Z/2\)'s coming from the non-induced cells, while the blue  \({\color{blue}\bullet}\) classes are direct sums of \(\Z/2\)'s coming from the induced cells.

 Figure~\ref{fig:C4E8page} depicts the \(E_8\)-page.  In that figure, all dots represent a copy of \(\Z/2\), with the exception of the white \(\circ\) classes which represent \(\Z/4\)'s. The pink denotes a degree where, on a previous page, there was a \(\Z/4\) but the generator supported a non-zero differential and is no longer present. As before, blue classes come from induced cells.
 
\begin{center}
\begin{figure}
\includegraphics[width=\textwidth]{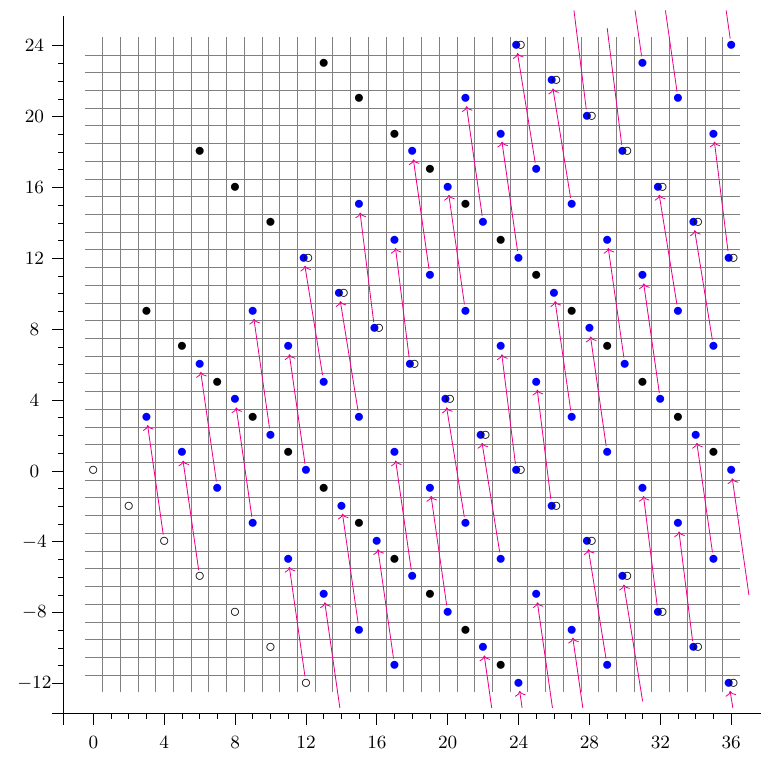}
\caption{The \(E_7\)-page of the \(a_\lambda\)-localized slice spectral sequence of \(a_\lambda^{-1}\BPCfourtwotwo\).}
\label{fig:BP22C4E7}
\end{figure}
\end{center}

\begin{center}
\begin{figure}
\includegraphics[width=\textwidth]{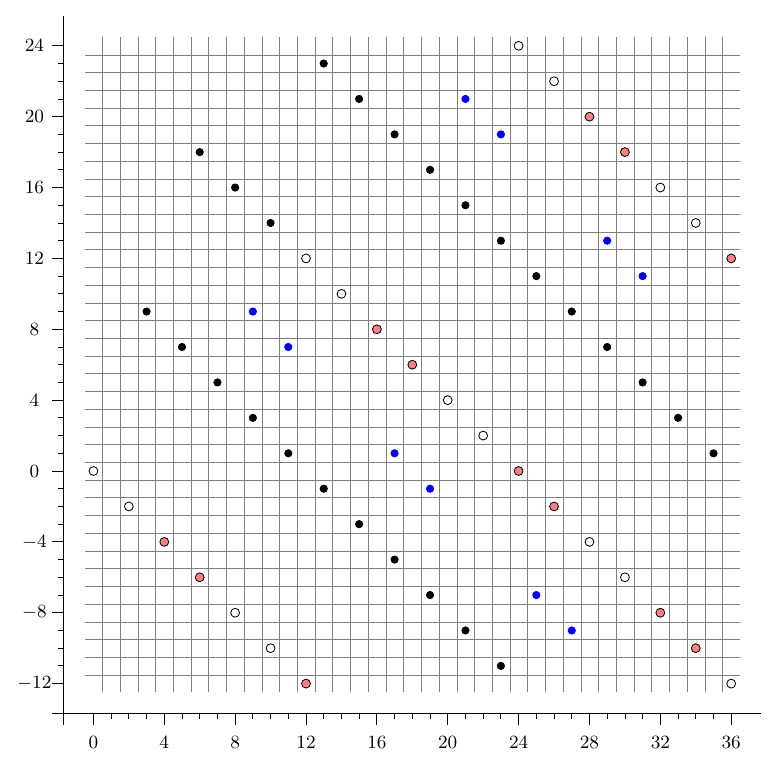}
\caption{The $E_8$-page of the $a_\lambda$-localized slice spectral sequence of $a_\lambda^{-1}\BPCfourtwotwo$.}
\label{fig:C4E8page}
\end{figure}
\end{center}

\subsection{Strategy for computing higher differentials}\label{sec:strategy}
Before computing the higher differentials (those of lengths $\geq 13$), we describe our strategy.  There are two classes that will be very important for our computation.  They are the classes 
\begin{align*}
    \alpha &=\dthree^8u_{24\sigma}a_{24\lambda} \quad \text{at $(48, 48)$}, \\
b^{32} &= u_{32\lambda}/a_{32\lambda} \quad \text{at $(64,-64)$.}
\end{align*}
We will use the following crucial facts about the spectral sequence of \(a_{\lambda}^{-1}\MUCfour\). These results are the only significant computational inputs from the localized slice spectral sequence for \(a_{\lambda}^{-1}\MUCfour\) and the much harder computations of \cite{HSWX}.

\begin{proposition}
The class $\alpha = \dthree^8u_{24\sigma}a_{24\lambda}$ is a permanent cycle in the localized slice spectral sequence for $a_{\lambda}^{-1}\MUCfour$. Consequently, the differentials are linear with respect to multiplication by $\alpha$.
\end{proposition}

\begin{proof}
This is a direct consequence of the Slice Differential Theorem of Hill--Hopkins--Ravenel \cite[Theorem~9.9]{HHR}.  More precisely, all the differentials in the region on or above the line of slope 1 in the $C_4$ localized slice spectral sequence for $a_{\lambda}^{-1}\MUCfour$ can be completely computed.
\end{proof}

\begin{proposition}
In the localized slice spectral sequence for $a_\lambda^{-1}\MUCfour$, we have the following facts:
\begin{enumerate}
\item The class $b^8 = u_{8\lambda}/a_{8\lambda}$ supports a $d_{15}$-differential.
\item The class $b^{16} = u_{16\lambda}/a_{16\lambda}$ supports a $d_{31}$-differential. 
\item The class $b^{32} = u_{32\lambda}/a_{32\lambda}$ is a $d_{61}$-cycle. 
\end{enumerate}
\end{proposition}
\begin{proof}
All the claims are direct consequences of the computations for the $C_4$ slice spectral sequence of $\BPCfour\langle 2 \rangle$ (see Table 1 in \cite{HSWX}).  We will elaborate on each of them below. 

For (1), the restriction of the class $u_{8\lambda}$ is $u_{16\sigma_2}$, which supports the $d_{31}$-differential 
\[d_{31}(u_{16\sigma_2}) = \bar{v}_4^{C_2}a_{31\sigma_2}\]
in the $C_2$-spectral sequence of $\MUCfour$.  Therefore by naturality of the restriction map, the class $u_{8\lambda}$ must support a differential of length at most 31.  By a stem-wise computation, it is impossible for the class $u_{8\lambda}$ to support a differential of length $\leq 13$, as there are no possible targets.  We will not carry out the details of this computation here, as \cite[Section~7.3]{HSWX} already contains the arguments and results to show that $u_{8\lambda}$ supports a $d_{15}$-differential in the slice spectral sequence of $\MUCfour$.  It follows that the class $\frac{u_{8\lambda}}{a_{8\lambda}}$ must also support a $d_{15}$-differential in the $a_\lambda$-localized slice spectral sequence of $a_\lambda^{-1}\MUCfour$.  This proves (1). 

For (2), the restriction of the class $u_{16\lambda}$ is $u_{32\sigma_2}$, which supports the $d_{63}$-differential 
\[d_{63}(u_{32\sigma_2}) = \bar{v}_5^{C_2}a_{63\sigma_2}\]
in the $C_2$-spectral sequence of $\MUCfour$.  Therefore by naturality of the restriction map, the class $u_{16\lambda}$ must support a differential of length at most 63.  By a stem-wise computation, it is impossible for the class $u_{16\lambda}$ to support a differential of length $\leq 19$, as there are no possible targets.  Again, we do not write down the details here, as the computation in \cite{HSWX} already shows that $u_{16\lambda}$ supports a $d_{31}$-differential in the slice spectral sequence of $\MUCfour$ (see the discussion in Section~11.2 and the chart in Section~13 of \cite{HSWX}).  It follows that the class $\frac{u_{16\lambda}}{a_{16\lambda}}$ must also support a $d_{31}$-differential in the $a_\lambda$-localized slice spectral sequence of $a_\lambda^{-1}\MUCfour$.  This proves (2).  

For (3), it is a consequence of the computation of $a_\lambda^{-1}\BPCfourtwotwo$ that $u_{32\lambda}$ is a $d_{61}$-cycle in the slice spectral sequence of $\MUCfour$ (in fact, it can be shown that it supports the $d_{63}$-differential $d_{63}(u_{32\lambda}) = N(\bar{v}_5a_{63\sigma_2})u_{16\lambda}$, but we will not prove it here).  
\end{proof} 

Multiplication by the classes $\alpha$ and $b^{32}$ give the spectral sequence a large amount of structure which we will exploit in our computation. We describe this below, focusing on each class at a time.

\subsubsection{Multiplication by \texorpdfstring{$\alpha$}{alpha}}
The class $\alpha$ is extremely important for this computation. 
A key consequence of the behavior we describe here 
is that this allows us to compute differentials out of order, flipping back and forth between different pages of the spectral sequence without loosing the thread of its story. We make this precise now, starting with the following straightforward lemma.

\begin{lemma}\label{lem:alphastuff}
In the $C_4$-localized slice spectral sequence of \(a_\lambda^{-1}\BPCfourtwotwo\), on the $E_{13}$-page, we have: 
\begin{enumerate}
    \item Multiplication by $\alpha$ is injective.
    \item\label{lem:alphastuff3} Let $x$ be a class in bidegree $(t-s,s)$. If
\[s-48 \geq -(t-s-48) \quad \text{and} \quad s-48 \leq 3(t-s-48)\]
then $x$ is $\alpha$ divisible. 
\end{enumerate}
See Figure~\ref{fig:StrategyCartoon}.
\end{lemma}

The key result is then the following:
\begin{proposition}\label{prop:alphafreeres}
Let $r\geq 13$.
Suppose that $d_r(x)=y$ is a non-trivial differential on the $E_r$-page of the localized slice spectral sequence of \(a_\lambda^{-1}\BPCfourtwotwo\). Then $y$ is $\alpha$-free in the sense that no multiple $\alpha^iy$ is zero at $E_r$. Consequently, $x$ is also $\alpha$-free.
\end{proposition}
\begin{proof}
First, note that if $y$ is $\alpha$-free, the linearity of the differentials with respect to multiplication by $\alpha$ implies that $x$ must also be $\alpha$-free.

We prove that $y$ is $\alpha$-free by induction on $r$. If $r=13$, the claim follows immediately from the fact that all classes are $\alpha$-free at $E_{13}$ (Lemma~\ref{lem:alphastuff}). Suppose that the claim holds for all $r'<r$, that $d_r(x)=y$ is a non-trivial differential and that $y$ is $\alpha$-torsion. Then there exists $i>0$, $r'<r$ and $z$ such that 
\[d_{r'}(z) = \alpha^iy.\] Choose a minimal such $i$, so that $\alpha^{i-1}y$ is non-zero at $E_r$. A comparison of degrees then implies that the bidegree of $z$ satisfies the conditions of Lemma~\ref{lem:alphastuff}~\eqref{lem:alphastuff3}, so that
$z$ is $\alpha$-divisible. 
It cannot be the case that $d_{r'}(z/\alpha) = \alpha^{i-1}y$ since this contradicts the minimality of $i$. So we must have that $d_{r''}(z/\alpha)=v\neq 0$ for some $r''<r'<r$. But by the induction hypothesis, $v$ is then $\alpha$-free which means that $d_{r''}(z)=\alpha v\neq 0$, which is also a contradiction.
\end{proof}

\begin{remark}\label{rem:vanishingpreview}
We will show in Section~\ref{subsec:VanishingThm} that $\alpha$ is killed by a $d_{61}$-differential. This will imply that for any permanent cycle $x$, the class $\alpha x$ must be hit by a differential of length at most $61$. 
\end{remark}

We now explain the upshot of Proposition~\ref{prop:alphafreeres}. 
Given any class $y$ at $E_{13}$, there is a unique class $x$ which is not $\alpha$-divisible (so is in the complement of the region of Lemma~\ref{lem:alphastuff}~\eqref{lem:alphastuff3}) and such that $y = \alpha^i x$ for some $i\geq 0$. We say that $x$ generates an $\alpha$-free family, where the family is the collection of classes $\{\alpha^ix \mid i\geq 0\}$.

Now, Proposition~\ref{prop:alphafreeres} implies that $\alpha$-free families come in pairs: one family in the pair, generated by $x$ say, supports differentials which truncate the second family in the pair, generated by $y$ say. (In fact, by Remark~\ref{rem:vanishingpreview}, the differentials must be of the form $d_r(\alpha^i x)=\alpha^{i+1} y$.) All classes in the $\alpha$-free family generated by $x$ are then gone at the $E_{r+1}$-page, having supported a no-trivial differential. The class $y$ is now $\alpha$-torsion, and by Proposition~\ref{prop:alphafreeres}, it cannot support any further differential. This allows us to discard $y$ from the rest of the computation, making the spectral sequence effectively sparser. Furthermore, we may now run differentials out of order if we find a unique possibility for pairing  $\alpha$-free families, even if this is through very long differentials.

\begin{remark}
This is in fact a common behavior for spectral sequences. For example, what we have here is very similar to the situation explained in a certain elliptic spectral sequence \cite[Section 6]{BeaudryBobkovaPhamXu}, where there the role of $\alpha$ is played by the class $\bar{\kappa}$. 
\end{remark}

\subsubsection{Multiplication by \texorpdfstring{$b^{32}$}{b to the 32}}
Multiplication by the permanent cycle $b^{32}$ acts as a periodicity generator for most of the spectral sequence. More precisely, we have:

\begin{lemma}\label{lem:bmultiplicationinjective}
Let $r\geq 13$.
Multiplication by $b^{32}$ is injective on the $E_r$-page  for classes on or below the line of slope 1 through the origin.  
\end{lemma}
It follows that if a differential has both source and target on or below the line of slope 1 through the origin, then $d_r(x) = y$ occurs if and only if $d_r(b^{32}x) = b^{32}y$ occurs. 
Differentials whose source and target are above the line of slope 1 through the origin are determined by the $a_{\sigma}$-localized spectral sequence. Some differentials, fall in neither category in the sense that they cross the line of slope one. That is, the source is on or below the line of slope one and the target is above. For these differentials, the target may be $b^{32}$-torsion while the source is not.

As one does the computation however, one sees that the target of such differentials have bidegree $(t-s,s)$ such that
\[s \leq (t-s)+14.\]
This can be seen from the $d_{13}$-differentials that are obtained from the $a_{\sigma}$-localized spectral sequence and the structure of the $E_{14}$-page. Since the longest differential is a $d_{61}$ and classes are concentrated in degrees with $t-s$ even, classes strictly below the line 
\[s=(t-s)-60\]
cannot support differentials that cross the line of slope 1 through the origin.
Therefore, to completely determine the differentials of the spectral sequence using $b^{32}$-linearity and $\alpha$-linearity, it is sufficient to determine:
\begin{itemize}
    \item The $d_{13}$'s with source on or above the line of slope 1 through the origin, all obtained from the $a_{\sigma}$-localized spectral sequence.
    \item The differentials on classes with source of bidegree $(t-s,s)$ where is in the rectangular region:
    \begin{align*}
 s&\leq (t-s) &
  s& \geq -(t-s) \\
  s &\geq  (t-s) -188 &
 s&< -(t-s)+96 
    \end{align*}
\end{itemize}
   This region is larger than what is needed in practice, but the goal of this discussion is simply to illustrate the strategy and make a rough estimate on what differentials need to be determined. As we go through the computation, we learn that the region that determines all differentials is in fact smaller but, a priori, this is not clear.

\subsubsection{Summary}
To summarize, we just have to focus on the classes in the shaded rectangular region of Figure~\ref{fig:StrategyCartoon} which is the union of a cone and a rectangle.  Once we have figured out the fate of all the classes in this region, we can propagate by the classes $\alpha$ and $b^{32}$ to obtain the rest of the differentials. Furthermore, once an $\alpha$-multiple of a class gets truncated by a differential, that class can no longer support differentials and can be disregarded from future arguments.

\begin{center}
\begin{figure}
\includegraphics[width=\textwidth]{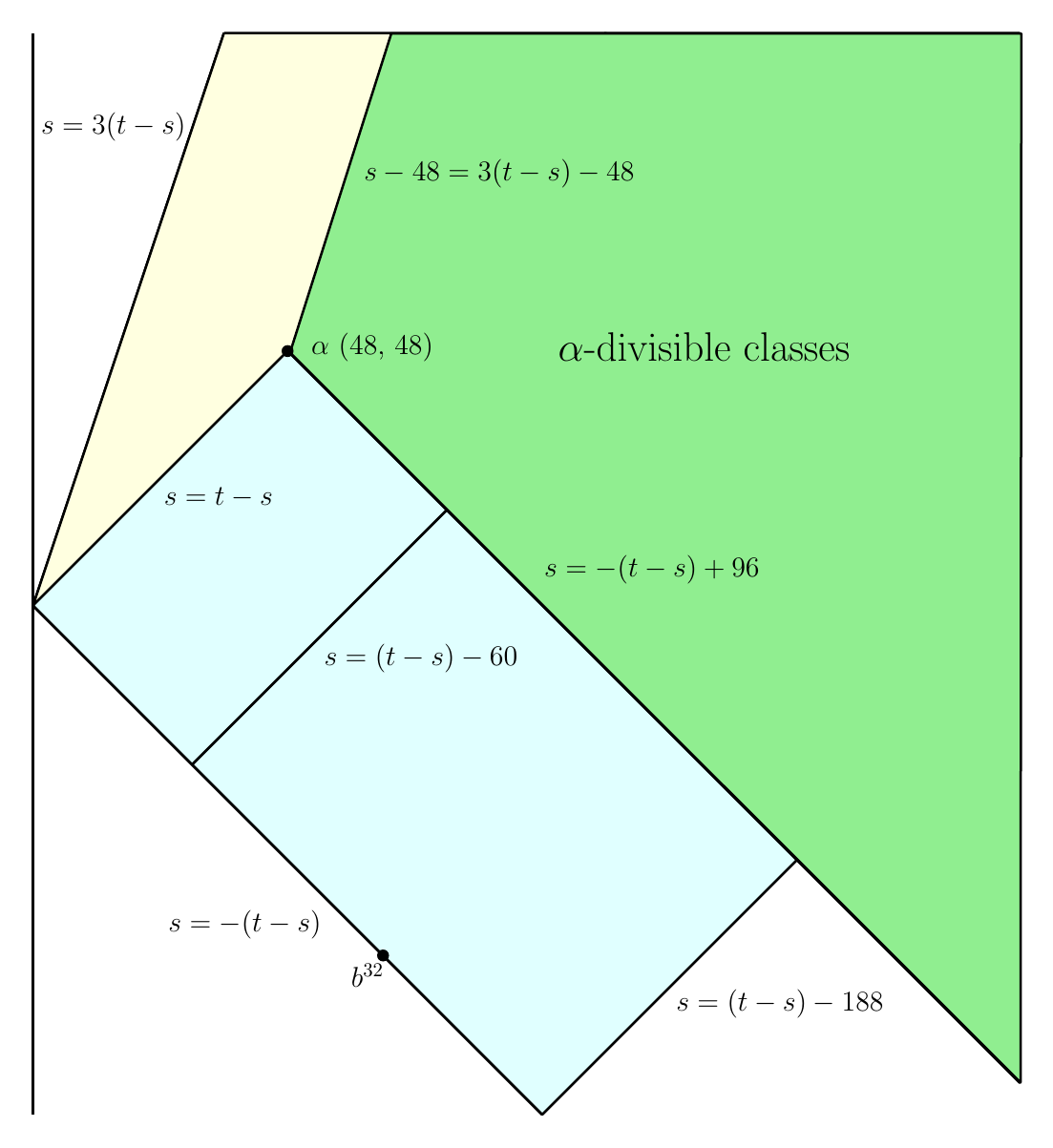}
\caption{Features of the \(a_{\lambda}\)-localized slice spectral sequence of \(a_{\lambda}^{-1}\BPCfourtwotwo\)}
\label{fig:StrategyCartoon}
\end{figure}
\end{center}

\subsection{Differentials of length at least 13}
\subsubsection{\texorpdfstring{$d_{13}$}{d-13}-differentials} \label{subsec:d13Diff}
By degree reasons, the next possible differentials are the $d_{13}$-differentials.  

The differentials on or above the line of slope 1 are all obtained by computing the \(a_{\sigma}\)-localized spectral sequence, as explained in Section~\ref{sec:gen}. This spectral sequence is depicted in Figure~\ref{fig:sliceSSBPCfourmm2}.  The differentials are summarized in the following proposition. \\

\begin{proposition}
The $d_{13}$-differentials that are on or above the line of slope 1 are generated by
\begin{enumerate}
\item $d_{13}\left((\dtwo^2u_{6\sigma}a_{6\lambda})^{i+4k}\right) = \dtwo^3 u_{2\sigma}a_{9\lambda}a_{7\sigma} \cdot (\dtwo^2u_{6\sigma}a_{6\lambda})^{i-1 + 4k}$ \\
$i = 1, 2$, $k \geq 0$. 
\item $d_{13}\left(\dtwo^2u_{4\sigma}a_{6\lambda}a_{2\sigma} \cdot (\dtwo^2u_{6\sigma}a_{6\lambda})^{i + 4k}\right) = \dtwo^3 a_{9\lambda}a_{9\sigma}\cdot (\dtwo^2u_{6\sigma}a_{6\lambda})^{i + 4k}$ \\
$i = 0, 3$, $k \geq 0$. 
\item $d_{13}\left(\dtwo^5u_{14\sigma}a_{15\lambda}a_\sigma \cdot (\dtwo^2u_{6\sigma}a_{6\lambda})^{i + 4k}\right) = \dtwo^6u_{10\sigma}a_{18\lambda}a_{8\sigma}\cdot (\dtwo^2u_{6\sigma}a_{6\lambda})^{i + 4k}$ \\
$i = 0,1$, $k \geq 0$. 

\end{enumerate}
\end{proposition}

To prove the $d_{13}$-differentials that are under the line of slope 1, we would like to first point out that the class \(\dtwo u_{\lambda} u_{2\sigma} a_{2\lambda}a_\sigma\) in bidegree \((7,5)\) is a permanent-cycle by degree reasons.

The class \(\dthree^4 u_{12 \sigma} a_{12\lambda}\) in bidegree \((24, 24)\) will also be important.  By the Hill--Hopkins--Ravenel Slice differential theorem \cite[Theorem~9.9]{HHR}.  This class supports the $d_{13}$-differential 
\[d_{13}(\dthree^4 u_{12 \sigma} a_{12\lambda}) = \dthree^5 u_{8\sigma}a_{15\lambda}a_{7\sigma}\]
in the slice spectral sequence of $\BPCfour$.  By naturality, this differential also appears in the slice spectral sequence of $\BPCfourtwotwo$.  When applying the Leibniz rule, the class \(\dthree^4 u_{12 \sigma} a_{12\lambda}\) $(24, 24)$ acts as if it is a $d_{13}$-cycle for differentials whose sources are below the line of slope 1.  More specifically, the target of the $d_{13}$-differential on this class multiplied with the source of another \(d_{13}\)-differential below the line of slope 1 is always 0.

\begin{proposition}
We have the following $d_{13}$-differentials: 
\begin{enumerate}
\item $d_{13}\left(b^4 \right) = \dtwo u_{\lambda} u_{2\sigma} a_{2\lambda}a_\sigma$ 

\item $d_{13}\left(b^5 \right) = \dtwo u_{2\lambda} u_{2\sigma} a_{\lambda}a_\sigma$ 

\item $d_{13}\left(\dtwo u_{4\lambda} u_{2\sigma} a_{-\lambda}a_\sigma\right) = 2 \dtwo^2 u_{6\sigma}a_{6\lambda}$ 
\item $d_{13}\left(\dtwo u_{4\lambda} u_{2\sigma} a_{-\lambda}a_\sigma \cdot b\right) = 2 \dtwo^2 u_{6\sigma}a_{6\lambda} \cdot b$ 
\item $d_{13}\left(\dtwo u_{4\lambda} u_{2\sigma} a_{-\lambda}a_\sigma \cdot b^2\right) = 2 \dtwo^2 u_{6\sigma}a_{6\lambda} \cdot b^2$ 
\item $d_{13}\left(\dtwo u_{4\lambda} u_{2\sigma} a_{-\lambda}a_\sigma \cdot b^3\right) = 2 \dtwo^2 u_{6\sigma}a_{6\lambda} \cdot b^3$ 
\end{enumerate}
\end{proposition}

\begin{proof}
To prove (1), we will first prove the differential 
\[d_{13}\left(b^{12} \right) = \dtwo u_{9\lambda} u_{2\sigma} a_{-6\lambda}a_\sigma \,\,\, (d_{13}(24,-24) = (23, -11))\]
The source of this differential restricts to a class in the $C_2$-spectral sequence that supports a $d_{31}$-differential. 
By naturality and degree reasons, we must have the $d_{13}$-differential claimed above.  Applying Leibniz with the class 
$b^8$ 
in degree
$(16, -16)$ proves (1).  

The source of (2) restricts to a class that supports a $d_{19}$-differential in the $C_2$-spectral sequence.  Therefore the source class must support either a $d_{13}$- or a $d_{19}$-differential.  By naturality, it cannot be a $d_{19}$-differential because the target does not restrict to the target of the $d_{19}$-differential in the $C_2$-spectral sequence.

The targets of (3) is in the image of the transfer.  The preimage is killed by a $d_{19}$-differential.  Therefore by naturality and degree reasons, we have the $d_{13}$-differential claimed in (3).

Differentials (4) and (5) are obtained by applying the Leibniz rule using the class $\dtwo u_{\lambda} u_{2\sigma} a_{2\lambda}a_\sigma$ $(7,5)$ with differentials (1) and (2) (and also using the gold relation).  

It remains to prove differential (6). Consider the class $\dthree u_{7\lambda} u_{2\sigma}a_{-4\lambda}$.  The restriction of this class is $\bar{t}_2\gamma \bar{t}_2 u_{14\sigma_2} a_{-8\sigma_2}$, which supports a $d_7$-differential in the $C_2$-spectral sequence.  Therefore, the class $\dthree u_{7\lambda} u_{2\sigma}a_{-4\lambda}$ also supports a $d_7$-differential in the $C_4$-spectral sequence.  The existence of this $d_7$-differential shows that there is an exotic restriction of filtration 6 for the class $\dthree u_{7\lambda} u_{2\sigma}a_{-4\lambda}a_\sigma$ $(19, -7)$.  It must have nonzero restriction, restricting to the class $\bar{t}_2^3 u_{10\sigma_2}a_{-\sigma_2}$ $(19, -1)$ after the $E_7$-page.  

Since the class $\bar{t}_2^3 u_{10\sigma_2}a_{-\sigma_2}$ $(19, -1)$ supports a $d_{19}$-differential in the $C_2$-spectral sequence, the class $\dthree u_{7\lambda} u_{2\sigma}a_{-4\lambda}a_\sigma$ $(19, -7)$ cannot survive past the $E_{25}$-page.  The only possibility is for it to support the $d_{13}$-differential claimed by (6).  
\end{proof}

The same proof for differentials (1)--(6) can be used to prove six more differentials that are obtained by multiplying both the source and the target of each differential by $(12, 12)$: $\dtwo u_{6\sigma} a_{6\lambda}$ (note that we can't just directly propagate by this class using the Leibniz rule because it supports a $d_5$-differential in $\SliceSS(\BPCfour)$). 

\begin{proposition}
The following classes are $d_{13}$-cycles: 
\begin{enumerate}
\item $2b^6$ $(12, -12)$;
\item $2\dthree^2 u_{6\lambda} u_{6\sigma}$ $(24,0)$;
\item $\dthree^3 u_{8\lambda}u_{8\sigma}a_\lambda a_\sigma$ $(33,3)$. 
\end{enumerate}
\end{proposition}
\begin{proof}
For (1), the class $2\dthree^4 u_{6\lambda}u_{12\sigma} a_{6\lambda}$ $(36, 12)$ is a $d_{13}$-cycle by using the class $\dtwo u_{\lambda} u_{2\sigma} a_{2\lambda}a_\sigma$ $(7,5)$ to apply the Leibniz rule to the $d_{13}$-differential 
\[d_{13}(\dthree^3 u_{6\lambda} u_{8\sigma} a_{3\lambda}a_\sigma) = 2\dthree^4 u_{2\lambda} u_{12\sigma}a_{10\lambda} \quad (d_{13}(29, 7) = (28, 20)). \]
Therefore, by Leibniz with the class $\dthree^4 u_{12 \sigma} a_{12\lambda}$ $(24, 24)$, the class $2b^6$ $(12, -12)$ is also a $d_{13}$-cycle.  

(2) is proven by the exact same method, by using the class $\dtwo u_{\lambda} u_{2\sigma} a_{2\lambda}a_\sigma$ $(7,5)$ to apply the Leibniz rule to the $d_{13}$-differential 
\[d_{13}(\dtwo u_{6\lambda} u_{2\sigma} a_{-3\lambda}a_\sigma) = 2 \dtwo^2 u_{2\lambda}u_{6\sigma}a_{4\lambda} \quad (d_{13}(17, -5) = (16, 8)).\]

For (3), the class $\dthree u_{8\lambda}u_{8\sigma}a_\lambda a_\sigma$ is a $d_{13}$-cycle in $\SliceSS(\BPCfour)$.  Therefore by naturality it is a $d_{13}$-cycle in $\SliceSS(\BPCfourtwotwo)$. 
\end{proof}

Now, we can propagate all the differentials by the classes $\dtwo^2 u_{12\sigma} a_{12\lambda}$ $(24, 24)$ and $b^8$ $(16,-16)$.  The $d_{13}$-differentials under the line of slope 1 are summarized in the following proposition. 

\begin{proposition}
The $d_{13}$-differentials that are under the line of slope 1 are
\begin{enumerate}
\item $d_{13}\left(b^{4+i+8j}\cdot(\dtwo^2 u_{6\sigma}a_{6\lambda})^k \right) = \dtwo u_{\lambda} u_{2\sigma} a_{2\lambda}a_\sigma \cdot b^{i+8j}(\dtwo^2 u_{6\sigma}a_{6\lambda})^k$, \\$0 \leq i \leq 1$, $j, k \geq 0$. 

\item $d_{13}\left(\dtwo u_{4\lambda} u_{2\sigma} a_{-\lambda}a_\sigma \cdot b^{i+8j}(\dtwo^2 u_{6\sigma}a_{6\lambda})^k\right) = 2 \dtwo^2 u_{6\sigma}a_{6\lambda}\cdot b^{i+8j}(\dtwo^2 u_{6\sigma}a_{6\lambda})^k$, \\$0 \leq i \leq 3$, $j, k \geq 0$.  
\end{enumerate}
They are shown in Figure~\ref{fig:C4E13page}.  
\end{proposition}

\begin{remark}
We will see that these are the last non-trivial \(d_{13}\) differentials. However, at this point in the computation, there are possibilities for other non-trivial \(d_{13}\) differentials. Later, (in Lemmas~\ref{lem:d31d13notExist} and
\ref{lem:d37d13notExist}) we will show that these do not occur.
\end{remark}

\begin{center}
\begin{figure}
\includegraphics[width=\textwidth]{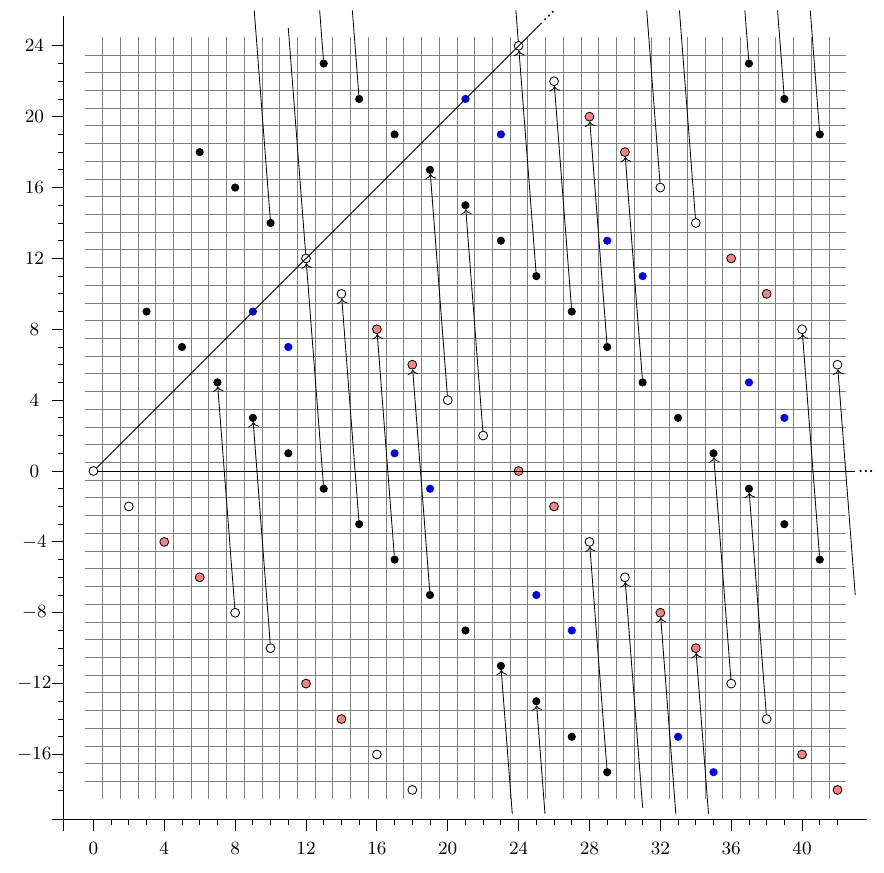}
\caption{The $E_{13}$-page of the $a_\lambda$-localized slice spectral sequence of $a_\lambda^{-1}\BPCfourtwotwo$.}
\label{fig:C4E13page}
\end{figure}
\end{center}

\subsubsection{\texorpdfstring{$d_{19}$}{d-19}-differentials}
\begin{proposition}\label{prop:BPC4twotwod19}
The following $d_{19}$-differentials exist: 
\begin{enumerate}
\item $d_{19}\left(2b^5 \right) = tr(\bar{t}_2^3 a_{9\sigma_2})$ ($d_{19}(10, -10) = (9,9)$)
\item $d_{19}\left(2b^6 \right) = tr(\bar{t}_2^3 u_{2\sigma_2}a_{7\sigma_2})$ ($d_{19}(12, -12) = (11,7)$)
\item $d_{19}\left(b^9 \right) = tr(\bar{t}_2^3 u_{8\sigma_2}a_{\sigma_2})$ ($d_{19}(18, -18) = (17,1)$)
\item $d_{19}\left(2b^{13} \right) = tr(\bar{t}_2^3 u_{16\sigma_2}a_{-7\sigma_2})$ ($d_{19}(26, -26) = (25,-7)$)
\end{enumerate}
\end{proposition}
\begin{proof}
Differential (1) is obtained by applying transfer to the $d_{19}$-differential 
\[d_{19}\left(\frac{u_{10\sigma_2}}{a_{10\sigma_2}} \right) = \bar{t}_2^3 a_{9\sigma_2} \]
in the $C_2$-spectral sequence. 

For differentials (2) and (3), the classes $\bar{t}_2^3 u_{2\sigma_2}a_{7\sigma_2}$ and $\bar{t}_2^3 u_{8\sigma_2}a_{\sigma_2}$ are killed by $d_{31}$-differentials in the $C_2$-spectral sequence.  Therefore their images under the transfer map must also be killed by differentials of lengths at most 31.  The only possibilities are the differentials claimed.  

Differential (4) is obtained by applying the transfer to the $d_{19}$-differential 
\[d_{19}\left(\frac{u_{26\sigma_2}}{a_{26\sigma_2}} \right) = \bar{t}_2^3 u_{16\sigma_2}a_{-7\sigma_2} \]
in the $C_2$-spectral sequence. 
\end{proof}

The same arguments in the proof above can be used to prove twelve more $d_{19}$-differentials, obtained by multiplying the four $d_{19}$-differentials in Proposition~\ref{prop:BPC4twotwod19} by $\dthree^2 u_{6\lambda}a_{6\lambda}$ $(12, 12)$, $\dthree^4 u_{12\lambda}a_{12\lambda}$ $(24, 24)$, and $\dthree^6 u_{18\lambda}a_{18\lambda}$ $(36, 36)$.

\begin{proposition}
The following $d_{19}$-differentials exist:
\begin{enumerate}
\item $d_{19}(2b^{14}) = tr(\bar{t}_2^3 u_{18\sigma_2}a_{-9\sigma_2})$ $(d_{19}(28, -28) = (27, -9))$
\item $d_{19}\left(2b^{14} \cdot (\dtwo^2u_{6\sigma}a_{6\lambda})\right) = tr(\bar{t}_2^7 u_{18\sigma_2}a_{3\sigma_2})$ $(d_{19}(40, -16) = (39, 3))$
\item $d_{19}\left(2b^{14} \cdot (\dtwo^2u_{6\sigma}a_{6\lambda})^2\right) = tr(\bar{t}_2^{11} u_{18\sigma_2}a_{15\sigma_2})$ $(d_{19}(52, -4) = (51, 15))$
\item $d_{19}\left(2b^{14} \cdot (\dtwo^2u_{6\sigma}a_{6\lambda})^3\right) = tr(\bar{t}_2^{15} u_{18\sigma_2}a_{27\sigma_2})$ $(d_{19}(64, 8) = (63, 27))$
\end{enumerate}
\end{proposition}
\begin{proof}
We will prove differential (1) first.  Consider the class $tr(\bar{t}_2^{19}u_{18\sigma_2}a_{39\sigma_2})$ $(75, 39)$.  This class must die on or before the $E_{61}$-page.  There are three possible ways for this class to die.  It can support a $d_{31}$-differential hitting the class $\dtwo^{12}u_{\lambda}u_{36\sigma}a_{35\lambda}$ $(74, 70)$; it can be the target of a $d_{19}$-differential from the class $2\dtwo^{8}u_{14\lambda}u_{24\sigma}a_{10\lambda}$ $(76, 20)$; or it can be the target of a $d_{43}$-differential from the class $2\dtwo^{6}u_{20\lambda}u_{18\sigma}a_{-2\lambda}$ $(76, -4)$.  

It is impossible for this class to support a $d_{31}$-differential because it is the transfer of a class that supports a $d_{31}$-differential in the $C_2$-slice spectral sequence, and the target does not transfer to the class $\dtwo^{12}u_{\lambda}u_{36\sigma}a_{35\lambda}$ $(74, 70)$.  

The $d_{43}$-differential also cannot happen because the class $2\dtwo^{6}u_{20\lambda}u_{18\sigma}a_{-2\lambda}$ $(76, -4)$ is the transfer of $\bar{t}_2^{12}u_{40\sigma_2}a_{-4\sigma_2}$, which is the target of a $d_{31}$-differential in the $C_2$-spectral sequence.  Therefore it must be killed by a differential of length at most 31.  

It follows that the $d_{19}$-differential 
\[d_{19}(2\dtwo^{8}u_{14\lambda}u_{24\sigma}a_{10\lambda}) = tr(\bar{t}_2^{19}u_{18\sigma_2}a_{39\sigma_2}) \,\,\, (d_{19}(76, 20) = (75, 39))\]
exists.  Applying Leibniz with respect to the class $ \dtwo^4 u_{24\sigma} a_{24\lambda}$ $(48, 48)$ proves (1).  

Differentials (2), (3), (4) are proven by the exact same method.  
\end{proof}

Now, we can propagate the $d_{19}$-differentials that we have proven by the classes $b^{16}$ $(32,-32)$ and $\dtwo^4 u_{24\sigma} a_{24\lambda}$ $(48, 48)$ to obtain the rest of the $d_{19}$-differentials. 

\begin{proposition}
The $d_{19}$-differentials are
\begin{enumerate}
\item $d_{19}\left(2b^{i + 8j} \cdot (\dtwo^2u_{6\sigma}a_{6\lambda})^{k} \right) = tr(\bar{t}_2^3 a_{9\sigma_2}) \cdot b^{i-5+8j}(\dtwo^2u_{6\sigma}a_{6\lambda})^{k}$ \\
$i = 5,6$, $j, k \geq 0$
\item $d_{19}\left(b^{9+16i} \cdot (\dtwo^2u_{6\sigma}a_{6\lambda})^{k}\right) = tr(\bar{t}_2^3 u_{8\sigma_2}a_{\sigma_2})\cdot b^{16i} (\dtwo^2u_{6\sigma}a_{6\lambda})^{k}$\\
$i, k \geq 0$.
\end{enumerate}
They are shown in Figure~\ref{fig:C4E19page}.  
\end{proposition}

\begin{center}
\begin{figure}
\includegraphics[width=\textwidth]{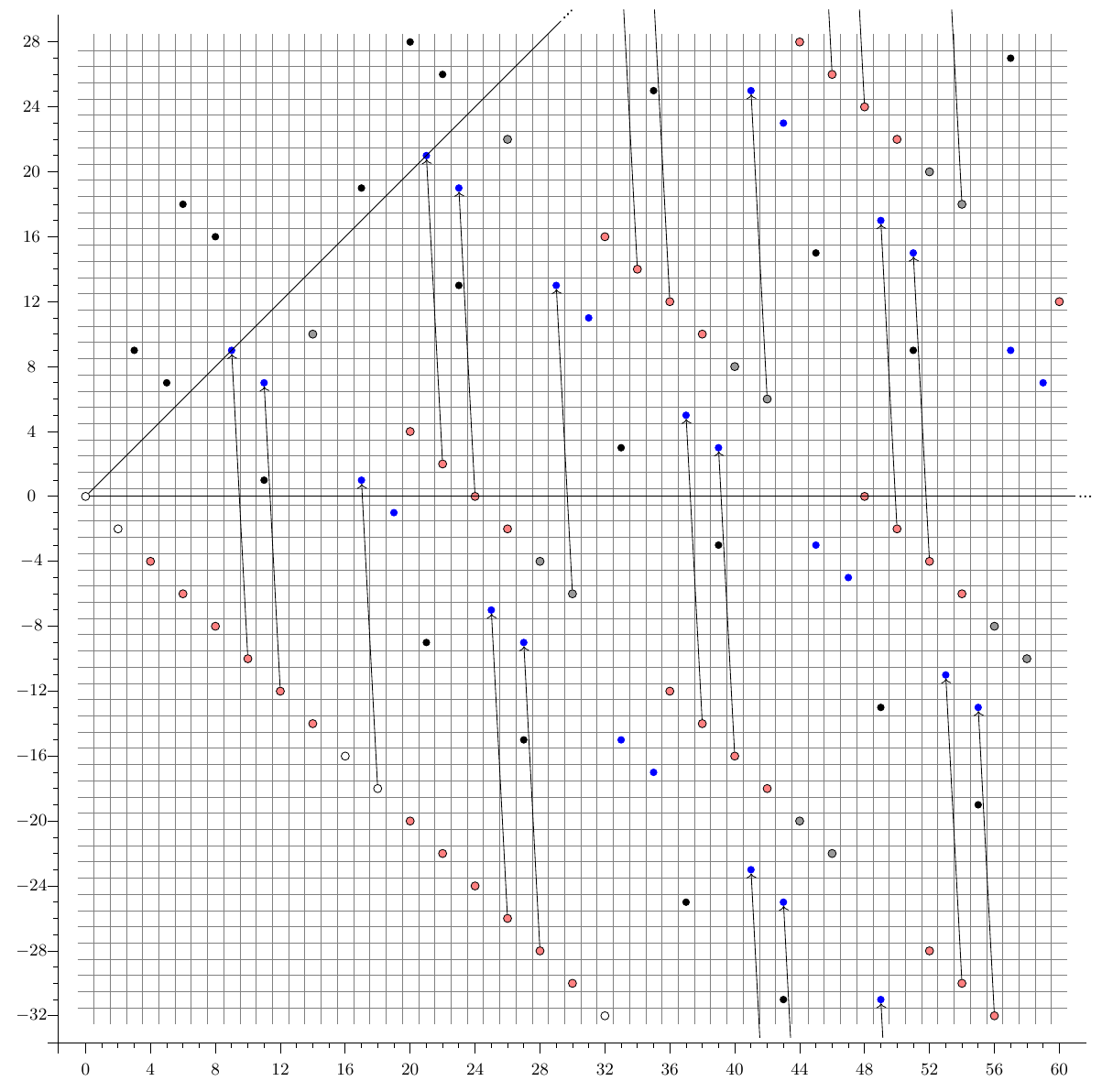}
\caption{The $E_{19}$-page of the $a_\lambda$-localized slice spectral sequence of $a_\lambda^{-1}\BPCfourtwotwo$.}
\label{fig:C4E19page}
\end{figure}
\end{center}

\subsubsection{The vanishing theorem}\label{subsec:VanishingThm}

\begin{theorem}[Vanishing Theorem] \label{thm:VanishingThm}
In the $a_\lambda$-localized slice spectral sequence for $a_\lambda^{-1}\BPCfourtwotwo$, we have the $d_{61}$-differential 
\[d_{61}(\dtwo^3 u_{16\lambda} u_{8\sigma} a_{-7\lambda} a_{\sigma}) = \dtwo^8u_{24\sigma} a_{24\lambda} \,\,\, (d_{61}(49,-13) = (48,48)). \]
Furthermore, any class of the form $(\dtwo^8u_{24\sigma} a_{24\lambda}) \cdot x$ must die on or before the $E_{61}$-page. 
\end{theorem}
\begin{proof}
In the $a_\lambda$-localized slice spectral sequence of $a_\lambda^{-1}\BPCfour$, the class $N(\bar{v}_4)u_{16\sigma}a_{31\lambda}$ must die on or before the $E_{61}$-page because it is the target of the predicted $d_{61}$-differential
\[d_{61}\left(u_{16\lambda} a_\sigma \right) = N(\bar{v}_4)u_{16\sigma} a_{31\lambda},\]
obtained by norming up the $d_{31}$-differential $d_{31}(u_{16\sigma_2}) = \bar{v}_4 a_{31\sigma_2}$ in the $C_2$-spectral sequence.  Therefore, if we multiply the target by $\dthree^{11}u_{32\sigma}a_{32\lambda}$, the class $\dthree^{11} N(\bar{v}_4) u_{48\sigma}u_{48\lambda}$ $(96, 96)$ must die on or before the $E_{61}$-page. 

Under the map
\[a_\lambda^{-1}\SliceSS(\BPCfour) \longrightarrow a_\lambda^{-1}\SliceSS(\BPCfourtwotwo),\]
the class $\dthree^{11} N(\bar{v}_4) u_{48\sigma}u_{48\lambda}$ is sent to $\dthree^{16}u_{48\sigma}u_{48\lambda}$ $(96, 96)$.  By naturality and degree reasons, the only possibility that this class can die on or before the $E_{61}$-page is for it to be killed by a $d_{61}$-differential.  This implies that the original class must also be killed by a $d_{61}$-differential in the $a_\lambda$-localized slice spectral sequence of $a_\lambda^{-1}\BPCfour$.  Furthermore, by the module structure, any class in the $a_\lambda$-localized slice spectral sequence of $a_\lambda^{-1}\BPCfourtwotwo$ of the form $(\dthree^{16}u_{48\sigma}u_{48\lambda}) \cdot x$ must die on or before the $E_{61}$-page.  

The class $\dthree^8 u_{24\sigma} a_{24\lambda}$ $(48, 48)$ is also the target of a $d_{61}$-differential because after multiplying it by $\dthree^{16} u_{48\sigma} a_{48\lambda}$ $(96, 96)$, it must die on or before the $E_{61}$-page.  By degree reasons, the only possibility is for it to be killed by a $d_{61}$-differential.  Since multiplication by $\dthree^{16} u_{48\sigma} a_{48\lambda}$ $(96, 96)$ induces an injection on the $E_2$-page, and all the classes above the line of slope $(-1)$ with this class as the origin are all divisible by it, the claimed $d_{61}$-differential must occur.

Similarly, for any class of the form $(\dtwo^8u_{24\sigma} a_{24\lambda}) \cdot x$, we can multiply it by $\dthree^{16} u_{48\sigma} a_{48\lambda}$ $(96, 96)$ to deduce that the product must die on or before the $E_{61}$-page.  It follows from the same reasoning as the previous paragraph that the original class must also die on or before the $E_{61}$-page.
\end{proof}

\subsubsection{\texorpdfstring{$d_{31}$}{d-31}-differentials}
To prove the $d_{31}$-differentials, we will first prove the nonexistence of certain $d_{13}$-differentials. 
\begin{lemma} \label{lem:d31d13notExist}
At the $E_{13}$-page, we have
\begin{enumerate}
\item $d_{13}(\dtwo u_{8\lambda} u_{2\sigma} a_{-5\lambda}a_\sigma) \neq 2 \dtwo^2 u_{4\lambda}u_{6\sigma}a_{2\lambda}$ $(d_{13}(21, -9) \neq (20, 4))$.
\item $d_{13}(\dtwo^3 u_{8\lambda}u_{8\sigma}a_{\lambda}a_\sigma) \neq 2\dthree^4 u_{4\lambda}u_{12\sigma}a_{8\lambda}$ $(d_{13}(33,3) \neq (32,16))$.
\end{enumerate}
\end{lemma}
\begin{proof}
Suppose (1) exists.  By applying the Leibniz rule with respect to the classes $\dthree^4 u_{12 \sigma} a_{12\lambda}$ $(24, 24)$ and $b^8$ $(16, -16)$, the $d_{13}$-differential 
\[d_{13}(\dtwo^5 u_{16\lambda} u_{14\sigma} a_{-\lambda} a_{\sigma}) = 2\dtwo^6u_{12\lambda}u_{18\sigma}a_{6\lambda} \quad(d_{13}(61,-1) = (60,12))\]
must also exist.  Consider the class $\dtwo^{9}u_{3\lambda}u_{26\sigma}a_{24\lambda}a_\sigma$ in $(59,49)$.  By Theorem~\ref{thm:VanishingThm}, this class must die on or before the $E_{61}$-page.  However, with the class $2\dtwo^6u_{12\lambda}u_{18\sigma}a_{6\lambda}$ $(60, 12)$ gone, there are no classes that could kill it or be killed by this class on or before the $E_{61}$-page.  Contradiction. 

Now, suppose (2) exists.  By applying the Leibniz rule with respect to the classes $\dthree^4 u_{12 \sigma} a_{12\lambda}$ $(24, 24)$ and $b^8$ $(16, -16)$, the $d_{13}$-differential 
\[d_{13}(\dtwo^{7}u_{16\lambda}u_{20\sigma}a_{5\lambda}a_{\sigma}) = 2\dtwo^{8}u_{12\lambda}u_{24\sigma}a_{12\lambda}\,\,\,(d_{13}(73,11) = (72,24))\]
must also exist.  Consider the class $\dtwo^{11}u_{3\lambda}u_{32\sigma}a_{30\lambda}a_{\sigma}$ $(71,61)$.  By Theorem~\ref{thm:VanishingThm}, this class must die on or before the $E_{61}$-page.  Just like the previous case, there is no class that could kill it or be killed by it on or before the $E_{61}$-page.  Contradiction.  
\end{proof}

\begin{proposition}\label{prop:BPCfourtwotwod31Differentials}
We have the following $d_{31}$-differentials: 
\begin{enumerate}
\item $d_{31}\left(tr(\ttwo^{11}u_{10\sigma_2}a_{23\sigma_2})\right) = \dtwo^{8}u_{18\sigma}a_{24\lambda}a_{6\sigma}$; 
\item $d_{31}\left(tr(\ttwo^{11}u_{10\sigma_2}a_{23\sigma_2})\cdot (\dtwo^2u_{6\sigma}a_{6\lambda})\right) = \dtwo^{8}u_{18\sigma}a_{24\lambda}a_{6\sigma}\cdot (\dtwo^2u_{6\sigma}a_{6\lambda})$; 
\item $d_{31}\left(tr(\ttwo^{11}u_{24\sigma_2}a_{9\sigma_2})\cdot (\dtwo^2u_{6\sigma}a_{6\lambda})^i\right) = 2\dtwo^{8}u_{4\lambda}u_{24\sigma}a_{20\lambda} \cdot (\dtwo^2u_{6\sigma}a_{6\lambda})^i$, \\
$0 \leq i \leq 3$;
\item $d_{31}\left(tr(\ttwo^{11}u_{24\sigma_2}a_{9\sigma_2})\cdot b^{16}(\dtwo^2u_{6\sigma}a_{6\lambda})^i\right) = 2\dtwo^{8}u_{4\lambda}u_{24\sigma}a_{20\lambda} \cdot b^{16} (\dtwo^2u_{6\sigma}a_{6\lambda})^i$, \\
$0 \leq i \leq 3$.
\end{enumerate}
\end{proposition}
\begin{proof}
To prove (1), first multiply the predicted target, $\dtwo^{8}u_{18\sigma}a_{24\lambda}a_{6\sigma}$ $(42, 54)$, by $\dthree^{16} u_{48\sigma} a_{48\lambda}$ $(96, 96)$.  By Theorem~\ref{thm:VanishingThm} and degree reasons, the product must be killed by a differential of length 61.  It follows that (1) must hold. 

By Theorem~\ref{thm:VanishingThm} and degree reasons, the target of (2) must be killed by a differential of length at most 61.  The only possible differential is the ones claimed.

To prove (3), note that in the $a_{\sigma_2}$-localized slice spectral sequence of $i_{C_2}^*\BPCfourtwotwo$, we have the differential 
\[d_{31}(\ttwo^{11}u_{24\sigma_2}a_{9\sigma_2}) =\ttwo^{16} u_{8\sigma_2}a_{40\sigma_2} \,\,\, (d_{31}(57,9) = (56,40)).\]
Applying transfer to the target shows that the image of the target under the transfer map must be killed on or before the $E_{31}$-page.  There are only two possibilities.  Either the claimed $d_{31}$-differential exists, or it is killed by a $d_{13}$-differential from $\dtwo^{7} u_{8\lambda}u_{20\sigma}a_{13\lambda}a_{\sigma}$ $(57,27)$.  By Lemma~\ref{lem:d31d13notExist}, the $d_{13}$-differential does not exist.  Therefore the claimed $d_{31}$-differential happens for $i = 0$.  The rest of the differentials in (3) and all the differentials in (4) are proven by the same method.
\end{proof}

We can propagate the differentials in Proposition~\ref{prop:BPCfourtwotwod31Differentials} with respect to the classes $\dtwo^8u_{24\sigma}a_{24\lambda}$ $(48, 48)$ and $b^{32}$ $(64, -64)$ to obtain the rest of the $d_{31}$-differentials.  

\begin{proposition}
The $d_{31}$-differentials are 
\begin{enumerate}
\item $d_{31}\left(tr(\ttwo^{11}u_{10\sigma_2}a_{23\sigma_2})\cdot (\dtwo^2u_{6\sigma}a_{6\lambda})^{i+4j}\right) = \dtwo^{8}u_{18\sigma}a_{24\lambda}a_{6\sigma}\cdot (\dtwo^2u_{6\sigma}a_{6\lambda})^{i+4j}$,\\
$i = 0, 1$, $j \geq 0$; 
\item $d_{31}\left(tr(\ttwo^3u_{24\sigma_2}a_{-15\sigma_2}) \cdot b^{16i}(\dtwo^2u_{6\sigma}a_{6\lambda})^{j} 
\right) = 2 \dtwo^4u_{4\lambda}u_{12\sigma}a_{8\lambda} \cdot b^{16i}(\dtwo^2u_{6\sigma}a_{6\lambda})^{j}$,\\
$i, j \geq 0$.
\end{enumerate}
They are shown in Figure~\ref{fig:C4E31page}.  

\end{proposition}

\begin{center}
\begin{figure}
\includegraphics[width=\textwidth, page =1]{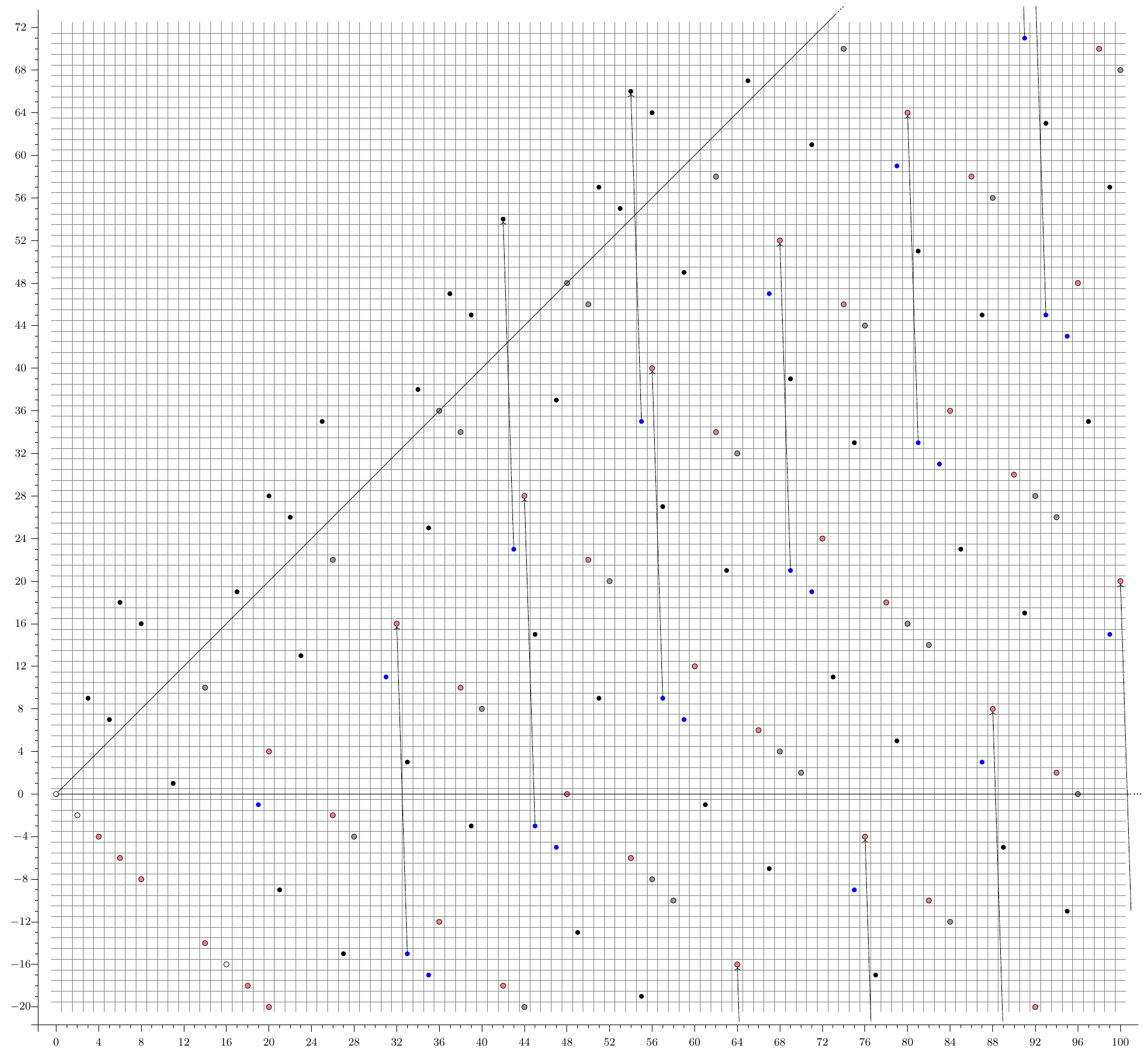}
\caption{The $E_{31}$-page of the $a_\lambda$-localized slice spectral sequence of $a_\lambda^{-1}\BPCfourtwotwo$.}
\label{fig:C4E31page}
\end{figure}
\end{center}

\subsubsection{\texorpdfstring{$d_{37}$}{d-37}-differentials}
To prove the $d_{37}$-differentials, we will first prove the nonexistence of certain $d_{13}$-differentials.

\begin{lemma}\label{lem:d37d13notExist} 
At the $E_{13}$-page, we have
\begin{enumerate}
\item $d_{13}\left(\dtwo^{12}u_{\lambda}u_{36\sigma}a_{35\lambda}\right) \neq \dtwo^{13}u_{34\sigma}a_{39\lambda}a_{5\sigma}$; 
\item $d_{13}\left(\dtwo^{12}u_{\lambda}u_{36\sigma}a_{35\lambda}\cdot (\dtwo^2u_{6\sigma}a_{6\lambda})\right) \neq \dtwo^{13}u_{34\sigma}a_{39\lambda}a_{5\sigma}\cdot (\dtwo^2u_{6\sigma}a_{6\lambda})$;
\item $d_{13}\left(\dtwo^{11}u_{3\lambda}u_{32\sigma}a_{30\lambda}a_{\sigma}\right) \neq \dtwo^{12}u_{34\sigma}a_{36\lambda}a_{2\sigma}$; 
\item $d_{13}\left(\dtwo^{11}u_{3\lambda}u_{32\sigma}a_{30\lambda}a_{\sigma}\cdot (\dtwo^2u_{6\sigma}a_{6\lambda})\right) \neq \dtwo^{12}u_{34\sigma}a_{36\lambda}a_{2\sigma}\cdot (\dtwo^2u_{6\sigma}a_{6\lambda})$;
\item $d_{13}\left(\dtwo^{9}u_{19\lambda}u_{26\sigma}a_{8\lambda}a_{\sigma}\cdot (\dtwo^2u_{6\sigma}a_{6\lambda})^i\right) \neq 2\dtwo^{10}u_{15\lambda}u_{30\sigma}a_{15\lambda}\cdot (\dtwo^2u_{6\sigma}a_{6\lambda})^i$,\\
$0 \leq i \leq 3$.

\end{enumerate}
\end{lemma}

\begin{proof}
To prove (1), note that if the class $\dtwo^{12}u_{\lambda}u_{36\sigma}a_{35\lambda}$ $(74, 70)$ supports a $d_{13}$-differential, then applying the Leibniz rule with respect to the class $\dthree^4 u_{12 \sigma} a_{12\lambda}$ $(24, 24)$ would show that the class $\dtwo^{8}u_{\lambda}u_{24\sigma}a_{23\lambda}$ $(50, 46)$ must support a differential of length at most 13.  This is a contradiction because there are no possible targets.  The nonexistence of differentials (2), (3), and (4) can be proven by the same method.  The differentials in (5) follows from (1)-(4) by applying the Leibniz rule with respect to $\dthree^4 u_{12 \sigma} a_{12\lambda}$ $(24, 24)$ and $b^{16}$ $(32, -32)$.  
\end{proof}

\begin{proposition}\label{prop:BPCfourtwotwod37Differentials1}
We have the following $d_{37}$-differentials for $i = 0, 1$: 
\[d_{37}(\dtwo^{5}u_{27\lambda}u_{14\sigma}a_{-12\lambda}a_{\sigma} \cdot (\dtwo^2u_{6\sigma}a_{6\lambda})^i) = \dtwo^{8}u_{17\lambda}u_{24\sigma}a_{7\lambda}\cdot (\dtwo^2u_{6\sigma}a_{6\lambda})^i.\]
\end{proposition}
\begin{proof}

To prove the differential when $i = 0$, we will show that the $d_{37}$-differential 
\[d_{37}(\dtwo^{13}u_{27\lambda}u_{38\sigma}a_{12\lambda}a_{\sigma})=\dtwo^{16}u_{17\lambda}u_{48\sigma}a_{31\lambda} \,\,\, (d_{37}(131,25) = (130,62))\]
exists.  Propagating with respect to the class $\dtwo^8u_{24\sigma}a_{24\lambda}$ $(48, 48)$ would then prove the desired differential.  Note that by Theorem~\ref{thm:VanishingThm}, the class $\dtwo^{16}u_{17\lambda}u_{48\sigma}a_{31\lambda}$ $(130, 62)$ must die on or before the $E_{61}$-page.  There are two possibilities: either it supports a $d_{37}$-differential hitting $\dtwo^{19}u_{8\lambda}u_{56\sigma}a_{49\lambda}a_{\sigma}$ $(129, 99)$, or the claimed differential exists.  Suppose the first case happens, then we claim there is no possibility for the class $\dtwo^{13}u_{27\lambda}u_{38\sigma}a_{12\lambda}a_{\sigma}$ $(131, 25)$ to die on or before the $E_{61}$-page.  This is because if the class does die, then the only possibility is for it to support a $d_{13}$-differential hitting $2\dtwo^{14}u_{23\lambda}u_{42\sigma}a_{19\lambda}$ $(130, 38)$.  However, if this $d_{13}$-differential exists, then by applying the Leibniz rule with respect to the class $\dthree^4 u_{12 \sigma} a_{12\lambda}$ $(24, 24)$, the class $\dtwo^{9}u_{27\lambda}u_{26\sigma}a_{\sigma}$ $(107, 1)$ must also support a differential of length at most 13.  This is a contradiction because we must have the $d_{37}$-differential
\[d_{37}(\dtwo^{9}u_{27\lambda}u_{26\sigma}a_{\sigma}) = \dtwo^{12}u_{17\lambda}u_{36\sigma}a_{19\lambda} \,\,\, (d_{37}(107,1) = (106,38))\]
by the Vanishing Theorem and degree reasons (Theorem~\ref{thm:VanishingThm}).  It follows that the class $\dtwo^{9}u_{27\lambda}u_{26\sigma}a_{\sigma}$ $(107, 1)$ supports a $d_{37}$-differential.  

The second differential, when $i=1$, is proven by the same method. 
\end{proof}

\begin{proposition}
The $d_{37}$-differentials are
\begin{enumerate}
\item $d_{37}\left(\dtwo^{4}u_{8\lambda}u_{12\sigma}a_{4\lambda} \cdot (\dtwo^2u_{6\sigma}a_{6\lambda})^{i+4j} \right) =\dtwo^{7}u_{18\sigma}a_{21\lambda}a_{3\sigma} \cdot (\dtwo^2u_{6\sigma}a_{6\lambda})^{i+4j}$,\\ $i = 0, 1$, $j \geq 0$;
\item $d_{37}\left(\dtwo u_{8\lambda}u_{2\sigma}a_{-5\lambda}a_{\sigma}\cdot (\dtwo^2u_{6\sigma}a_{6\lambda})^{i + 4j} \right) = \dtwo^{4}u_{8\sigma}a_{12\lambda}a_{4\sigma} \cdot (\dtwo^2u_{6\sigma}a_{6\lambda})^{i + 4j}$, \\
$i = 0, 3$, $j \geq 0$;
\item $d_{37}\left(2\dtwo^{2}u_{7\lambda}u_{6\sigma}a_{-\lambda} \cdot (\dtwo^2u_{6\sigma}a_{6\lambda})^{i+4j} \right) =\dtwo^{5}u_{10\sigma}a_{15\lambda}a_{5\sigma} \cdot (\dtwo^2u_{6\sigma}a_{6\lambda})^{i+4j}$, \\
$i = 0, 1$, $j \geq 0$;
\item $d_{37}\left(2b^{12+16i} \cdot (\dtwo^2u_{6\sigma}a_{6\lambda})^{j} \right) =\dtwo^{3}u_{3\lambda}u_{8\sigma}a_{6\lambda}a_{\sigma} \cdot b^{16i}(\dtwo^2u_{6\sigma}a_{6\lambda})^{j}$,\\
$i, j \geq 0$;
\item $d_{37}\left(\dtwo u_{11\lambda}u_{2\sigma}a_{-8\lambda}a_{\sigma} \cdot b^{16i}(\dtwo^2u_{6\sigma}a_{6\lambda})^{j} \right) = \dtwo^{4}u_{\lambda}u_{12\sigma}a_{11\lambda}\cdot b^{16i}(\dtwo^2u_{6\sigma}a_{6\lambda})^{j}$,\\
$i, j \geq 0$.

\end{enumerate}
They are shown in Figure~\ref{fig:C4E37page}.  
\end{proposition}
\begin{proof}
All the differentials can be proven immediately from the Vanishing Theorem and degree reasons (Theorem~\ref{thm:VanishingThm} and Lemma~\ref{lem:d37d13notExist}), Proposition~\ref{prop:BPCfourtwotwod37Differentials1}, and propagation with respect to the classes $\dtwo^8u_{24\sigma}a_{24\lambda}$ $(48, 48)$ and $b^{32}$ $(64, -64)$. 
\end{proof}

\begin{center}
\begin{figure}
\includegraphics[width=\textwidth]{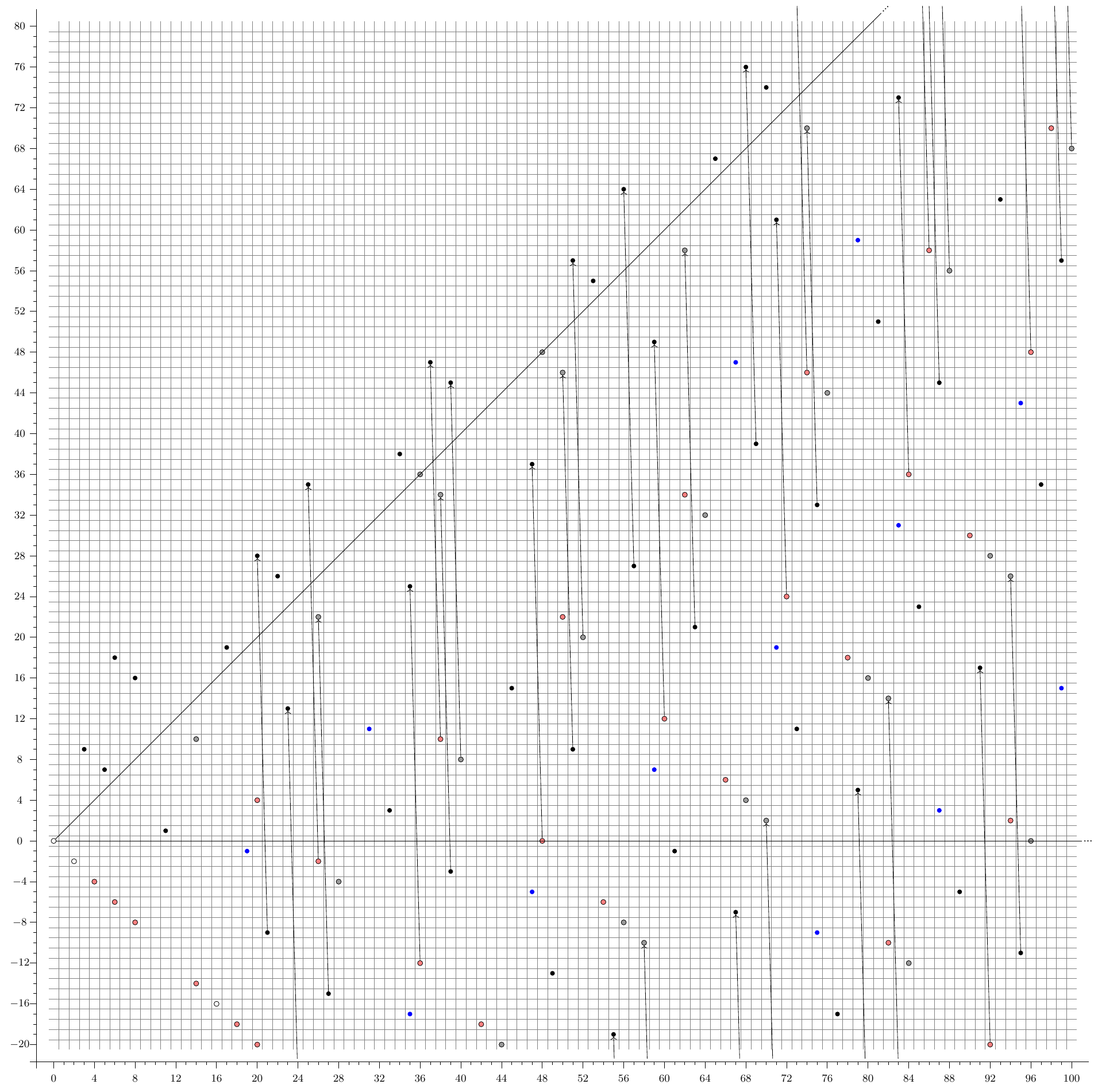}
\caption{The $E_{37}$-page of the $a_\lambda$-localized slice spectral sequence of $a_\lambda^{-1}\BPCfourtwotwo$.}
\label{fig:C4E37page}
\end{figure}
\end{center}

\subsubsection{\texorpdfstring{$d_{43}$}{d-43}-differentials}
\begin{proposition}\label{prop:BPCfourtwotwod43Differentials1}
The following $d_{43}$-differentials exist for $i = 0, 1$: 
\[d_{43}(\dtwo^{12}u_{40\lambda}u_{36\sigma}a_{-4\lambda} \cdot (\dtwo^2 u_{6\sigma}a_{6\lambda})^i) = tr(\ttwo^{31}u_{58\sigma_2}a_{35\sigma_2})\cdot (\dtwo^2 u_{6\sigma}a_{6\lambda})^i.\]
\end{proposition}
\begin{proof}
When $i = 0$, note that by the Vanishing Theorem (Theorem~\ref{thm:VanishingThm}), the class $tr(\ttwo^{31}u_{58\sigma_2}a_{35\sigma_2})$ $(151, 35)$ must die on or before the $E_{61}$-page.  There are two possibilities.  Either the claimed differential occurs, or it supports a $d_{55}$-differential hitting $2\dtwo^{20}u_{15\lambda}u_{60\sigma}a_{45\lambda}$ $(150, 90)$.  The second case does not occur because the class $2\dtwo^{20}u_{15\lambda}u_{60\sigma}a_{45\lambda}$ $(150, 90)$ needs to support a $d_{61}$-differential killing the class $\dtwo^{25}u_{74\sigma}a_{75\lambda}a_{\sigma}$ $(149, 151)$, or else no class would be able to kill $\dtwo^{25}u_{74\sigma}a_{75\lambda}a_{\sigma}$ $(149, 151)$ on or before the $E_{61}$-page and we would reach a contradiction with Theorem~\ref{thm:VanishingThm}.  

The second differential is proven by the same method.  
\end{proof}

\begin{proposition}
The $d_{43}$-differentials are 
\begin{enumerate}
 \item $d_{43}\left(b^{16+32i} \cdot (\dtwo^2u_{6\sigma}a_{6\lambda})^{j+4k}\right) =tr(\ttwo^7u_{10\sigma_2}a_{11\sigma_2}) \cdot b^{32i}(\dtwo^2u_{6\sigma}a_{6\lambda})^{j+4k}$, \\
 $i, k \geq 0$, $j = 0, 3$;
 \item $d_{43}\left(b^{24+32i} \cdot (\dtwo^2u_{6\sigma}a_{6\lambda})^{j+4k}(\dtwo^2u_{8\lambda}u_{6\sigma}a_{-2\lambda})^\ell\right) \\= tr(\ttwo^7u_{26\sigma_2}a_{-5\sigma_2})\cdot b^{32i}(\dtwo^2u_{6\sigma}a_{6\lambda})^{j+4k}(\dtwo^2u_{8\lambda}u_{6\sigma}a_{-2\lambda})^\ell$, \\
$i, k \geq 0$, $j = 0, 1$, $\ell = 0, 1, 2$.
\end{enumerate}
They are shown in Figure~\ref{fig:C4E43-61page}.  
\end{proposition}
\begin{proof}
All the differentials can be proven immediately from the Vanishing Theorem and degree reasons (Theorem~\ref{thm:VanishingThm}), Proposition~\ref{prop:BPCfourtwotwod43Differentials1}, and propagation with respect to the classes $\dtwo^8u_{24\sigma}a_{24\lambda}$ $(48, 48)$ and $b^{32}$ $(64, -64)$. 
\end{proof}

\subsubsection{\texorpdfstring{$d_{55}$}{d-55}-differentials}
\begin{proposition}\label{prop:BPCfourtwotwod55Differentials}
The $d_{55}$-differentials are 
\begin{enumerate}
\item $d_{55}\left(tr(\ttwo^{3}u_{26\sigma_2}a_{-17\sigma_2}) \cdot b^{32i}(\dtwo^2u_{6\sigma}a_{6\lambda})^{j+4k}\right) = \dtwo^{6}u_{16\sigma}a_{18\lambda}a_{2\sigma} \cdot b^{32i}(\dtwo^2u_{6\sigma}a_{6\lambda})^{j+4k}$, \\
$i, k \geq 0$, $j = 0, 3$;
\item $d_{55}\left(tr(\ttwo^{3}u_{26\sigma_2}a_{-17\sigma_2}) \cdot b^{8+32i}(\dtwo^2u_{6\sigma}a_{6\lambda})^{j+4k}(\dtwo^2u_{8\lambda}u_{6\sigma}a_{-2\lambda})^\ell\right) \\= \dtwo^{6}u_{16\sigma}a_{18\lambda}a_{2\sigma}\cdot b^{8+32i}(\dtwo^2u_{6\sigma}a_{6\lambda})^{j+4k}(\dtwo^2u_{8\lambda}u_{6\sigma}a_{-2\lambda})^\ell$, \\
$i, k \geq 0$, $j = 0, 1$, $\ell = 0, 1, 2$.
\end{enumerate}
They are shown in Figure~\ref{fig:C4E43-61page}.  
\end{proposition}
\begin{proof}
All the differentials can be deduced from the Vanishing Theorem and degree reasons (Theorem~\ref{thm:VanishingThm}), and propagation with respect to the classes $\dtwo^8u_{24\sigma}a_{24\lambda}$ $(48, 48)$ and $b^{32}$ $(64, -64)$. 
\end{proof}

\subsubsection{\texorpdfstring{$d_{61}$}{d-61}-differentials}
\begin{proposition}\label{prop:BPCfourtwotwod61Differentials}
We have the following $d_{61}$-differentials: 
\begin{enumerate}
\item $d_{61}\left(\dtwo u_{16\lambda}u_{2\sigma}a_{-13\lambda}a_{\sigma} \cdot (\dtwo^2u_{8\lambda}u_{6\sigma}a_{-2\lambda})^i (\dtwo^2u_{6\sigma}a_{6\lambda})^{j+4k}b^{32\ell}\right) \\
= \dtwo^{6}u_{18\sigma}a_{18\lambda} \cdot (\dtwo^2u_{8\lambda}u_{6\sigma}a_{-2\lambda})^i (\dtwo^2u_{6\sigma}a_{6\lambda})^{j+4k}b^{32\ell}$, \\
$(i, j) = (0,0), (0,1), (1,0), (1,1), (2,0), (2,1), (3,-3), (3,0)$, $k, \ell \geq 0$;
\item $d_{61}\left(2\dtwo^{4}u_{15\lambda}u_{12\sigma}a_{-3\lambda} \cdot (\dtwo^2u_{8\lambda}u_{6\sigma}a_{-2\lambda})^i (\dtwo^2u_{6\sigma}a_{6\lambda})^{j+4k}b^{32\ell}\right) \\
= \dtwo^{9}u_{26\sigma}a_{27\lambda}a_{\sigma}\cdot (\dtwo^2u_{8\lambda}u_{6\sigma}a_{-2\lambda})^i (\dtwo^2u_{6\sigma}a_{6\lambda})^{j+4k}b^{32\ell}$, \\
$(i, j) = (0,0), (0, 1), (1,-3), (1,0), (2, -4), (2, -3), (3, -4), (3, -3)$, $k, \ell \geq 0$.
\end{enumerate}
They are shown in Figure~\ref{fig:C4E43-61page}. 
\end{proposition}
\begin{proof}
All the differentials can be deduced from the Vanishing Theorem and degree reasons (Theorem~\ref{thm:VanishingThm}), and propagation with respect to the classes $\dtwo^8u_{24\sigma}a_{24\lambda}$ $(48, 48)$ and $b^{32}$ $(64, -64)$. 
\end{proof}

\begin{center}
\begin{figure}
\includegraphics[width=\textwidth, page =1]{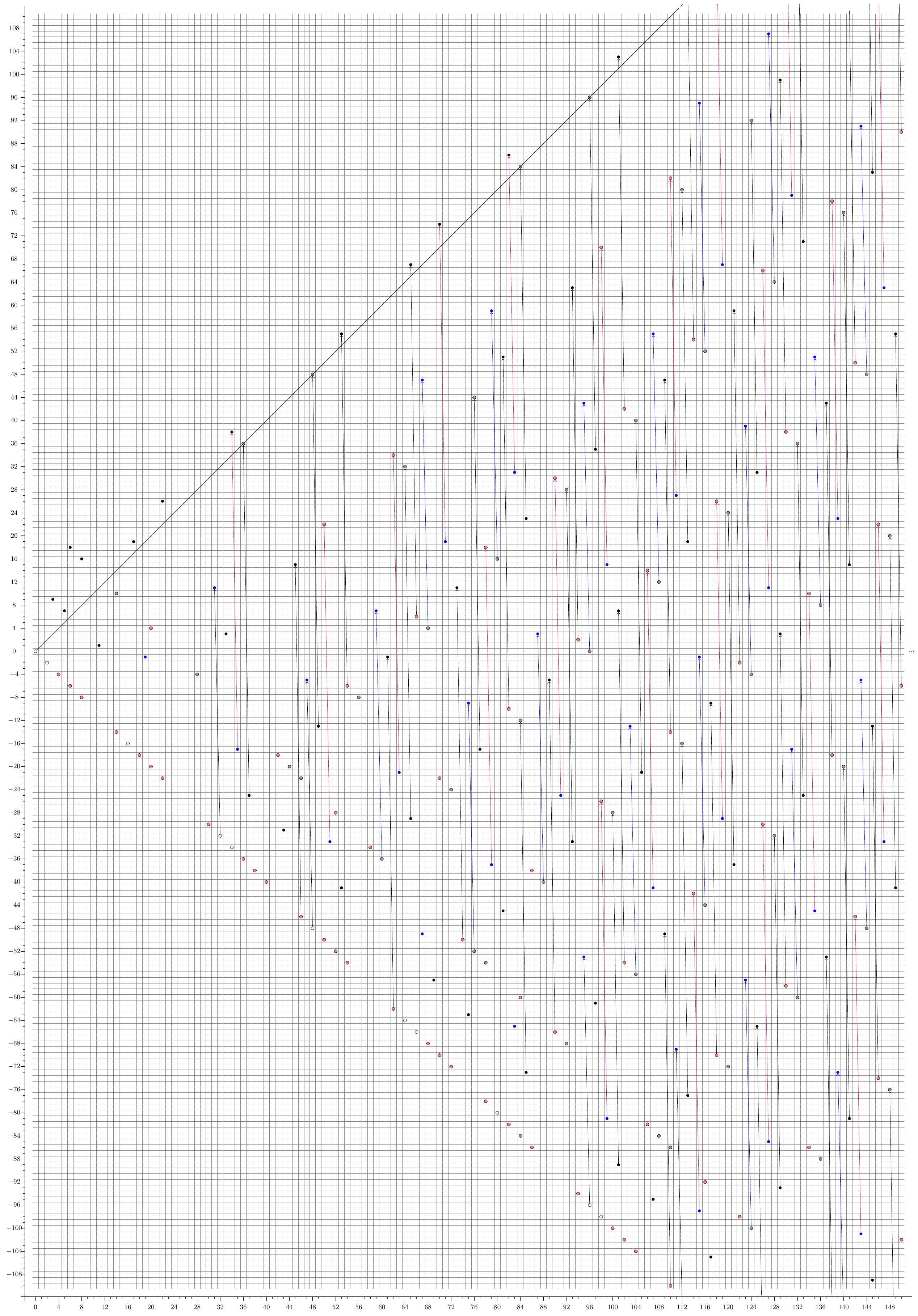}
\caption{The $d_{43}$ ({\color{blue} blue}), $d_{55}$ ({\color{magenta} magenta}), and $d_{61}$-differentials (black) in the $a_\lambda$-localized slice spectral sequence of $a_\lambda^{-1}\BPCfourtwotwo$.}
\label{fig:C4E43-61page}
\end{figure}
\end{center}

\newpage
\begin{center}
\begin{figure}
\includegraphics[width=\textwidth, page =2]{BPC422SliceSSE435561}
\caption{The $E_\infty$-page of the $a_\lambda$-localized slice spectral sequence of $a_\lambda^{-1}\BPCfourtwotwo$.}
\label{fig:C4Einftypage}
\end{figure}
\end{center}

\bibliographystyle{plain}
\bibliography{KTheory}

\end{document}